\documentclass[11pt,a4paper]{article}
\usepackage{makeidx}
\usepackage{amsmath,amsthm}

\theoremstyle{plain}
\providecommand*{\theoremincountername}{section}
\newtheorem{theorem}{Theorem}[\theoremincountername]
\newtheorem{lemma}[theorem]{Lemma}
\newtheorem{corollary}[theorem]{Corollary}
\theoremstyle{definition}
\newtheorem{definition}[theorem]{Definition}
\newtheorem{example}[theorem]{Example}
\theoremstyle{remark}
\newtheorem*{remark}{Remark}

\numberwithin{equation}{section}

\makeatletter
\DeclareRobustCommand\SMC{%
  \ifx\@currsize\normalsize\small\else
   \ifx\@currsize\small\footnotesize\else
    \ifx\@currsize\footnotesize\scriptsize\else
     \ifx\@currsize\large\normalsize\else
      \ifx\@currsize\Large\large\else
       \ifx\@currsize\LARGE\Large\else
        \ifx\@currsize\scriptsize\tiny\else
         \ifx\@currsize\tiny\tiny\else
          \ifx\@currsize\huge\LARGE\else
           \ifx\@currsize\Huge\huge\else
            \small\SMC@unknown@warning
 \fi\fi\fi\fi\fi\fi\fi\fi\fi\fi
}
\newcommand{\SMC@unknown@warning}{\TBWarning{\string\SMC: unrecognised
    text font size command -- using \string\small}}
\makeatother
\DeclareTextFontCommand{\smalltexttt}{\SMC\ttfamily}

\usepackage{amssymb}

\usepackage{graphicx}
\providecommand{\includegraphics}[2][]{%
   \vrule\ %
   \parbox{0.5\columnwidth}{\texttt{#2} will be inserted here; too bad 
   you can't see it now.}%
   \ \vrule
}

\makeindex
\makeatletter
\newcommand{\DefOrd}[2][**]{\DoIndex#1{#2}\textbf{#2}}
\newcommand{\emDefOrd}[2][**]{\DoIndex#1{#2}\emph{#2}}
\newcommand{\DoIndex}{%
   \@ifstar{\@ifstar\index\indexAndGobble}\DoIndexX
}    
\newcommand{\indexAndGobble}[2]{\index{#1}}
\newcommand{\DoIndexX}[1]{%
   \def\search@key{#1}%
   \@ifstar{\joinIndexArgs{}}{\joinIndexArgs\@gobble}%
}
\newcommand{\joinIndexArgs}[2]{\index{\search@key @#2}#1}
\let\gobblepage=\@firstoftwo
\makeatother

\theoremstyle{plain}
\newtheorem*{proposition-nonum}{Proposition}
\newtheorem{assumption}{Assumption}
\theoremstyle{definition}
\newtheorem{construction}[theorem]{Construction}

\theoremstyle{remark}
\newtheorem*{exercise}{Exercise}

\providecommand*{\DefOrd}[2][]{\textbf{#2}}
\providecommand*{\emDefOrd}[2][]{\emph{#2}}
\newcommand*{\parenthetic}[1]{\/\textup{(#1)}}
\providecommand*{\textprime}{\('\)}
\providecommand*{\textbis}{\textprime\textprime}
\providecommand*{\texttris}{\textprime\textprime\textprime}
\providecommand*{\ISBN}{ISBN}
\providecommand*{\Dash}{%
   \hspace*{0.166667em}\textemdash\hspace{0.166667em}%
}
\providecommand*{\Ldash}{%
   \hspace{0.166667em}\textemdash\hspace*{0.166667em}%
}

\newcommand{\PROP}{{\SMC PROP}}
\newcommand{\PROPs}{{\SMC PROP}s}

\newcommand*{\mc}[1]{\mathcal{#1}}

\newcommand*{\pin}[1]{
   \mathchoice{%
      \mathrel{\mathrm{in}}%
   }{%
      \mathrel{\mathrm{in}}%
   }{%
      \mathop{\mathrm{in}}%
   }{%
      \mathop{\mathrm{in}}%
   }#1%
}

\makeatletter
\newcommand*{\tS}{t^S\@ifnextchar({\mkern -2mu}{}}
\makeatother

\DeclareMathOperator{\Div}{div}

\makeatletter
\newcommand*{\Norm}[2][\@gobble]{\left\|#1. #2 \right\|}
\newcommand*{\norm}[2][\@gobble]{\left|#1. #2 \right|}
\newcommand*{\setOf}[3][\@gobble]{%
   \left\{ \, #2 \,\,\vrule\relax#1.\,\, #3 \, \right\}%
}
\makeatother

\newcommand{\mbin}[1]{\mathbin{#1}\nobreak}

\newcommand{\Epil}{\quad\Longleftrightarrow\quad}
\newcommand{\Fpil}{\longrightarrow}
\newcommand{\Ipil}{\quad\Longrightarrow\quad}
\newcommand{\Lpil}{\rightarrow}

\newcommand{\id}{\mathrm{id}}
\newcommand{\DIS}{\mathrm{DIS}}

\newcommand{\DSM}{\mathrm{DSM}}
\newcommand{\Irr}{\mathrm{Irr}}
\newcommand{\Per}{\mathrm{Per}}
\newcommand{\Red}{\mathrm{Red}}
\newcommand{\Span}{\mathrm{Span}}
\newcommand{\Cspan}{\mathrm{Cspan}}

\newcommand{\supp}{\mathrm{supp}}
\newcommand{\LM}{\mathrm{LM}}
\newcommand{\lm}{\mathrm{lm}}

\newcommand{\cmplM}{\cmpl{\mc{M}}}
\newcommand{\cmplO}{\widehat{\mc{O}}}

\newcommand{\N}{\mathbb{N}}
\newcommand{\Q}{\mathbb{Q}}
\newcommand{\R}{\mathbb{R}}

\newcommand{\Rp}{\R^+}

\newcommand{\Z}{\mathbb{Z}}
\newcommand{\Zp}{\Z^+}

\newcommand{\ve}{\varepsilon}

\newcommand{\ssa}{{\mathsf{a}}}
\newcommand{\ssb}{{\mathsf{b}}}
\newcommand{\ssc}{{\mathsf{c}}}

\newcommand{\ssx}{{\mathsf{x}}}

\newcommand{\ssI}{\mathsf{1}}

\newcommand*{\FAlg}[2]{#2 \langle #1 \rangle }
\newcommand{\RavX}{\FAlg{X}{\mc{R}}}

\newcommand{\RstarY}{R^*\mkern-2mu\mc{Y}}

\makeatletter
\newcommand{\cmpl}[1]{%
   \sbox\z@{$#1$}%
   \dimen@=\wd\z@
   \advance \dimen@ -\strip@pt\fontdimen\@ne\textfont\@ne \ht\z@
   \setbox\tw@=\hb@xt@\dimen@{}%
   \ht\tw@=\ht\z@ \dp\tw@=\dp\z@ 
   \box\z@
   \llap{$\overline{\box\tw@}$}%
}
\makeatother

\advance \textfloatsep 0pt plus 1.5\baselineskip

\usepackage[%
   pdfborder={0 0 0},%
   pdfstartview={XYZ null null null}%
]{hyperref}

\begin{document}

\title{A Generic Framework for Diamond Lemmas} 
\author{Lars Hellstr\"om%
  \thanks{E-mail: \texttt{Lars.Hellstrom@residenset.net}. 
    Postal address: Lars Hellstr\"om, Sand 216, 
    S-881\,91~Sollefte\r{a}, Sweden.}%
}
\date{}
\maketitle

\begin{abstract}
  This paper gives a generic form of the diamond lemma, which includes 
  support for additive and topological structures of the base set, 
  and which does not require any further structure (e.g.~an 
  associative multiplication operation) to be present. 
  This result is intended to be used as the core of diamond 
  lemmas for particular algebraic structures, taking care of 
  all the common technicalities. 
  With this generic diamond lemma, the main steps needed to prove a 
  specialised diamond lemma is to define the reduction maps and 
  analyse the structure of critical ambiguities.
  
  The abstract machinery is backed up with concrete suggestions for 
  how one should set things up in order to reproduce traditional 
  results in the general setting. 
  Several instances of the fundamental theorem of Gr\"obner basis 
  theory are derived as corollaries of the main result.
\end{abstract}

\section{Introduction}

The \emph{Diamond Lemma for Ring Theory}~\cite{Bergman} of Bergman is 
an important theorem that links together several branches of 
mathematics. On one hand it is the bridge between associative algebra 
and mathematical logic that can make the definition of an algebra 
through generators and relations effective. On another it marks a middle 
ground between the theory of Gr\"obner bases and the theory of term 
rewriting, which can be seen as belonging to either of the two. 
Yet it is only one member in a family of results on similar connections, 
which can be quite 
different in their technical details even though the essential ideas 
are mostly the same. Furthermore many of these results exist in the 
literature only as sketches (which, it seems, everybody is waiting for 
someone else to flesh out, as it is all so ``obvious'' anyway), and 
as a result the rigor in many arguments becomes somewhat lacking, as 
they should rightly have been \emph{proofs} making use of some 
particular diamond lemma.

My intention here is to state and prove a generic form of the diamond 
lemma from which one can easily derive more specialised results 
suitable for particular problems. It is probably not the most generic 
form that is possible, but it can deal with the technicalities in all 
cases I know of, and does so without making extensive assumptions 
about the structure to which it is applied. 

From a strictly technical perspective, 
the theorem given here generalises that of Bergman in three directions:
\begin{enumerate}
  \item
    A topological aspect is added to the basic machinery. This makes 
    it possible to treat e.g.~formal power series problems within the 
    diamond lemma framework.
  \item
    The assumptions about a multiplicative structure have been 
    dropped from the core theorem. Auxiliary theorems are provided 
    which together with the core theorem cover what the associative 
    algebra diamond lemma can do, but also apply for a much broader 
    range of algebraic structures (nonassociative algebras, operads, 
    \PROPs~\cite{MacLane65}, etc.). 
  \item
    The definition of reductions has been separated from the diamond 
    lemma, so that it no longer depends on these having a particular 
    form or that all reductions of a particular form are active. 
    The latter is useful if one wishes to cover Shirshov's theory of 
    bases for Lie algebras~\cite{Shirshov}.
\end{enumerate}
The first generalisation was the subject of my Ph.D. 
thesis~\cite{Avhandlingen}, but the presentation here has been 
refined in that it eliminates many minor assumptions on how the 
multiplicative and topological structures interact. Readers who want 
concrete examples may however prefer the thesis presentation, as it 
treats some applications in great detail.

The main advantage of the topological aspect is that it enables one to 
handle both polynomials and power series (or their respective 
counterparts from less traditional algebraic structures) using the same 
machinery. 
A less apparent advantage is that problems that can be posed entirely 
in terms of finite sums (i.e., polynomials) sometimes have solutions 
where the normal form is an infinite sum (power series), and in this 
case one has to employ the topologized version in order to prove 
things about this normal form. In so doing, one can take advantage of 
certain relaxations of the conditions of the classical result; 
Definition~\ref{Def:TDCC} of the descending chain condition and 
Definition~\ref{Def:Tvetydighet} of ambiguity resolution both admit 
more than in Bergman's diamond lemma.

The second generalisation has been the main direction in my 
subsequent work, initiated in response to a question from Loday on 
whether there is a diamond lemma for operads. It's not too hard to 
see that there is such a creature\Ldash more work had to be spent 
sorting out the details of claims than the details of their 
proofs\Dash but one fundamental change when going from algebras to 
operads is that one goes from a single-sorted algebraic structure 
(there is \emph{one} set of elements) to a multiple-sorted algebraic 
structure (in an operad, elements of different arities don't mix, 
and hence there is a separate sort of element for each arity). 
The interactions 
between these elements of different sorts is certainly a kind of 
multiplicative structure, but one that syntactically is much more 
unwieldy (regardless of whether one prefers to phrase it using the 
structure map formalism or the $i$th composition formalism) than that 
of a ring, and bundling these interactions with the diamond lemma would 
turn an already very technical result into something even worse. 
Furthermore the generalisations do not stop at operads. There are good 
reasons to at least go on to \PROPs\ (because each operad is a part of 
some \PROP, and \PROPs\ have a more concise set of axioms), and after 
that there are more general diagrammatic structures that one may 
wish to consider. Handling them all in one result does not seem a 
likely achievement.

What turns out to work is instead to separate the parts of the 
classical diamond lemma that deal with the multiplicative structure 
from the parts that ignore this structure. The core of the diamond lemma 
(Theorem~\ref{S:CDL}, with the familiar equivalence of four different 
conditions) can be very neatly captured as a result on one sort 
(hence ignoring the multiplicative structure under which sorts may 
interact), whereas the construction of reductions and verifications 
that some ambiguities are trivially resolvable fall into the other part. 
This is actually rather fortunate, 
because the first part will then deal with the classical technicalities, 
whereas the second will deal with the particular features of rings, 
operads, \PROPs, or whatever; there is an almost complete separation of 
responsibilities.

The third generalisation is thus in part a natural consequence of the 
second, but there are also other advantages to it. One is that many 
defining identities of the classical nonassociative algebras do not fit 
well to make ``unconditional'' rules from; a simple example is the 
anticommutativity identity of a Lie algebra, which if expressed as a 
rule `\([x,y] \rightarrow -[y,x]\) for all $x$ and $y$' would lead to 
the infinite rewrite cycle \([x,y] \rightarrow -[y,x] \rightarrow 
-\bigl(-[x,y]\bigr) = [x,y] \rightarrow -[y,x] \rightarrow \dotsb\). 
One way to handle that in practice is to instead make a conditional 
rule `\([x,y] \rightarrow -[y,x]\) if \(x>y\)' out of it, and the 
machinery constructed here can handle that; since each pair $(x,y)$ 
of factors gives rise to a separate reduction map, it is merely a 
matter of considering only those pairs for which \(x>y\) in some 
suitable ordering of the factors. While there is a certain 
price to pay in that ambiguity resolution becomes less automatic, 
this price effectively only comes into play when the conditional 
rules are involved.

The structure of this paper is as follows.
Sections~\ref{Sec:Basics} and~\ref{Sec:Reduktioner} introduce the 
framework within which the core diamond lemma (Theorem~\ref{S:CDL}) 
is formulated. These sections also contain plenty of minor 
constructions for setting up various aspects of this framework, to 
illustrate features of the formalism used, and to aid the reader in 
applying the results. 

Sections~\ref{Sec:Reducibilitet} and~\ref{Sec:Monomordning} contain 
the bulk of the proof of the diamond lemma. The former section is 
about more abstract algebraic--topological properties of monoids of 
maps, whereas the latter introduces an order and uses induction 
to link these properties to conditions that can be verified through 
explicit calculations. Notable lemmas are \ref{L:(b)<=>(c)}~(linking 
normal form uniqueness to univocality of the pointwise limit of 
reductions), \ref{L: Persistently red.}~(existence of normal 
forms), and \ref{L:Red=cmplM}~(uniqueness of normal form given 
relative resolvability of ambiguities). Besides the main 
Theorem~\ref{S:CDL}, there is also Theorem~\ref{S:Konstr.Irr(S)} 
which provides a characterisation of irreducible elements.

Section~\ref{Sec:Tvetydigheter} is about ambiguities 
(a.k.a.~critical pairs or overlaps) and how one in a multisorted 
situation can discard non-critical ambiguities from consideration. 
This is as much about defining `critical ambiguity'\Ldash a subject 
which spans definitions~\ref{Def:Framflyttbar2}, \ref{Def:Montage}, 
and~\ref{Def:V-kritisk}\Dash as it is about proving them discardable. 
The claim that checking the critical ambiguities is as good as checking 
all ambiguities can be found in Theorem~\ref{S:V-kritisk}. 
Example~\ref{Ex:BergmanskTvetydighet} derives Bergman's diamond lemma 
from the generic theory. 
Theorem~\ref{S:DropRule} is aimed more at completion calculations; 
it justifies dropping unnecessary rules while in the middle of 
completing a rewriting system.

Section~\ref{Sec:Konstruktion} collects a construction and some 
technical lemmas that may be used in applications to demonstrate 
that the most common setting (a collection of free modules) leads 
to a framework suitable for the generic diamond lemma. 
Again the aim is to bridge the gap between concrete conditions that 
are easy to verify and more abstract conditions used in the generic 
theory.

The final Section~\ref{Sec:Grobner} is about Gr\"obner bases, where 
Theorem~\ref{S:Grobner-DL} extends the big equivalence in the generic 
diamond lemma with some GB-style claims. Several instances of ``the 
fundamental theorem on Gr\"obner bases'' (in commutative, associative, 
and nonassociative polynomial algebras) are derived as corollaries of 
this theorem, and the theory is shown to also cover the case of path 
algebras.

A more practical application of the generic diamond lemma theory can 
be found in~\cite{RwGraphInvariant}. Unlike the applications in 
Section~\ref{Sec:Grobner}, this exercises the multisorted aspects of 
the framework.

\subsection*{Notation}

The set $\N$ of natural numbers is considered to include $0$. $\Zp$ 
is the set of positive integers and $\Rp$ is the set (sometimes the 
multiplicative group) of positive real numbers. 
The shorthand $f(A)$ for $\setOf[\big]{ f(a) }{ a \in A }$ is 
frequently applied.

Formal variables are typically written using a sans-serif font: 
$\ssa$, $\ssb$, $\ssc$, etc. When $X$ is a set of such letters, 
$X^*$\index{X*@$X^*$} denotes the free monoid on $X$, i.e., the 
set of all finite strings of elements from $X$. The identity 
element in $X^*$ is denoted $\ssI$\index{1@$\ssI$}.

On the matter of monomials versus terms, a \emDefOrd{monomial} is 
considered to not include a coefficient, whereas a \emDefOrd{term} 
generally contains a coefficient. The relation symbol $\equiv$ 
denotes congruence rather than identity.

\section{Basics}
\label{Sec:Basics}

For the machinery employed here, it is convenient to fix a framework 
with five pieces of data:
\begin{itemize}
  \item
    An abelian group $\mc{M}$\index{M@$\mc{M}$} (written additively). 
    This will play the role of set of all finite expressions.
  \item 
    A set $R$\index{R@$R$} of maps \(\mc{M} \Fpil \mc{M}\). This can 
    be used to encode a module structure on $\mc{M}$.
  \item
    A subset $\mc{Y}$\index{Y@$\mc{Y}$} of $\mc{M}$. This will play 
    the role of set of monomials.
  \item 
    A family \index{O@$\mc{O}$}\(\mc{O} = 
    \{B_n\}_{n=1}^\infty\)\index{B n@$B_n$} of subsets of 
    $\mc{M}$. This will become the fundamental system of 
    neighbourhoods of \(0 \in \mc{M}\) and is thus defining the 
    topology.
  \item
    A family $T_1(S)$\index{T 1 S@$T_1(S)$} of maps 
    \(\cmplM \Fpil \cmplM\), where \(\cmplM \supseteq \mc{M}\) is the 
    set of all expressions. These are what in the end specify the 
    wanted congruence on $\mc{M}$.
\end{itemize}
When applying the diamond lemma to a multiple-sorted structure, there 
will be one such quintuplet $\bigl( \mc{M}, R, \mc{Y}, \mc{O}, T_1(S) 
\bigr)$ for each sort, but since the core diamond lemma itself is 
applied separately for each sort, one does not have to take this 
multiplicity into account when proving it. Notation for and 
interactions between different framework quintuplets for a structure 
are considered in Section~\ref{Sec:Tvetydigheter}.

In the main theorem there will also be:
\begin{itemize}
  \item
    a partial order $P$ on $\mc{Y}$;
\end{itemize}
but that can without too much difficulty be separated from the rest 
of the machinery, so it will instead be introduced explicitly 
whenever it is needed. Having it separate is sometimes convenient, as 
one in complicated arguments might want to make use of several 
different orders. Finally, there is in several supporting results:
\begin{itemize}
  \item
    a family $V$ of maps \(\cmplM \Fpil \cmplM\), which can be used 
    to enforce compatibility with a multiplicative structure;
\end{itemize}
but the typical use of that item is rather on the level of 
constructing $T_1(S)$ or proving things about it. 

The choices of $R$, $\mc{Y}$, $\mc{O}$, and $T_1(S)$ are subject to a 
couple of additional conditions, which are specified as 
\emph{assumptions} below. Technically it would be possible to instead 
include them as additional conditions in all theorems and lemmas that 
depend on them, but it is more convenient to throughout the presentation 
assume them to be satisfied. There are plenty of suggestions for how 
one may choose the framework data to ensure that the assumptions are 
met. $T_1(S)$ is treated in the next section, but assumptions on $R$, 
$\mc{Y}$, and $\mc{O}$ are given here.

\begin{assumption}
  Every element of $R$ is a group endomorphism of $\mc{M}$.
\end{assumption}

\begin{definition}
  A subgroup \(N \subseteq \mc{M}\) is said to be an 
  \DefOrd[{module}*]{$R$-module} if \(r(a) \in N\) for all 
  \(a \in N\) and \(r \in R\).
\end{definition}

If $\mc{M}$ has an $\mc{R}$-module structure for some ring $\mc{R}$, 
then it is natural to choose as $R$ the set of maps \(a \mapsto ra : 
\mc{M} \Fpil \mc{M}\) for all \(r \in \mc{R}\); in this case the above 
$R$-module concept coincides with the standard $\mc{R}$-module concept 
and all is as one expects it to be. It may however in some cases be 
necessary to impose a restriction on the elements of $\mc{R}$ which may 
contribute to $R$, and in that case it is weaker to be an $R$-module 
than to be an $\mc{R}$-module. It is also perfectly possible to take 
\(R = \varnothing\) if no particular module structure is available.

\begin{assumption} \label{Ant:YSpan}
  If \(N \subseteq \mc{M}\) is an $R$-module such that \(\mc{Y} 
  \subseteq N\) then \(N = \mc{M}\).
\end{assumption}

In other words, $\mc{Y}$ spans $\mc{M}$. 
The traditional approach is to begin with $\mc{Y}$ just being some 
set, pick some ring $\mc{R}$, and then \emph{construct} $\mc{M}$ as 
the free $\mc{R}$-module with basis $\mc{Y}$; if in particular 
$\mc{Y}$ is the set $X^*$ of words on the alphabet $X$ then this will 
make $\mc{M}$ equal to the free $\mc{R}$-algebra $\RavX$. 
An alternative approach for the free algebra $\RavX$ is however to let 
$\mc{Y}$ be the set of all terms\Ldash products $r\mu$ of a scalar \(r 
\in \mc{R}\) and a monomial \(\mu \in X^*\)\Dash as this makes 
it possible to take \(R=\varnothing\). It is also possible to 
interpolate between these two extremes, or pick a set $\mc{Y}$ with 
more complicated linear dependencies between elements, although the 
latter is likely to make it more complicated to construct $T_1(S)$.

\begin{definition}
  Let $R^*$\index{R*@$R^*$} denote the set of all finite compositions 
  of elements of $R$; in particular, $R^*$ is considered to contain 
  the identity map \(\id\colon \mc{M} \Fpil \mc{M}\). 
  Let \index{+-R*@$\pm R^*$}\(\pm R^*\) denote the set 
  $\setOf{r, -r}{r \in R^*}$. Let \index{R*Y@$\RstarY$}$\RstarY$ 
  denote the set \(\setOf[\big]{ r(\mu) }{ r \in R^*, \mu \in \mc{Y} } 
  \subseteq \mc{M}\).
\end{definition}

\begin{lemma} \label{L:M-form}
  Every element \(a \in \mc{M}\) can be expressed as
  \begin{equation} \label{Eq:M-form}
    a = \sum_{k=1}^n r_k(\mu_k)
  \end{equation}
  for some \(n\in\N\), \(r_1,\dotsc,r_n \in \pm R^*\), and 
  \(\mu_1,\dotsc,\mu_n \in \mc{Y}\).
\end{lemma}
\begin{proof}
  The set of elements on the form \eqref{Eq:M-form} constitutes an 
  $R$-module that contains $\mc{Y}$. Hence by 
  Assumption~\ref{Ant:YSpan} the set of such elements is the whole of 
  $\mc{M}$.
\end{proof}

\begin{assumption} \label{Ant:B_n}
  \(\mc{O} = \{B_n\}_{n=1}^\infty\) is a family of $R$-modules such 
  that \(B_n \supseteq B_{n+1}\) for all \(n\in\Zp\) and 
  \(\bigcap_{n=1}^\infty B_n = \{0\}\).
\end{assumption}

\begin{definition}
  A set \(N \subseteq \mc{M}\) is said to be \DefOrd{open} (in 
  $\mc{M}$) if there for every \(a \in N\) exists some \(\ve \in 
  \mc{O}\) such that
  \begin{equation} \label{Eq:Def.Open}
    N \supseteq \setOf{ a+b }{ b \in \ve }
    \text{.}
  \end{equation}
  To put it differently: The topology on $\mc{M}$ is the group 
  topology for which $\mc{O}$ is a fundamental system of 
  neighbourhoods of $0$. 
\end{definition}

Many arguments involving topology in subsequent sections will be 
expressed using $\ve$-$\delta$-formalism, but since $\ve$ and 
$\delta$ will be neighbourhoods of $0$ rather than the conventional 
positive real numbers, a few examples of what this formalism looks 
like may be in order. First and foremost, the 
\emDefOrd[{neighbourhood}*]{$\ve$-neighbourhood} of 
an element $a$ is the set \(a+\ve := \setOf{ a+b }{ b \in \ve }\). 
A map $f$ is continuous at $0$ if there for every \(\ve \in 
\mc{O}\) exists some \(\delta \in \mc{O}\) such that \(f(\delta) := 
\setOf[\big]{ f(a) }{ a \in \delta } \subseteq \ve\). It should 
furthermore be observed that continuity at $0$, for group 
homomorphisms, is equivalent to continuity everywhere (and even to 
uniform continuity everywhere). That $\delta$ is smaller than (or 
equal to) $\ve$ is of course expressed as \(\delta \subseteq \ve\), 
and the minimum of $\ve$ and $\delta$ is $\ve\cap\delta$.

In addition to these general properties of neighbourhood arithmetic, 
there are also some special properties following from 
Assumption~\ref{Ant:B_n} that are of great importance here. Firstly 
\(\ve + \ve = \ve\) for every \(\ve \in \mc{O}\) since 
$\ve$ is a group. Similarly \(\ve - \ve = \ve\) and \(r(\ve) 
\subseteq \ve\) for all \(r \in R\). Finally the inclusion of 
$B_m$ in $B_n$ whenever \(m>n\) implies that \(\ve + \delta = \ve \cup 
\delta\) for all \(\ve,\delta \in \mc{O}\).

\medskip

The choice of $\mc{O}$ is a rather extensive topic, with many 
different approaches that should be mentioned, so it seems best to 
leave that for the end of this section, and instead proceed with the 
things that can be done as soon as $\mc{O}$ is in place.

\begin{lemma}
  The group operations on $\mc{M}$ and all elements in $R$ are 
  continuous.
\end{lemma}
\begin{proof}
  Let \(s\colon \mc{M} \times \mc{M} \Fpil \mc{M} : (a,b) \mapsto 
  a-b\) be subtraction as a map; proving it continuous implies the 
  same for the standard group operations addition and negation. Let 
  \(N \subseteq \mc{M}\) be an arbitrary open set and let \((a,b) \in 
  s^{-1}(N)\) be arbitrary too. Since \(a-b \in N\) there exists some 
  \(\ve \in \mc{O}\) such that \((a -\nobreak b) + \ve \subseteq N\). 
  For this $\ve$, \((a +\nobreak \ve) - (b +\nobreak \ve) = 
  (a -\nobreak b) + (\ve -\nobreak \ve) = (a -\nobreak b) + \ve 
  \subseteq N\), and hence \( (a +\nobreak \ve) \times (b +\nobreak \ve) 
  \subseteq s^{-1}(N)\), which means $(a,b)$ is an interior point of 
  $s^{-1}(N)$. It follows that $s^{-1}$ maps open sets to open sets, 
  and hence $s$ is continuous.
  
  Now let \(r \in R\) and an open set \(N \subseteq \mc{M}\) be 
  arbitrary. For every \(a \in r^{-1}(N)\) there exists some \(\ve 
  \in \mc{O}\) such that \(r(a) + \ve \subseteq N\), and hence \(a + 
  \ve \subseteq r^{-1}(N)\) because \(r(a +\nobreak \ve) = r(a) + 
  r(\ve) \subseteq r(a) + \ve \subseteq N\). Thus $r$ is continuous.
\end{proof}

The next step is to go from $\mc{M}$ to its completion $\cmplM$, 
which can be constructed in the standard way as the set of 
equivalence classes of Cauchy sequences in $\mc{M}$, where two 
sequences $\{a_n\}_{n=1}^\infty$ and $\{b_n\}_{n=1}^\infty$ are 
considered equivalent if \(\lim_{n \Lpil \infty} (a_n -\nobreak b_n) = 
0\). The topology in the completion can be defined in terms of 
limits: \(a = \lim_{n\Lpil\infty}a_n\) for \(a_n = \bigl[ 
\{b_{n,k}\}_{k=1}^\infty \bigr]\) if and only if 
\(\{b_{n,n}\}_{n=1}^\infty\) is a Cauchy sequence in $\mc{M}$ and 
\(a = \bigl[\{b_{n,n}\}_{n=1}^\infty\bigr]\). 
The equivalence classes of the constant sequences provide the 
canonical embedding of $\mc{M}$ into its completion, and it is 
convenient to identify this with the original $\mc{M}$.

An alternative approach, which fits better in with many textbook 
definitions of the completion, is to turn $\mc{M}$ into a metric space 
and make use of this explicit metric when defining e.g.~the topology 
of $\cmplM$. (Both approaches yield the same end result.) There are 
several metrics which all reproduce the topology of $\mc{M}$, but the 
following is often the most natural:
\begin{equation} \label{Eq:O-metrik}
  d(a,b) = \begin{cases}
    1& \text{if \(a-b \notin B_1\),}\\
    \inf \setOf{ 2^{-n} }{ a-b \in B_n }& \text{otherwise}
  \end{cases}
\end{equation}
for all \(a,b \in \mc{M}\). That \(d(a,a)=0\) makes use of the 
infimum, whereas in the formula for $d(a,b)$ when \(a \neq b\) this 
$\inf$ is equivalent to a $\min$.

\begin{definition}
  The completion of $\mc{M}$ is denoted \index{M-bar@$\cmplM$}$\cmplM$, 
  and $\mc{M}$ is considered to be a subset of $\cmplM$. For any 
  \(N \subseteq \cmplM\), the topological closure in $\cmplM$ of $N$ 
  is denoted $\cmpl{N}$. Let \index{O hat@$\cmplO$}\(\cmplO = 
  \left\{ \cmpl{B_n} \right\}_{n=1}^\infty\).
  
  The group operations extend by continuity to the whole of 
  $\cmplM$, as do the homomorphisms in $R$, and will henceforth be 
  considered to be defined on the whole of $\cmplM$. Accordingly, 
  any subgroup \(N \subseteq \cmplM\) is said to be an 
  \DefOrd[{module}*]{$R$-module} for which \(r(a) \in N\) whenever 
  \(a \in N\) and \(r \in R\). If \(Z \subseteq \cmplM\) is some set, 
  then $\Span(Z)$\index{Span@$\Span$} will denote the smallest 
  $R$-module which contains $Z$. Denote by 
  $\Cspan(Z)$\index{Cspan@$\Cspan$} the topological closure of 
  $\Span(Z)$.
\end{definition}

\begin{lemma}
  $\cmplO$ is a fundamental system of neighbourhoods of $0$ in 
  $\cmplM$. In particular the elements of $\cmplO$ are clopen 
  (simultaneously closed and open), whence the topology of $\cmplM$ 
  is zero-dimensional and totally disconnected.
\end{lemma}
\begin{proof}
  First consider an arbitrary \(F \subseteq \cmplM\) that is closed 
  and not disjoint from $\cmpl{B_n}$ for any \(n\in\Zp\); as an 
  auxiliary result it will be shown that such an \(F \owns 0\). 
  The closure $\cmpl{N}$ of some \(N \subseteq \mc{M}\) 
  consists of those points \(a \in \cmplM\) for which there exists 
  some Cauchy sequence \(\{a_k\}_{k=1}^\infty \subseteq N\) such that 
  \(a = \lim_{k\Lpil\infty} a_k\). Let $\{a_n\}_{n=1}^\infty \subseteq 
  \cmplM$ be a sequence such that \(a_n \in F \cap \cmpl{B_n}\) and 
  let \(\{b_{n,m}\}_{m,n=1}^\infty \subseteq \mc{M}\) be a collection 
  of points such that \(a_n = \lim_{m\Lpil\infty}b_{n,m}\) and 
  \(\{b_{n,m}\}_{m=1}^\infty \subseteq B_n\) for all $n$. Clearly 
  \(b_{n,n} \in B_n\) for all $n$ and thus \(b_{n,n} \Lpil 0\) as 
  \(n\Lpil\infty\), which implies \(\lim_{n\Lpil\infty} a_n = 0\) as 
  well. Since $F$ was closed, it follows that \(0 \in F\).
  
  Now let \(U \subseteq \cmplM\) be an arbitrary open neighbourhood of 
  $0$, and consider the matter of whether $U$ contains some 
  $\cmpl{B_n}$ for $n$ large enough. The complement \(F = \cmplM 
  \setminus U\) is a closed set that does not contain $0$, and hence 
  by the converse of the above result there is some $n$ for which 
  \(F \cap \cmpl{B_n} = \varnothing\), meaning that \(\cmpl{B_n} 
  \subseteq U\).
  
  Consider next the problem of showing that all the $\cmpl{B_n}$ are 
  clopen (both open and closed). A Cauchy sequence 
  \(\{a_k\}_{k=1}^\infty \subseteq \mc{M}\) has the property that it 
  is \emph{either} eventually in $B_n$ \emph{or} eventually in the 
  complement $\mc{M}\setminus B_n$, because by the definition of 
  Cauchy sequence there exists some $m$ such that if \(i,j \geqslant 
  m\) then \(a_i-a_j \in B_n\), or equivalently \(a_i \in a_j + 
  B_n\), and thus if \(a_j \in B_n\) for some \(j \geqslant m\) then 
  \(a_i \in a_j + B_n \subseteq B_n + B_n = B_n\) for all \(i 
  \geqslant m\), in which case $\{a_k\}_{k=1}^\infty$ indeed is 
  eventually in $B_n$. If no \(a_j \in B_n\) for \(j \geqslant m\) 
  then instead \(a_j \in \mc{M}\setminus B_n\) for all \(j \geqslant 
  m\), and consequently $\{a_k\}_{k=1}^\infty$ will be eventually in 
  $\mc{M} \setminus B_n$. This property of Cauchy sequences means a 
  sequence can converge either to an element of $\cmpl{B_n}$ or to an 
  element of $\cmpl{\mc{M} \setminus B_n}$, but not both, and 
  therefore these sets will be disjoint; $\cmpl{B_n} \cup 
  \cmpl{\mc{M} \setminus B_n}$ is a partition of $\cmplM$. Since both 
  parts in this partition are closed by definition, they are also both 
  open, and in particular $\cmpl{B_n}$ is both closed and open.
  
  This has shown that $\cmplO$ is a fundamental system of 
  neighbourhoods of $0$ in $\cmplM$, and that its members are all 
  clopen. A topology is said to be zero-dimensional if it has a 
  basis consisting entirely of clopen sets, and every space with a 
  zero-dimensional topology is totally disconnected.
\end{proof}

Zero-dimensional topologies are, just like Zariski topologies, 
perfectly fine topologies (i.e., all the axioms hold and hence the 
basic theorems follow), but a bit unsettling when one first 
encounters them as things do not 
behave in quite the way one has gotten used to in the standard 
topology on $\R$\Dash the multitude of sets that are open \emph{and} 
closed at the same time being the most obvious oddity. Being metric, 
and consequently Hausdorff, the topology on $\cmplM$ does however 
have much more in common with the standard topology on $\R$ than it 
has with Zariski topologies, so it is not all that far out.

\begin{exercise}
  An intuition for spaces like $\cmplM$ may be found by comparing 
  them to Cantor sets, as the two have many traits in common. 
  Indeed, for \(\mc{M} = \Z_2[x]\) (univariate polynomials over 
  \(\Z_2 = \Z/2\Z\)) and \(B_n = \mc{M} x^{\lceil \alpha n \rceil}\) 
  where \(\alpha = \log_3 2\), 
  the completion $\cmplM$ is very similar to the standard Cantor set. 
  Show that the map \(\phi\colon \mc{M} \Fpil [0,1]\) defined by
  \begin{equation*}
    \phi\biggl( \sum_{k=0}^n (s_k + 2\Z) x^k \biggr) =
    \frac{2}{3} \sum_{k=0}^n s_k 3^{-k}
    \qquad
    \text{for all \(\{s_k\}_{k=0}^n \subseteq \{0,1\}\)}
  \end{equation*}
  satisfies
  \begin{equation*}
    C_1 d(a,b) \leqslant \norm[\big]{ \phi(a) - \phi(b) } \leqslant 
    C_2 d(a,b)
    \quad\text{for all \(a,b \in \mc{M}\),}
  \end{equation*}
  for some positive constants $C_1$ and $C_2$, where $\norm{\cdot}$ 
  is the standard absolute value on $\R$ and $d$ is the metric from 
  \eqref{Eq:O-metrik}. Conclude that $\phi$ extends to a 
  homeomorphism from $\cmplM$ to the remove-middle-third Cantor set 
  on the unit interval.
\end{exercise}

\begin{lemma} \label{L:Cspan-uppdelning}
  If \(A \subseteq \cmplM\) and \(a \in \Cspan(A)\) then for every 
  \(\ve \in \cmplO\) there exists a natural number $n$, some elements 
  \(\{a_i\}_{i=1}^n \subseteq A\), and some maps \(\{r_i\}_{i=1}^n 
  \subseteq \pm R^*\) such that
  \begin{equation}
    \sum_{i=1}^n r_i(a_i) \in a + \ve \text{.}
  \end{equation}
\end{lemma}
\begin{proof}
  By definition of topological closure applied to $\Cspan(A)$, there 
  exists some \(b \in \Span(A)\) such that \(a-b \in \ve\). Since the 
  set of all elements on the form $\sum_{i=1}^n r_i(a_i)$ for 
  \(\{a_i\}_{i=1}^n \subseteq A\) and \(\{r_i\}_{i=1}^n \subseteq 
  \pm R^*\) constitute an $R$-module containing $A$, it follows that 
  $b$ has an expression on that form.
\end{proof}

For \(A = \mc{Y}\), this lemma is a topologized version of 
Lemma~\ref{L:M-form}, but Lemma~\ref{L:Slutenhet,I'(S)}, 
Definition~\ref{Def:Kompatibel}, Theorem~\ref{S:Konstr.Irr(S)}, and 
Definition~\ref{Def:DIS} all characterise important subsets of 
$\cmplM$ as being on the form $\Cspan(A)$ for a suitable 
\(A \subseteq \cmplM\).

\medskip

The rest of this section is a discussion of some important methods 
for constructing a topology on $\mc{M}$, i.e., for choosing a system 
of neighbourhoods $\mc{O}$. The trivial choice is to let \(B_n = \{0\}\) 
for all $n$; this equips $\mc{M}$ with the discrete topology, makes 
\(\cmplM = \mc{M}\), and simplifies the machinery below quite 
considerably. This is also the choice one should use if one wishes to 
reproduce Bergman's diamond lemma.

A nontrivial choice of topology which has long traditions in algebra 
is that of an ideal-adic topology. In this case it is assumed that 
$\mc{M}$ also has a multiplicative structure, and as $B_1$ is chosen 
a nontrivial ideal in $\mc{M}$. Then each $B_n$ is defined as the $n$th 
ideal power $B_1^n$ of $B_1$, i.e., the ideal generated by all products 
of $n$ elements from $B_1$. Such choices of $\mc{O}$ 
allow localisations of $\mc{M}$ to be treated within this framework. 
The condition that all these $B_n$ are $R$-modules is not necessarily 
fulfilled for this construction, but it follows very naturally when 
for example $\mc{M}$ is an $\mc{R}$-algebra and $R$ is the set of 
multiplication-by-a-scalar maps. Nor is necessarily 
\(\bigcap_{n=1}^\infty B_1^n = \{0\}\) for every ideal $B_1$, but it 
typically holds for the interesting choices of $B_1$.

A generalisation of the class of ideal-adic topologies is provided by 
the `filtered structures' of Robbiano~\cite{Robbiano} and 
Mora~\cite{Mora:Seven}. Here it is again assumed that $\mc{M}$ is a 
ring, and a $\Gamma$-filtration $\{F_\gamma\}_{\gamma \in \Gamma}$ of 
$\mc{M}$ is given. This means $\Gamma$ is assumed to be a totally 
ordered semigroup (written additively, but at least in 
\cite{Mora:Seven} not assumed to be commutative), and the $F_\gamma$ 
are assumed to be subgroups of $\mc{M}$ which satisfy:
\begin{enumerate}
  \item[(R1)] 
    If \(\gamma,\delta \in \Gamma\) are such that \(\gamma < \delta\) 
    then \(F_\gamma \subseteq F_\delta\).
  \item[(R2)]
    \(F_\gamma \cdot F_\delta \subseteq F_{\gamma+\delta}\) for all 
    \(\gamma,\delta \in \Gamma\).
  \item[(R3)]
    For every \(a \in \mc{M} \setminus \{0\}\) the set 
    \(\setOf{\gamma \in \Gamma}{F_\gamma \owns a}\) has a minimal 
    element.
\end{enumerate}
For actual results, these authors typically also assume that $\Gamma$ 
is `inf-limited', which means that for any infinite strictly 
descending sequence \(\gamma_1 > \gamma_2 > \gamma_3 > \dotsb\) in 
$\Gamma$ and any given \(\gamma \in \Gamma\), there exists an $n$ 
such that \(\gamma_n < \gamma\). In this case, one can simply choose 
one such infinite strictly descending sequence \(\gamma_1 > \gamma_2 
> \gamma_3 > \dotsb\) in $\Gamma$ and define $\mc{O}$ by letting 
\(B_n = F_{\gamma_n}\) for all \(n\in\Zp\); it follows from (R3) and 
inf-limitedness that \(\bigcap_{n=1}^\infty B_n = \{0\}\). Also 
observe that the resulting topology on $\mc{M}$ is the same 
regardless of which sequence \(\{\gamma_n\}_{n=1}^\infty\) is chosen.

The abstract setting of a filtered structure only supports setting 
\(R=\varnothing\), but again the $F_\gamma$ are in many concrete cases 
modules over a ring of scalars, and then it is possible to encode the 
whole of that ring into $R$. Otherwise it is a rather striking 
feature of the filtered structure machinery that one does not assume 
any ``coefficients'' to exist from start, but rather constructs them 
from the filtered structure. Defining
\begin{align*}
  V_\gamma :={}& 
  \bigcup_{\substack{\delta\in\Gamma \\ \delta < \gamma}} F_\delta
    \text{,}\\
  G_\gamma :={}& F_\gamma / V_\gamma \text{,}\\
  G :={}& \bigoplus_{\gamma \in \Gamma} G_\gamma
\end{align*}
one gets the associated graded ring $G$ that can be used as a 
coordinatized form of $\mc{M}$. Each coordinate $a_\gamma$ then 
assumes values in the corresponding $G_\gamma$; these groups may 
vary quite a lot in size and structure, but for reasonable cases 
of $\mc{M}$ being an $\mc{R}$-algebra it often holds that each 
$G_\gamma$ is as a group isomorphic to either $\mc{R}$ or $\{0\}$.
It is also common that a filtered structure is equipped with a map 
\(f\colon \Gamma \Fpil \mc{M}\) such that \(f(\gamma) \in F_\gamma\) 
for all \(\gamma \in \Gamma\) (and \(f(\gamma) \notin V_\gamma\) 
whenever \(V_\gamma \neq F_\gamma\)), although that map is not part 
of the formal definition. The image of such an $f$ is typically the 
primary candidate for $\mc{Y}$, even though there is nothing in the 
generic formalism from which one may deduce that this image should 
span $\mc{M}$.

Even more general is the approach to define $\mc{O}$ as a family of 
balls with respect to an ultranorm $\Norm{\cdot}$ on $\mc{M}$:
\begin{equation} \label{Eq:Norm--B_n}
  B_n = \setOf[\big]{ a \in \mc{M} }{ \Norm{a} < 2^{-n} }
  \qquad\text{for all \(n\in\Zp\).}
\end{equation}
In one sense this construction is universal, because if \(\mc{O} = 
\{B_n\}_{n=1}^\infty\) is given then one can always use 
\eqref{Eq:O-metrik} to reconstruct a norm \(\Norm{a} = d(a,0)\) that 
in \eqref{Eq:Norm--B_n} would give rise to the original 
neighbourhood system $\mc{O}$, but more important is that it often 
provides a convenient method for arriving at a topology with 
desirable properties.

The standard construction of an ultranorm in the case \(\mc{M} = 
\RavX\) is to pick any function \(w\colon X \Fpil \R\)
and define the ultranorm $\Norm{\cdot}$ on \(\mc{Y} = X^*\) 
to be the unique monoid homomorphism \(\mc{Y} \Fpil \Rp\) that 
satisfies \(\Norm{x} = 2^{w(x)}\) for all \(x \in X\); in other words 
\(\Norm{\prod_{i=1}^n x_i} := \prod_{i=1}^n 2^{w(x_i)}\) for all 
\(x_1 x_2 \dotsb x_n \in \mc{Y}\). This is then extended to the whole 
of $\mc{M}$ by defining 
\begin{equation} \label{Eq:Norm&Span}
  \Norm{a} := \max_{\mu \in Z} \Norm{\mu}
  \quad\text{where \(Z \subset \mc{Y}\) is minimal such that 
  \(a \in \Span(Z)\),}
\end{equation}
and in particular letting \(\Norm{0}:=0\), as a sort of \(0 = 
\max\varnothing\). The $w$ is known as the \emDefOrd{weight function} 
for the norm, and its sign determines how the formal variables behave; 
if \(w(x) \geqslant 0\) then $x$ will be a polynomial-style variable, 
whereas if \(w(x)<0\) then $x$ will be a power-series-style variable. 
The logarithm of $\Norm{\cdot}$ behaves as a weighted polynomial-style 
degree function on \(\mc{M} = \RavX\).

\begin{definition} \label{Def:Ultranorm}
  Formally, a function \(a \mapsto \Norm{a} : \mc{M} \Fpil \R\) is said 
  to be a (group) \DefOrd{ultranorm} if
  \begin{enumerate}
    \item[(i)] \(\Norm{a} \geqslant 0\) for all \(a \in \mc{M}\).
    \item[(ii)] \(\Norm{a - b} \leqslant \max\bigl\{ \Norm{a}, \Norm{b} 
      \bigr\}\) for all \(a,b \in \mc{M}\).
    \item[(iii)] \(\Norm{a} = 0\) for some \(a \in \mc{M}\) if and only 
      if \(a=0\).
  \end{enumerate}
  If $\mc{M}$ is a ring and in addition
  \begin{enumerate}
    \item[(iv)] \(\Norm{ab} \leqslant \Norm{a}\Norm{b}\) for all 
    \(a,b \in \mc{M}\)
  \end{enumerate}
  then $\Norm{\cdot}$ is said to be a \DefOrd{ring ultranorm}. 
  If instead $\mc{R}$ is a ring with ultranorm $\norm{\cdot}$ and 
  $\mc{M}$ is an $\mc{R}$-module, then a group ultranorm $\Norm{\cdot}$ 
  on $\mc{M}$ is said to be a \DefOrd{module ultranorm} if
  \begin{enumerate}
    \item[(v)] \(\Norm{ra} \leqslant \norm{r} \Norm{a}\) for all 
      \(r \in \mc{R}\) and \(a \in \mc{M}\).
  \end{enumerate}
  An \DefOrd{algebra ultranorm} has to satisfy all of (i)--(v).
\end{definition}

The `ultra' prefix pertains primarily to property (ii)\Ldash the 
\emDefOrd{strong triangle inequality}\Dash and in particular to its 
right hand side \(\max\bigl\{ \Norm{a}, \Norm{b} \bigr\}\), which is 
more strict than the $\Norm{a}+\Norm{b}$ of the ordinary triangle 
inequality. Among the direct consequences of~(ii) are that any 
$\ve$-neighbourhood of $0$\Ldash i.e., any set of the form 
$\setOf[\big]{ a \in \mc{M} }{ \Norm{a} < \ve }$ for a real number 
\(\ve>0\)\Dash is a subgroup of~$\mc{M}$. 

The \emDefOrd{trivial ultranorm} has \(\Norm{0}=0\) and \(\Norm{a}=1\) 
for all \(a \neq 0\); it exists for all groups and reproduces the 
discrete topology.

If $\mc{M}$ is an $\mc{R}$-module and the ultranorm $\Norm{\cdot}$ on 
$\mc{M}$ satisfies \eqref{Eq:Norm&Span} then equipping $\mc{R}$ with 
the trivial ultranorm will make $\Norm{\cdot}$ an $\mc{R}$-module 
ultranorm. This is typically the ``correct'' scalar norm for a formal 
power series problem, as all nonzero scalar values are then equivalent 
for matters of series convergence. Conversely it is often convenient 
to define the norm on $\mc{M}$ so that it becomes an $\mc{R}$-module 
norm with respect to some given norm $\norm{\cdot}$ on $\mc{R}$. In 
the particular cases where $\mc{Y}$ is an $\mc{R}$-module basis for 
$\mc{M}$, then one may choose to make \(\Norm{r\mu} = \norm{r} 
\Norm{\mu}\) for all \(r \in \mc{R}\) and \(\mu \in \mc{Y}\), which 
has its advantages when it comes to defining $T_1(S)$ below. 
Non-trivial scalar norms may however require nontrivial choices also 
of $R$ and $\mc{Y}$.

What complicates the choice of a scalar norm is the assumption that 
each $B_n$ is an $R$-module, which given \eqref{Eq:Norm--B_n} is 
equivalent to the claim that \(\Norm[\big]{r(a)} < 2^{-n}\) for all 
\(r \in R\) and \(a \in \mc{M}\) satisfying \(\Norm{a} < 2^{-n}\). 
One would typically ensure this by enforcing the condition that 
\(\Norm[\big]{r(a)} \leqslant \Norm{a}\) for all \(r \in R\) and \(a 
\in \mc{M}\), and this will indeed be the case provided \(r \colon 
\mc{M} \Fpil \mc{M}\) is a map on the form \(a \mapsto sa\) for some 
scalar \(s \in \mc{R}\) such that \(\norm{s} \leqslant 1\), as then 
\(\Norm[\big]{r(a)} = \Norm{sa} \leqslant \norm{s}\Norm{a} \leqslant 
\Norm{a}\). Hence the natural choice of $R$ when $\mc{M}$ is an 
$\mc{R}$-module with a ditto ultranorm is to take
\begin{equation} \label{Eq:Def.R,norm}
  R = \setOf[\big]{ a \mapsto ra : \mc{M} \Fpil \mc{M} }{ 
    \text{\(r \in \mc{R}\) and \(\norm{r} \leqslant 1\)}
  } \text{.}
\end{equation}
Unless all scalars \(r \in \mc{R}\) satisfy \(\norm{r} \leqslant 1\), 
this will make the $R$-module concept distinct from that of an 
$\mc{R}$-module however, and this has repercussions elsewhere. 
\label{Sec2:Basics}
$\mc{Y}$ must span $\mc{M}$ as an $R$-module, so if $Y$ is an 
$\mc{R}$-module basis of $\mc{M}$ then $\mc{Y}$ may have to be chosen 
as something like the set of all products $r\mu$ for \(r\in\mc{R}\) 
and \(\mu \in Y\) in order to make it all fit.

A case where this predicament arises is that of $\mc{M}$ being a module 
over the $p$-adic numbers $\Q_p$, as these come equipped with an 
ultranorm (the $p$-adic valuation) that has \(\norm{p^n} = p^{-n}\) 
for all \(n\in\Z\). The $R$ defined by \eqref{Eq:Def.R,norm} for 
\(\mc{R} = \Q_p\) is isomorphic to the $p$-adic integers, but not to 
the entire field of $p$-adic numbers; conversely any $B_n$ defined by 
\eqref{Eq:Norm--B_n} will fail to be closed under multiplication by 
the scalar $p^{-1}$ and is thus not a $\Q_p$-module although it will 
be an $R$-module. A $\Q_p$-module basis $Y$ for $\mc{M}$ will in this 
case not be large enough to serve as $\mc{Y}$. One can instead use 
the set of all terms as suggested above, but since $p^{-1}$ together 
with the $p$-adic integers generate the whole of $\Q_p$, it is also 
sufficient to make $\mc{Y}$ the set of all products $p^{-n}\mu$ for 
\(\mu \in Y\) and \(n\in\Z\) (or even \(n\in\N\)). It is typically 
easier to construct the partial order $P$ on $\mc{Y}$ if the latter 
has a simple, discrete structure.

\section{Reductions}
\label{Sec:Reduktioner}

While the purpose of introducing $\mc{M}$, $R$, $\mc{Y}$, and $\mc{O}$ 
is primarily to fix and structure a stage for the diamond lemma, 
$T_1(S)$ is what provides the actors on that stage, so that a play 
of equivalence and normal forms may be performed. The relation of 
$T_1(S)$ to the equivalence of elements in $\cmplM$ is primarily that 
these maps preserve it\Ldash $t(a)$ must be equivalent to $a$ 
for all \(t \in T_1(S)\) and \(a \in \cmplM\)\Dash but as the 
equivalence concept is derived from $T_1(S)$, this is a theorem rather 
than an assumption. The role of 
the collection $T_1(S)$ of maps \(\cmplM\Fpil\cmplM\) is mostly that 
of a presentation: it is not uniquely determined by that which 
it is used to define, but it gets the job done, and you may use it to 
verify at least some conditions about the whole.

In applications one rarely starts from $T_1(S)$\Ldash hence the 
somewhat odd notation; $T_1(S)$ is typically constructed from a more 
fundamental set \index{S@$S$}$S$ of directed equivalences, a 
so-called \emDefOrd{rewriting system}\Dash but 
for this general proof it provides the best balance between abstract 
adaptability and concrete constructibility. As in the previous section 
some examples will be given below of how $T_1(S)$ can be constructed 
from such a more fundamental $S$, but this should be taken more as 
hints than as a full survey; it is sometimes necessary to combine 
several different methods of construction. That applications typically 
make some $S$ the fundamental entity has however influenced the choice 
of notations below, in that every object that \emph{formally} depends 
on the choice of $T_1(S)$ is \emph{written} as though it would depend 
on `$S$'. Besides being more convenient in applied instances of the 
diamond lemma, this choice of notations also simplifies comparisons 
with~\cite{Avhandlingen}.

\begin{definition}
  Let \index{T S@$T(S)$}$T(S)$ 
  be the set of all finite compositions of elements from 
  $T_1(S)$; in particular the identity map \(\id\colon \cmplM \Fpil 
  \cmplM\) is considered to be an element of $T(S)$, on account of 
  being the composition of an empty sequence of maps in $T_1(S)$. 
  The elements of $T(S)$ are called \DefOrd[*{reduction}]{reductions} 
  and the elements of $T_1(S)$ in particular are called 
  \DefOrd[*{simple reduction}]{simple reductions}.
\end{definition}

The name \emDefOrd{reduction} suggests that these maps take something 
away, and this is indeed typically the case. Standard constructions 
of reductions tend to make them more or less projections, 
and although there is no formal need for them to be, it may be helpful 
on a first reading to think of them that way. It should however be 
observed that even those reductions which really are projections tend 
to be rather skew and have very small kernels, so don't expect to use 
just one and be done with it; getting anywhere is much more like a 
round of golf, where one has to hit the ball repeatedly (typically 
making use of many different clubs) in order to get it into the hole. 

The distinction between simple and non-simple reductions is mostly in 
the eye of the beholder, because nothing prevents picking as $T_1(S)$ 
a set of maps that constitute a monoid under composition, in which 
case one would have \(T(S)=T_1(S)\) and all reductions would be simple. 
The point of letting the user designate some reductions as being simple 
is that it is often sufficient to verify a condition only for the 
simple ones, as the property in question easily extends to all 
reductions.

\begin{assumption} \label{Ant:T(S)-kont-hom}
  Every simple reduction \(t \in T_1(S)\) is a continuous group 
  homomorphism \(\cmplM \Fpil \cmplM\) which satisfies \(t \circ r = 
  r \circ t\) for all \(r \in R\).
\end{assumption}

Reductions may of course satisfy this property for a larger class of 
maps than $R$; they may for example all be $\mc{R}$-linear for some 
ring $\mc{R}$ that is larger than $R$. Therefore many lemmas below 
that say some set is an $R$-module will have a `more generally, 
\dots' clause which covers the case of additional maps $r$ that 
commute with all reductions. 

\begin{lemma}
  Every reduction \(t \in T(S)\) is a continuous group homomorphism 
  \(\cmplM \Fpil \cmplM\) which satisfies \(t \circ r = r \circ t\) 
  for all \(r \in R^*\).
\end{lemma}
\begin{proof}
  The identity map \(\id \in T(S)\) trivially satisfies the 
  properties in the lemma. All other reductions are finite 
  compositions of simple reductions, and since the composition of two 
  continuous group homomorphisms is a continuous group homomorphism, 
  it follows from Assumption~\ref{Ant:T(S)-kont-hom} that all 
  reductions are continuous group homomorphisms. Finally if \(t = 
  t_n \circ \dotsb \circ t_1\) for \(t_1,\dotsc,t_n \in T_1(S)\) and 
  \(r = r_m \circ \dotsb \circ r_1\) for \(r_1,\dotsc,r_m \in R\), 
  then \(t_i \circ r_j = r_j \circ t_i\) for all $i$ and $j$ by 
  Assumption~\ref{Ant:T(S)-kont-hom}, whence \(t \circ r = 
  t_n \circ \dotsb \circ t_1 \circ r_m \circ \dotsb \circ r_1 =
  r_m \circ \dotsb \circ r_1 \circ t_n \circ \dotsb \circ t_1 = 
  r \circ t\).
\end{proof}


\begin{lemma}
  A reduction is uniquely determined by its values on $\mc{Y}$.
\end{lemma}
\begin{proof}
  Let \(t \in T(S)\) be arbitrary. Since $t$ is continuous and 
  $\mc{M}$ is dense in $\cmplM$, the values on $\mc{M}$ uniquely 
  determine $t$.
  By Lemma~\ref{L:M-form}, any \(a \in \mc{M}\) can be expressed as
  \(a = \sum_{k=1}^n r_k(\mu_k)\) for some \(n\in\N\), 
  \(r_1,\dotsc,r_n \in \pm R^*\), and \(\mu_1,\dotsc,\mu_n \in 
  \mc{Y}\). Hence \(t(a) = \sum_{k=1}^n t\bigl( r_k(\mu_k) \bigr) 
  = \sum_{k=1}^n r_k\bigl( t(\mu_k) \bigr)\), which expresses $t(a)$ 
  purely in terms of the values on $\mc{Y}$ of $t$. 
\end{proof}

The definitions of reductions are accordingly often simplified to 
stating how they act on elements of $\mc{Y}$. A common approach is to 
define simple reductions so that they change precisely one element 
of $\mc{Y}$ while leaving all other elements the same. Concretely the 
simple reduction which changes \(\mu \in \mc{Y}\) to \(a \in \cmplM\) 
would be defined by
\begin{equation} \label{Eq:Reduktion-monom-def}
  \index{t mu mapsto a@$t_{\mu\mapsto a}$}
  t_{\mu\mapsto a}(\lambda) = \begin{cases}
    a& \text{if \(\lambda=\mu\),}\\
    \lambda& \text{otherwise,}
  \end{cases}
  \qquad\text{for all \(\lambda \in \mc{Y}\).}
\end{equation}
Such a map satisfies \(t_{\mu\mapsto a}(b) = b\) for all \(b \in 
\Cspan\bigl( \mc{Y} \setminus \{\mu\} \bigr)\), so if in addition 
\(a \in \Cspan\bigl( \mc{Y} \setminus \{\mu\} \bigr)\) (which will be 
hard to avoid while satisfying the compatibility condition of 
Definition~\ref{Def:Kompatibel}), then the image of $t_{\mu\mapsto a}$ 
will be contained in \(\Cspan\bigl( \mc{Y} \setminus \{\mu\} \bigr)\) 
and consequently this map becomes a projection. Its kernel is however 
as small as it can possibly be without being trivial, and the $a$ is 
only rarely zero, which means the projection is typically skew.

While \eqref{Eq:Reduktion-monom-def} is the standard definition of a 
simple reduction from a conceptual point of view, it is not obviously 
one which is formally sound; if for example $\mc{Y}$ is not an 
independent set in $\mc{M}$ then a map $t_{\mu\mapsto a}$ for 
arbitrary \(\mu \in \mc{Y}\) and \(a \in \cmplM\) can probably not 
both be a group homomorphism and satisfy 
\eqref{Eq:Reduktion-monom-def}. An alternative definition of 
$t_{\mu\mapsto a}$, which often is better suited for proving 
properties of this reduction, is
\begin{equation} \label{Eq:Def.t-f_mu}
  t_{\mu\mapsto a}(b) = b - f_\mu(b) \cdot (\mu - a)
  \qquad\text{for all \(b\in\cmplM\).}
\end{equation}
Prerequisites for this formula is that $\cmplM$ is some sort of 
$\mc{R}$-module that furthermore comes with coefficient-of-$\mu$ 
homomorphisms \(f_\mu\colon \cmplM \Fpil \mc{R}\) for all 
\(\mu\in\mc{Y}\); if these satisfy \(f_\mu(\mu)=1\) and 
\(f_\mu(\nu)=0\) for all \(\nu \in \mc{Y}\setminus\{\mu\}\) then 
\eqref{Eq:Reduktion-monom-def} becomes an immediate consequence of 
\eqref{Eq:Def.t-f_mu}. Assuming the $\mc{R}$-module operations on 
$\cmplM$ are continuous, the continuity of $t_{\mu\mapsto a}$ is 
furthermore implied by the continuity of the coefficient function 
$f_\mu$, and this depends only on the choices of $\mc{M}$, $R$, 
$\mc{Y}$, and $\mc{O}$. With ultranorms defined using 
\eqref{Eq:Norm&Span}, the continuity of these $f_\mu$ maps is 
typically something one gets for 
free~\cite[Ssec.~2.3.2]{Avhandlingen}. See also 
Lemma~\ref{L:Modul-ultranorm}.

Formulae like \eqref{Eq:Def.t-f_mu} can often be used to define the 
simple reductions even in cases where \eqref{Eq:Reduktion-monom-def} 
leads to contradictions due to dependencies between elements of 
$\mc{Y}$. One example of this is the situation that $\mc{M}$ is a free 
$\mc{R}$-module with basis $Y$, but $R$ is less than $\mc{R}$ and 
$\mc{Y}$ therefore has been chosen as the set of all multiples 
$r\mu$ for \(r\in\mc{R}\) and \(\mu \in Y\). If the range of $\mu$ in 
\eqref{Eq:Def.t-f_mu} is restricted to $Y$ then this formula still 
makes perfect sense, but the result is of course rather a map 
$t_{\mu \mapsto a}$ satisfying
\begin{equation*}
  t_{\mu\mapsto a}(\lambda) = \begin{cases}
    ra& \text{if \(\lambda=r\mu\) for some \(r\in\mc{R}\),}\\
    \lambda& \text{otherwise,}
  \end{cases}
  \qquad\text{for all \(\lambda \in \mc{Y}\).}
\end{equation*}
The underlying idea for all these definitions of a reduction $t_{\mu 
\mapsto a}$ is to distinguish the part of a general element of 
$\cmplM$ that corresponds to the particular undesired element $\mu$ 
of $\mc{Y}$, and then replace this part by something it is equivalent 
to. This process is often straightforward for concrete problems, even 
though it may seem difficult to formalise in general.

\begin{definition}
  A reduction \(t \in T(S)\) is said to \DefOrd{act trivially} on 
  some \(a \in \cmplM\) if \(t(a)=a\). An element \(a \in \cmplM\) 
  is said to be \DefOrd{irreducible} (with respect to $T(S)$) if 
  all \(t \in T(S)\) act trivially on it. The set of all irreducible 
  elements in $\cmplM$ is denoted \index{Irr@$\Irr$}\(\Irr(S)\).
  Also let\index{I(S)@$\mc{I}(S)$}
  \begin{equation} \label{Eq:Def.I'(S)}
    \mc{I}(S) = \overline{ \sum_{t \in T(S)} 
      \setOf[\big]{ a - t(a) }{ a \in \cmplM } }
  \end{equation}
  and write \(a \equiv b \pmod{S}\)\index{= mod S@$\equiv\pmod{S}$} 
  for \(a - b \in \mc{I}(S)\). An \(a \in \Irr(S)\) is said to be a 
  \DefOrd[*{normal form}]{normal form of \(b \in \cmplM\)} if 
  \(a \equiv b \pmod{S}\).
\end{definition}

The main theme in the next two sections is to define a projection $\tS$ 
of $\cmplM$ onto $\Irr(S)$ that constitutes a kind of pointwise limit 
of $T(S)$, and then demonstrate that $\mc{I}(S)$ is the kernel of that 
projection; from this will follow that there exists a unique normal 
form (which is computed by the map $\tS$) for every element of $\cmplM$. 
When it all works out, there is an equivalence 
\begin{equation}
  a \equiv b \pmod{S}  \Epil  \tS(a)=\tS(b)
  \qquad\text{for all \(a,b \in \cmplM\),}
\end{equation}
where the right hand side is algorithmic in style and well suited for 
calculations, whereas the congruence relation $\equiv$ in the left hand 
side is identifiable as the 
reflexive--symmetric--transitive--algebraic--topological closure of 
`\(a \equiv t(a)\) for all \(t \in T_1(S)\) and \(a \in \cmplM\)'. 
A problem that is difficult on one side of the equivalence may have 
an obvious solution when transported to the other side of it; the main 
direction for decision problems is left to right, whereas identities 
tend to be simpler to derive on the left side. Applied calculations 
often focus on the irreducible elements, because the set $\Irr(S)$ 
can be used as a model for the quotient set \(\cmplM \big/ 
{\equiv}\,(\mathrm{mod}\,S)\).

It should be pointed out that this concept of irreducibility has 
nothing to do with \emph{multiplicative irreducibility} (the property 
that the only factorisations of an element are the trivial ones), nor 
for that matter with for example \emph{join-irreducibility} (which in 
lattice theory is the equally important property that an element 
cannot be expressed as the $\vee$ of two other elements), but the point 
about multiplicative irreducibility needs to be stressed since many 
algebraists are accustomed to interpreting an unqualified `irreducible' 
as referring to precisely multiplicative irreducibility; indeed this 
tendency is so strong that many authors seek other names to use for 
this concept. One such synonym is \emDefOrd{normal} (which in this 
sense most commonly occurs in the phrase `normal form'), that 
unfortunately also has the alternative interpretations ``having norm 
$1$'' and ``being orthogonal to tangents'', which are quite different. 
Another synonym is \emDefOrd{terminal}, which refers to the fact that 
reduction stops when reaching one of these elements\Dash however in this 
topologized setting reductions do not in general stop completely; they 
merely ``slow down'' when approaching the limit. `Irreducible' is the 
term that is used in~\cite{Bergman} and must therefore be considered 
established and standard.

\begin{lemma} \label{L: Irr modul}
  An element of $\cmplM$ is irreducible if and only if every simple 
  reduction acts trivially on it. 
  The set $\Irr(S)$ is a topologically closed $R$-module. 
  More generally, any continuous group homomorphism \(r\colon \cmplM 
  \Fpil \cmplM\) satisfying \(r \circ t = t \circ r\) for all \(t \in 
  T(S)\) maps $\Irr(S)$ into itself.
\end{lemma}
\begin{proof}
  Clearly every simple reduction acts trivially on an irreducible 
  element. Conversely every non-simple reduction is a composition of 
  simple reductions, and if all of these act trivially on an element, 
  then the composite reduction must do so too. Hence all elements 
  which simple reductions act trivially upon are irreducible.
  
  Let \(t \in T(S)\) be arbitrary. Then the set $I_t$ of all \(b \in 
  \cmplM\) such that \(t(b)=b\) can alternatively be characterised as 
  the kernel of the map \(t'\colon \cmplM\Fpil\cmplM\) defined by 
  \(t'(b)=t(b)-b\). Since $t$ is a continuous group homomorphism, 
  $t'$ will be one too, and thus the set $I_t$ will be a subgroup of 
  $\cmplM$. Moreover $I_t$ is topologically closed since it is the 
  inverse image of $\{0\}$, which is a closed set. If $r$ is a 
  homomorphism commuting with $t$ then for any \(b \in I_t\), 
  \(0 = r \bigl( t'(b) \bigr) = (r \circ\nobreak t)(b) - r(b) = 
  (t \circ\nobreak r)(b) - r(b) = t'\bigl( r(b) \bigr)\), and hence 
  $r$ maps $I_t$ into itself.
  Finally \(\Irr(S) = \bigcap_{t \in T(S)} I_t\) and hence 
  $\Irr(S)$ must also be a topologically closed $R$-module, since 
  these properties are preserved under arbitrary intersections.
\end{proof}

There is a similar set of basic properties that hold for the 
complementary set $\mc{I}(S)$ of elements equivalent to $0$, 
but before going into 
that it is convenient to expound a bit on some additional twists in 
the usual construction of simple reductions. It was said above that 
$S$ could be a set of ``directed equivalences'', usually known as 
\emDefOrd[*{rewrite rule}]{rewrite rules} or simply \emph{rules}. 
Concretely these rules may be expressed as pairs \((\mu,a) \in \mc{Y} 
\times \cmplM\) where $\mu$ (the `principal' or `leading' part) is 
to be replaced by $a$. The homomorphism $t_{\mu \mapsto a}$ as 
constructed in \eqref{Eq:Reduktion-monom-def} 
or \eqref{Eq:Def.t-f_mu} implements this replacement, but $T_1(S)$ 
typically contains more than just those $t_{\mu \mapsto a}$ maps for 
which \((\mu,a) \in S\); there will also be reductions which arise 
from placing the basic rule into various contexts. This produces 
reductions that apply in cases where $\mu$ occurs as a part 
of a larger expression.

In the classical case of Bergman's diamond lemma, where \(\cmplM = 
\RavX\) and \(\mc{Y} = X^*\), this means that a pair 
\((\mu,a) \in S\) should not only give rise to a simple reduction 
which maps $\mu$ to $a$, it should also for every multiple 
$\nu_1\mu\nu_2$ of $\mu$ give rise to a simple reduction designed to 
map that $\nu_1\mu\nu_2$ to the corresponding multiple $\nu_1 a \nu_2$ 
of $a$. Thus if one defines
\begin{equation} \label{Eq1:Trad.def.T_1(S)}
  \index{t nu1 s nu2@$t_{\nu_1 s \nu_2}$}
  t_{\nu_1 s \nu_2}(\lambda) = \begin{cases}
    \nu_1 a_s \nu_2& \text{if \(\lambda = \nu_1\mu_s\nu_2\),}\\
    \lambda& \text{otherwise,}
  \end{cases}
\end{equation}
for all \((\mu_s,a_s) = s \in S\) and \(\lambda,\nu_1,\nu_2 \in 
\mc{Y}\), then the corresponding construction of $T_1(S)$ is
\begin{equation} \label{Eq2:Trad.def.T_1(S)}
  T_1(S) = 
  \setOf{ t_{\nu_1 s \nu_2} }{ \nu_1,\nu_2 \in \mc{Y}, s \in S }
  \text{.}
\end{equation}
In this case, it follows that a monomial \(\lambda \in X^*\) is 
irreducible if and only if it is not a multiple of $\mu_s$ for any 
\(s \in S\). If $S$ is finite then the irreducible words furthermore 
constitute a regular language (i.e., they can be described by a 
regexp), although infinite rewriting systems are unavoidable for some 
equivalences. Conversely, it would typically not be possible to make 
do with a finite system $S$ unless there was this kind of ``put rule 
into all possible contexts'' mechanism manufacturing an infinite 
family of reductions from every concrete rule. Though this twist is 
not a technical necessity, it is in practice very helpful.

The basic idea of putting rules into all possible contexts remains 
useful in general, but beyond associative algebra it quickly becomes 
difficult to express concretely in elementary notation. The abstract 
form of this construction is that one has a set $V$ of continuous 
homomorphisms \(\cmplM \Fpil \cmplM\), which furthermore map $\mc{Y}$ 
into itself, and defines simple reductions $t_{v,s}$ through
\begin{equation} \label{Eq:Def.Red.Prolong}
  \index{t v s@$t_{v,s}$}
  t_{v,s}(\lambda) = \begin{cases}
    v(a_s)& \text{if \(\lambda = v(\mu_s)\),}\\
    \lambda& \text{otherwise,}
  \end{cases}
\end{equation}
for all \((\mu_s,a_s) = s \in S\) and \(v \in V\). The set of maps 
which gives rise to the classical case described above is
\begin{equation} \label{Eq:2side-mult.}
  V = \{ b \mapsto \nu_1 b \nu_2 \}_{\nu_1,\nu_2 \in X^*} 
  \text{,}
\end{equation}
but one can also consider other families\Dash see below for some 
examples.

If $V$ (as above) is a monoid under composition then the construction 
\eqref{Eq:Def.Red.Prolong} has the effect that there for every such 
reduction $t_{v,s}$, every \(w \in V\), and every \(\lambda \in 
\mc{Y}\) exists another reduction $t_{w \circ v,s}$ which satisfies 
\(t_{w \circ v,s} \bigl( w(\lambda) \bigr) = w \bigl( 
t_{v,s}(\lambda) \bigr)\), and if $w$ is injective then it's even 
\(t_{w \circ v,s} \circ w = w \circ t_{v,s}\). This turns 
out to be a very useful property, so it deserves a name.

\begin{definition} \label{Def:Framflyttbar1}
  A map \(v\colon \cmplM \Fpil \cmplM\) is said to be 
  \DefOrd{advanceable} (with respect to $T_1(S)$) if there for every 
  \(t \in T_1(S)\) and \(b \in \RstarY\) exists some \(u \in T(S)\) 
  such that
  \(
    u\bigl( v(b) \bigr) = v\bigl( t(b) \bigr) 
  \).
  The map $v$ is said to be 
  \DefOrd[*{advanceable!absolutely}]{absolutely advanceable} (with 
  respect to $T(S)$) if there for every \(t \in T(S)\) exists some 
  \(u \in T(S)\) such that \(v \circ t = u \circ v\). For contrast, 
  the ordinary advanceability may also be called 
  \emDefOrd[*{advanceable!conditionally}]{conditional advanceability}.
\end{definition}

Advanceable maps provide a way to reason about and take advantage of 
the kind of structures in the set of reductions that arise from using 
\eqref{Eq:Def.Red.Prolong} or some variation thereof. The name comes 
from the point of view that if an advanceable map $v$ and a reduction 
are both to be applied to some element, then one can always arrange 
things so that $v$ is applied \emph{before} the reduction (it can be 
\emph{advanced} past a reduction), even if that may come at the price 
of having to apply a different reduction. A point worth noticing is 
that absolute advanceability only needs to be checked for simple 
reductions, as a map $v$ can be advanced past $t_1 \circ t_2$ if it 
can be advanced past $t_1$ and $t_2$, whereas a conditional 
advancement need not have this composition property; different terms 
of $t_2(b)$ may call for different translations of $t_1$ when one 
tries to advance $v$ past $t_1$. It is however the conditional 
variant that in practice is most important.

Absolute advanceability can be viewed as a weaker form of the 
`more generally' condition in Lemma~\ref{L: Irr modul}, in that it 
doesn't require the two reductions to be equal; this point of view is 
employed in the next lemma. Another way of expressing this condition 
is that \(v \circ T(S) \subseteq T(S) \circ v\), and therefore some 
may prefer to describe an absolute advanceable map as being an element 
of the left-normaliser of $T(S)$, but that characterisation seems 
difficult to use for conditional advanceability. Furthermore the 
characterisation as element of the left-normaliser breaks down in the 
more general case of a multi-sorted structure; 
Definition~\ref{Def:Framflyttbar2} gives the whole story.

\begin{lemma} \label{L:Slutenhet,I'(S)}
  The set $\mc{I}(S)$ is a topologically closed $R$-module. More 
  generally, any advanceable continuous group endomorphism on $\cmplM$ 
  maps $\mc{I}(S)$ into itself. Any reduction maps $\mc{I}(S)$ into 
  itself, and in particular \(\ker t \subseteq \mc{I}(S)\) for all 
  \(t \in T(S)\). Furthermore
  \begin{multline} \label{Eq:AltSpecI(S)}
    \mc{I}(S) = 
    \overline{ \sum_{t \in T_1(S)} \!\!
      \setOf[\big]{ a - t(a) }{ a \in \cmplM } } =
    \overline{ \sum_{t \in T_1(S)} \!\!
      \setOf[\big]{ a - t(a) }{ a \in \mc{M} } } = \\ =
    \Cspan\Bigl( 
      \setOf[\big]{ \mu - t(\mu) }{ \mu \in \mc{Y}, t \in T_1(S)} 
      \Bigr)
    \text{.}
  \end{multline}
\end{lemma}
\begin{proof}
  Let \(t \in T(S)\) be arbitrary. For any \(b,c \in \cmplM\) one 
  finds that
  \begin{equation*}
    \bigl( b - t(b) \bigr) - \bigl( c - t(c) \bigr) =
    b - c - t(b) + t(c) =
    (b-c) - t(b-c) \in
    \setOf[\big]{ a - t(a) }{ a \in \cmplM }
  \end{equation*}
  and hence \(N_t := \setOf[\big]{ a - t(a) }{ a \in \cmplM }\) is a 
  subgroup of $\cmplM$. Obviously the sum of a family of groups is a 
  group, and the closure of a group is a group because the group 
  operation (addition) is continuous. Hence $\mc{I}(S)$ is a group, 
  and it is topologically closed by definition. From the similar 
  observation that any \(a \in \ker t\) satisfies \(a = a - t(a) \in 
  N_t \subseteq \mc{I}(S)\), it follows that \(\ker t \subseteq 
  \mc{I}(S)\).
  
  Now let \(r\colon \cmplM \Fpil \cmplM\) be an absolutely advanceable 
  continuous group homomorphism. Let \(t \in T(S)\) be arbitrary and 
  choose some \(t' \in T(S)\) such that \(r \circ t = t' \circ r\). 
  For any \(b \in \cmplM\) one finds that
  \begin{equation*}
    r\bigl( b - t(b) \bigr) =
    r(b) - r\bigl( t(b) \bigr) =
    r(b) - t'\bigl( r(b) \bigr) \in N_{t'}
  \end{equation*}
  and hence $r$ maps $N_t$ into $N_{t'}$. 
  It follows from the fact that $r$ is a homomorphism that $r$ maps 
  \(N = \sum_{t \in T(S)} N_t\) into itself, and then from the fact 
  that $r$ is continuous that it maps \(\mc{I}(S) = \cmpl{N}\) into 
  itself. Since in particular all \(r \in R\) are absolutely 
  advanceable, it follows that $\mc{I}(S)$ is an $R$-module.
  
  Again let \(t \in T(S)\) be arbitrary and consider the matter of 
  whether $t$ maps $\mc{I}(S)$ into itself. For any \(t' \in T(S)\) 
  and \(b \in N_{t'}\) it holds that \(b=a-t'(a)\) for some \(a \in 
  \cmplM\), and thus
  $$
    t(b) = t(a) - (t \circ\nobreak t')(a) = 
    a - (t \circ\nobreak t')(a) - a + t(a) \in 
    N_{t \circ t'} + N_t \subseteq \mc{I}(S) \text{.}
  $$
  Hence \(t(N_{t'}) \subseteq \mc{I}(S)\) for all \(t' \subseteq 
  T(S)\) and this extends as above to arbitrary elements of 
  $\mc{I}(S)$; an arbitrary reduction \(t \in T(S)\) maps 
  $\mc{I}(S)$ into itself.
  
  For the claim that $\mc{I}(S)$ can be constructed from the $N_t$ 
  groups of simple reductions, one may first observe that $N_{\id}$ 
  is just $\{0\}$, and thus does not contribute anything unique to 
  $\mc{I}(S)$. Any other nonsimple reduction \(t \in T(S)\) is a 
  finite composition \(t_n \circ \dotsb \circ t_1 = t\) of simple 
  reductions \(t_1,\dotsc,t_n \in T(S)\) and it holds that \(N_t 
  \subseteq \sum_{k=1}^n N_{t_k}\), because if \(u_k = t_k \circ 
  \dotsb \circ t_1\) for \(k=1,\dotsc,n\) then any \(a - t(a) \in N_t\) 
  can be written as \(a - t_1(a) + u_1(a) - t_2\bigl(u_1(a)\bigr) +
  \dotsb + u_{n-1}(a) - t_n\bigl( u_{n-1}(a) \bigr) \in N_{t_1} + 
  N_{t_2} + \dotsb + N_{t_n}\).
  Hence \(\sum_{t \in T(S)} N_t \subseteq \sum_{t \in T_1(S)} N_t\); 
  the terms for simple reductions suffice for producing the total sum.
  
  Define \(M_t := \setOf[\big]{ a - t(a) }{ a \in \mc{M} }\) for all 
  \(t \in T_1(S)\). The last equality in \eqref{Eq:AltSpecI(S)} 
  follows from the observation that \(M_t =
  \Span\Bigl( \setOf[\big]{ \mu - t(\mu) }{ \mu \in \mc{Y} } \Bigr)\) 
  for all \(t \in T_1(S)\). In the middle equality, 
  the $\supseteq$ inclusion trivially follows from \(N_t \supseteq 
  M_t\). For the reverse inclusion, let $b$ in the closure of 
  $\sum_{t \in T_1(S)} N_t$ be given. Let \(\ve \in \cmplO\) be 
  arbitrary. Since $b+\ve$ is a neighbourhood of $b$ it contains some 
  element of $\sum_{t \in T_1(S)} N_t$, i.e., there exists a finite 
  \(U \subseteq T_1(S)\) and \(\{a_t\}_{t \in U} \subseteq \cmplM\) 
  such that \(b - \sum_{t \in U} \bigl( a_t -\nobreak t(a_t) \bigr) 
  \in \ve\). Let \(\delta \in \cmplO\) be such that \(\delta \subseteq 
  \ve\) and \(t(\delta) \subseteq \ve\) for all \(t \in U\). Let 
  \(\{c_t\}_{t \in U} \subseteq \mc{M}\) be such that \(a_t-c_t \in 
  \delta\) for all \(t \in U\). Then
  \begin{equation*}
    \sum_{t \in U} \bigl( c_t - t(c_t) \bigr) = 
    \sum_{t \in U} \bigl( a_t - t(a_t) \bigr) + 
      \sum_{t \in U} (c_t - a_t) - \sum_{t \in U} t(c_t - a_t) \in
    b + \ve + \delta - \ve = b + \ve
  \end{equation*}
  and hence $b$, by the arbitrariness of $\ve$, is in the closure of 
  $\sum_{t \in T_1(S)} M_t$.
  
  The claim that also a conditionally advanceable continuous 
  homomorphism $v$ will map $\mc{I}(S)$ into itself is now an easy 
  consequence of \eqref{Eq:AltSpecI(S)}: by continuity and since $v$ 
  is a homomorphism, it suffices to show \(v(b) \in \mc{I}(S)\) for 
  arbitrary \(b \in \Span\bigl( \bigl\{ \mu -\nobreak t(\mu) \bigr\} 
  \bigr)\), \(\mu \in \mc{Y}\), and \(t \in T_1(S)\). Let such $b$, 
  $\mu$, and $t$ be given. There exist \(\{r_i\}_{i=1}^n \subseteq 
  \pm R^*\) such that \(b = \sum_{i=1}^n r_i\bigl( \mu -\nobreak 
  t(\mu) \bigr)\) and \(\{t_i\}_{i=1}^n \in T(S)\) such that \((t_i 
  \circ\nobreak v \circ\nobreak r_i)(\mu) = (v \circ\nobreak t 
  \circ\nobreak r_i)(\mu)\). Thus 
  \begin{align*}
    v(b) = 
    v\biggl( \sum_{i=1}^n r_i\bigl( \mu - t(\mu) \bigr) \biggr) 
      ={}& 
    \sum_{i=1}^n \biggl( v\bigl( r_i(\mu) \bigr) - 
      v\Bigl( r_i\bigl( t(\mu) \bigr) \Bigr) \biggr) 
      = \\ ={}&
    \sum_{i=1}^n \biggl( v\bigl( r_i(\mu) \bigr) - 
      t_i\Bigl( v\bigl( r_i(\mu) \bigr) \Bigr) \biggr) 
      \in \mc{I}(S)\text{.}
  \end{align*}
\end{proof}

In the case of Bergman's diamond lemma, where all maps on the form 
\eqref{Eq:2side-mult.} are advanceable, this lemma implies that 
$\mc{I}(S)$ is a two-sided ideal: it is closed under addition, 
multiplication by a scalar, and multiplication on either side by an 
arbitrary generator of the algebra \(\cmplM = \RavX\), so by 
distributivity it is closed under multiplication by arbitrary elements. 
The letter `$\mc{I}$' was chosen in anticipation of this, since `I' 
is the initial of `ideal', but it is by no means restricted to 
two-sided ideals.

\begin{definition}
  Let $V$ be a set of maps \(\cmplM \Fpil \cmplM\). A nonempty 
  $R$-module \(N \subseteq \cmplM\) is said to be a 
  \DefOrd[{ideal}*]{$V$-ideal} if it is topologically closed and 
  \(v(a) \in N\) for all \(v \in V\) and \(a \in N\). A set 
  \(A \subseteq \cmplM\) is said to be a 
  \DefOrd[{ideal basis}*]{$V$-ideal basis} for $N$ if 
  \(N = \Cspan\bigl( \{ v(a) \}_{a \in A, v \in V} \bigr)\).
\end{definition}

From a minimalistic formal perspective the `$V$-ideal' concept is 
unnecessary, as it is equivalent to `topologically closed $(R 
\cup\nobreak V)$-module', but (apart from being shorter) the 
`$V$-ideal' terminology has the advantage of being closer to the 
familiar terms `left ideal', `right ideal', and `two-sided ideal' 
of which `$V$-ideal' is a common generalisation. Furthermore, it 
is `$V$-ideal basis' that is the more important concept in the above 
definition, since that is a first step on the way to defining 
a \emDefOrd{Gr\"obner basis}. Ideal bases have the right associations 
for this, whereas any combination of `module' and `basis' is likely 
to give rise to incorrect expectations about independence between 
basis elements.

Besides preparing for subsequent developments, this definition also 
gives an opportunity to summarise some of the above constructions of 
simple reductions into a formal statement that actually claims 
something, even if it isn't very much.

\begin{corollary} \label{Kor:t_v,s-reduktion}
  Let \(\mc{R} \supseteq R\) be a unital ring of continuous 
  endomorphisms on $\cmplM$ that is equipped with a topology such 
  that the $\mc{R}$-module action \(\mc{R} \times \cmplM \Fpil 
  \cmplM : (r,b) \mapsto r(b) =: r \cdot b\) is continuous and 
  $\mc{R}$ is complete. 
  Assume $\mc{M}$ is a free $\mc{R}$-module with basis $\mc{Y}$ such 
  that each coefficient-of-$\mu$ homomorphism \(f_\mu\colon 
  \mc{M} \Fpil \mc{R}\) is continuous. 
  
  If $V$ is a monoid of continuous $\mc{R}$-module homomorphisms 
  \(\cmplM \Fpil \cmplM\) that map $\mc{Y}$ into itself then the 
  following holds for all \(S \subseteq \mc{Y} \times \cmplM\):
  \begin{enumerate}
    \item
      For any \(v \in V\) and \(s = (\mu_s,a_s) \in S\), the map 
      \(t_{v,s}\colon \cmplM \Fpil \cmplM\) defined by
      \begin{equation} \label{Eq:t_v,s-reduktion}
        t_{v,s}(b) = 
        b - f_{v(\mu_s)}(b)\bigl( v(\mu_s) - v(a_s) \bigr)
        \qquad\text{for all \(b \in \cmplM\)}
      \end{equation}
      is a continuous homomorphism that commutes with all elements of 
      $\mc{R}$.
    \item
      If \(T_1(S) = \{t_{v,s}\}_{v \in V, s \in S}\) then every 
      \(v \in V\) is advanceable and the set 
      \(\{\mu_s -\nobreak a_s\}_{s \in S}\) is a 
      $V$-ideal basis for $\mc{I}(S)$.
    \item
      Any injective element of $V$ is absolutely advanceable.
  \end{enumerate}
\end{corollary}
\begin{proof}
  The definition \eqref{Eq:t_v,s-reduktion} of $t_{v,s}$ is clearly a 
  composition of maps that by assumption are continuous 
  homomorphisms. Furthermore \(t_{v,s}(r \cdot\nobreak b) = r \cdot b 
  - f_{v(\mu_s)}(r \cdot\nobreak b) \cdot v(\mu_s -\nobreak a_s) = 
  r \cdot b - \bigl( r \circ\nobreak f_{v(\mu_s)}(b) \bigr) \cdot 
  v(\mu_s -\nobreak a_s) = r \cdot t_{v,s}(b)\) for all \(r \in 
  \mc{R}\) and \(b \in \cmplM\) since $f_{v(\mu_s)}$ is an 
  $\mc{R}$-module homomorphism and $\cdot$ is a left module action. 
  Hence \(t_{v,s} \circ r = r \circ t_{v,s}\).
  
  By the last part of \eqref{Eq:AltSpecI(S)}, $\mc{I}(S)$ is spanned 
  by all $\lambda - t(\lambda)$ for \(\lambda \in \mc{Y}\) and 
  \(t \in T_1(S)\), i.e., all \(\lambda - t_{v,s}(\lambda) = \lambda 
  - \lambda + f_{v(\mu_s)}(\lambda) \cdot v(\mu_s - \nobreak a_s)\). 
  This is $0$ unless \(\lambda = v(\mu_s)\), in which case it is 
  equal to $v(\mu_s - \nobreak a_s)$. Hence \(\Cspan\bigl( \{ 
  v(\mu_s - \nobreak a_s) \}_{v \in V, s \in S} \bigr) = \mc{I}(S)\).
  
  In order to show that \(w \in V\) is advanceable, let \(t_{v,s} \in 
  T_1(S)\) and \(\lambda \in \mc{Y}\) be arbitrary. If \(\lambda = 
  v(\mu_s)\) then
  \begin{multline*}
    w\bigl( t_{v,s}(\lambda) \bigr) =
    w\Bigl( \lambda - f_{v(\mu_s)}(\lambda) \cdot v(\mu_s - a_s) 
      \Bigr) =
    w\bigl( v(\mu_s) - v(\mu_s - a_s) \bigr) = \\ =
    (w \circ v)(a_s) =
    (w \circ v)(\mu_s) - (w \circ v)(\mu_s - a_s) = \\ =
    w(\lambda) - f_{(w \circ v)(\mu_s)}\bigl( (w \circ v)(\mu_s) 
      \bigr) \cdot (w \circ v)(\mu_s - a_s) =
    t_{w \circ v,s}\bigl( w(\lambda) \bigr)
  \end{multline*}
  and since $w$, $t_{v,s}$, and $t_{w \circ v,s}$ are all 
  $\mc{R}$-module homomorphisms it follows that 
  \(w\bigl( t_{v,s}(r \cdot\nobreak \lambda) \bigr) =
  t_{w \circ v,s}\bigl( w(r \cdot\nobreak \lambda) \bigr)\) for all 
  \(r \in \mc{R}\). If instead \(\lambda \neq v(\mu_s)\) then
  \begin{equation*}
    w\bigl( t_{v,s}(\lambda) \bigr) =
    w\Bigl( \lambda - f_{v(\mu_s)}(\lambda) \cdot v(\mu_s - a_s) 
      \Bigr) =
    w(\lambda) =
    \id\bigl( w(\lambda) \bigr)
  \end{equation*}
  and since $w$, $t_{v,s}$, and $\id$ are all $\mc{R}$-module 
  homomorphisms it follows that \(w\bigl( t_{v,s}(r \cdot\nobreak 
  \lambda) \bigr) = \id\bigl( w(r \cdot\nobreak \lambda) \bigr)\) 
  for all \(r \in \mc{R}\). Either way, $w$ can be advanced past 
  $t_{v,s}$ and hence $w$ is advanceable.
  
  If $w$ is injective then \(w(\lambda) = (w \circ\nobreak v)(\mu_s)\) 
  if and only if \(\lambda = v(\mu_s)\) and hence 
  \(f_{(w \circ v)(\mu_s)} \circ w = f_{v(\mu_s)}\). In this case 
  \(w\bigl( t_{v,s}(b) \bigr) = t_{w \circ v,s}\bigl( w(b) \bigr)\) 
  for all \(b \in \cmplM\).
\end{proof}

In the case \(\mc{M} = \RavX\), the choices of $V$ that give rise to 
one-sided ideals are
\begin{align*}
  V ={}& \{ b \mapsto\nobreak \nu b \}_{\nu\in X^*}&&
    \text{(left ideal),}\\
  V ={}& \{ b \mapsto\nobreak b\nu \}_{\nu\in X^*}&&
    \text{(right ideal);}
\end{align*}
multiplication by a non-monomial element of $\mc{M}$ does not map 
\(\mc{Y} = X^*\) into itself and can therefore not be used with 
Corollary~\ref{Kor:t_v,s-reduktion}. In the fourth classical case 
that $\mc{M}$ is an $\mc{R}[X]$-module and one wants $\mc{I}(S)$ to 
be an $\mc{R}[X]$-submodule, the right choice is to make $V$ the set 
of all maps \(b \mapsto b \prod_{x \in X} x^{n_x}\) for \(\{n_x\}_{x 
\in X} \subset \N\), i.e., the set of maps that multiply by power 
products on $X$.

\section{The limit of all reductions}
\label{Sec:Reducibilitet}

The head-on approach for defining the sought projection $\tS$ of 
$\cmplM$ onto $\Irr(S)$ would be to immediately seek a map \(\cmplM 
\Fpil \Irr(S)\), but often a subtler approach is more convenient. 
The route taken here is to (i)~define subsets of $\cmplM$ where 
the collective of reductions has nice properties, (ii)~use those 
properties in the definition of $\tS$, and only afterwards (iii)~show 
that the subsets defined in~(i) are in fact the whole of $\cmplM$; the 
same approach was followed in~\cite{Bergman}. Steps~(i) and~(ii) are 
carried out in this section, whereas step~(iii) is the subject of the 
next.

The definition of $\tS$ as a pointwise limit of $T(S)$ would be to say 
that $\tS(a)$ is the element of $\Irr(S)$ that is a limit point of 
$\setOf[\big]{t(a)}{t \in T(S)}$; in a discrete topology, this simply 
says that $\tS(a)$ is the element of $\Irr(S)$ which is equal to $t(a)$ 
for some \(t \in T(S)\), but in general one will have to make do with 
being able to get arbitrarily close to some irreducible element. For 
this idea to work as a definition of $\tS(a)$ it is of course first 
necessary that such a limit point exists (and second necessary that it 
is unique), but the existence condition becomes much more 
convenient if one strengthens it a bit.

\begin{definition} \label{Def:PerRed}
  An \(a \in \cmplM\) is said to be \DefOrd{stuck in} 
  \(N \subseteq \cmplM\) under $T(S)$ if \(t(a) \in N\) for all 
  \(t \in T(S)\). 
  Given an \(\ve \in \cmplO\), an \(a \in \cmplM\) is said to 
  be \DefOrd{persistently $\ve$-reducible} under $T(S)$ if there 
  for every \(t_1 \in T(S)\) exists some \(t_2 \in T(S)\) and 
  \(b \in \Irr(S)\) such that \(t_2\bigl(t_1(a)\bigr)\) is stuck in 
  $b+\ve$.
  
  If $a$ is persistently $\ve$-reducible for all \(\ve \in \cmplO\) 
  then $a$ is said to be \DefOrd{persistently reducible}. The set of 
  all elements in $\cmplM$ that are persistently reducible under $T(S)$ 
  is denoted \index{Per@$\Per$}$\Per(S)$ and the set of all elements 
  in $\cmplM$ that are persistently $\ve$-reducible under $T(S)$ is 
  denoted $\Per_\ve(S)$.
\end{definition}


The extra power added in this definition is that the wanted outcome of 
being close to $\Irr(S)$ should \emph{persist} no matter what has to be 
done before or after the $t_2$ reduction chosen to get there: \(a \in 
\Per_\ve(S)\) if and only if there for every \(t_1 \in T(S)\) exists 
some \(t_2 \in T(S)\) and \(b \in \Irr(S)\) such that it holds for every 
\(t_3 \in T(S)\) that \((t_3 \circ\nobreak t_2 \circ\nobreak t_1)(a) - b 
\in \ve\). The corresponding property in~\cite{Bergman} is 
`reduction-finiteness'; see \cite[Ssec.~3.1.2]{Avhandlingen} for a 
comparison and analysis of the two.


\begin{lemma} \label{L:PerTillIrr}
  For each \(a \in \Per(S)\), there exists some \(b \in \Irr(S)\) such 
  that for every \(\ve\in\cmplO\) there is some \(t \in T(S)\) such 
  that $t(a)$ is stuck in $b+\ve$. 
  Furthermore $a$ and $b$ are such that \(a-b \in \mc{I}(S)\), and 
  hence \(\Per(S) \subseteq \mc{I}(S) + \Irr(S)\).
\end{lemma}
\begin{proof}
  Construct a sequence \(\{u_i\}_{i=0}^\infty \subseteq T(S)\) by 
  letting \(u_0 = \id\) and recursively defining $u_n$ for 
  \(n>0\) as follows: by persistent reducibility of $a$ there exist 
  \(t_n \in T(S)\) and \(b_n \in \Irr(S)\) such that $t_n\bigl( 
  u_{n-1}(a) \bigr)$ is stuck in $b_n+\cmpl{B_n}$, therefore let 
  \(u_n = t_n \circ u_{n-1}\).
  
  For any \(n,i,j \in \N\) such that \(i>j\geqslant n\), the element 
  $u_i(a)$ is in $b_i+\cmpl{B_i}$ as well as in $b_j+\cmpl{B_j}$, 
  which means
  $$
    b_i-b_j = 
    \bigl( b_i - u_i(a) \bigr) + \bigl( u_i(a) - b_j \bigr) \in
    \cmpl{B_i} + \cmpl{B_j} \subseteq \cmpl{B_n}
    \text{.}
  $$
  Hence \(\{b_n\}_{n=0}^\infty\) is a Cauchy sequence in $\Irr(S)$ and 
  consequently it converges to some \(b \in \Irr(S)\). 
  Fix some \(n \in \N\) such that \(\cmpl{B_n} \subseteq \ve\) and 
  consider \(\lim_{i \Lpil \infty} (b_i \mbin{-} b_n)\). Since all 
  elements in this sequence are in the closed set \(\cmpl{B_n}\) it 
  follows that the limit \(b - b_n \in \cmpl{B_n} \subseteq \ve\). 
  Thus \(b_n + \ve = b + \ve\) and $u_n(a)$ is stuck in \(b+\ve\) as 
  claimed.
  
  To see the second claim, observe that \(u_n(a) \Lpil b\) as 
  \(n \Lpil \infty\). For any \(n \in \N\), \(a-u_n(a) \in \mc{I}(S)\) 
  by definition and consequently \(a-b = \lim_{n \Lpil \infty} 
  \bigl(a \mbin{-} u_n(a) \bigr) \in \mc{I}(S)\) as well. Hence 
  \(a \in \mc{I}(S) + \Irr(S)\).
\end{proof}

This proof highlights a subtle point in the basic set-up of 
Section~\ref{Sec:Basics} which may be regarded as a restriction, 
namely that the family \(\mc{O} = \{B_n\}_{n=1}^\infty\) must be 
countable: the construction of $\{u_n\}_{n=0}^\infty$ would not 
necessarily suffice for demonstrating convergence if $\mc{O}$ was 
uncountable. In other proofs it is possible to treat $\cmplO$ as an 
arbitrary collection of neighbourhoods, which might suggest a 
generalisation to topologies defined by an uncountable $\mc{O}$ is 
not unreasonable, but on the other hand the countability of $\mc{O}$ 
is in this lemma closely tied to the status of $T(S)$ as a set of 
finite compositions of elements of $T_1(S)$, and \emph{that} is 
something several proofs rely on. Hence removing the condition that 
$\mc{O}$ is countable would probably require a more powerful 
construction of reductions than as mere compositions of simple 
reductions; it is presently not something for which I see any 
precedence.

Besides this existence of a limit point property, it is also important 
that the set of persistently reducible elements is closed under 
algebraic operations. This exercises the slightly different aspect of 
persistent reducibility that one can find reductions which 
simultaneously take several persistently reducible elements close to 
their normal forms.

\begin{lemma} \label{L:Per modul}
  The set $\Per(S)$ and for every \(\ve \in \cmplO\) the set 
  $\Per_\ve(S)$ are $R$-modules. More generally, $\Per(S)$ is mapped 
  into itself by every continuous group homomorphism \(\cmplM \Fpil 
  \cmplM\) which commutes with all reductions.
\end{lemma}
\begin{proof}
  Let \(\ve \in \cmplO\) be given. 
  Clearly \(\Irr(S) \subseteq \Per_\ve(S)\), and hence to see that 
  the latter is a group, it suffices to show for two arbitrary 
  elements in it that their difference is also in this set. 
  Therefore let \(a_1,a_2 \in \Per_\ve(S)\) and \(t_1 \in T(S)\) be 
  arbitrary. There exists some \(t_2 \in T(S)\) and \(b_1 \in \Irr(S)\) 
  such that \(t_2\bigl(t_1(a_1)\bigr)\) is stuck in $b_1+\ve$. There 
  also exists some \(t_3 \in T(S)\) and \(b_2 \in \Irr(S)\) such that 
  \(t_3\bigl( (t_2\mbin{\circ}t_1)(a_2) \bigr)\) is stuck in $b_2+\ve$. 
  Now let \(t_4 \in T(S)\) be arbitrary. Since
  \begin{multline*}
    t_4 \Bigl( (t_3 \circ t_2)\bigl( t_1(a_1-a_2) \bigr) \Bigr) - 
      (b_1-b_2) 
      = \\ =
    \biggl(
      (t_4 \circ t_3) \Bigl( t_2 \bigl( t_1(a_1) \bigr) \Bigr) - b_1
    \biggr) - \biggl(
      t_4 \Bigl( t_3 \bigl( (t_2 \circ t_1)(a_2) \bigr) \Bigr) - b_2
    \biggr) \in \ve + \ve = \ve \text{,}
  \end{multline*}
  it follows that \(t_4 \Bigl( (t_3 \mbin{\circ} t_2) \bigl( 
  t_1(a_1\mbin{-}a_2) \bigr) \Bigr) \in (b_1\mbin{-}b_2)+\ve\). Hence 
  the element \((t_3 \mbin{\circ} t_2)\bigl( t_1(a_1\mbin{-}a_2) \bigr)\) 
  is stuck in \((b_1\mbin{-}b_2)+\ve\), and by arbitrariness of $t_1$ 
  it follows that \(a_1 - a_2 \in \Per_\ve(S)\).
  
  Now let \(r\colon \cmplM \Fpil \cmplM\) be an arbitrary continuous 
  group homomorphism which satisfies \(r \circ t = t \circ r\) for all 
  \(t \in T(S)\), and let \(\delta \in \cmplO\) be such that 
  \(r(\delta) \subseteq \ve\). Let \(a \in \Per_\delta(S)\) and 
  \(t_1 \in T(S)\) be arbitrary. Let \(t_2 \in T(S)\) and \(b \in 
  \Irr(S)\) be such that $(t_2 \circ\nobreak t_1)(a)$ is stuck in 
  $b+\ve$. Then for any \(t_3 \in T(S)\),
  \begin{multline*}
    (t_3 \circ t_2 \circ t_1)\bigl( r(a) \bigr) - r(b) =
    r\bigl( (t_3 \circ t_2 \circ t_1)(a) \bigr) - r(b) = \\ =
    r\bigl( (t_3 \circ t_2 \circ t_1)(a) - b \bigr) \in
    r(\delta) \subseteq \ve
  \end{multline*}
  and thus $(t_2 \circ\nobreak t_1)\bigl( r(a) \bigr)$ is stuck in 
  \(r(b)+\ve\). It follows that $r(a)$ is persistently 
  $\ve$-reducible.
  
  In the case that \(r \in R\), one knows from the fact that $\ve$ 
  is an $R$-module that one can take \(\delta=\ve\), and then the 
  above has shown that $\Per_\ve(S)$ is an $R$-module. For more 
  general $r$ this need not be the case, and then one has only shown 
  about $r$ that there for every \(\ve \in \cmplO\) is some \(\delta 
  \in \cmplO\) such that $r$ maps $\Per_\delta(S)$ into 
  $\Per_\ve(S)$. Suppose now that \(a \in \Per(S)\) is arbitrary, and 
  consider the question of whether \(r(a) \in \Per(S)\). Let \(\ve \in 
  \cmplO\) be arbitrary and let \(\delta \in \cmplO\) be such that 
  \(r\bigl( \Per_\delta(S) \bigr) \subseteq \Per_\ve(S)\). Since 
  \(a \in \Per(S) \subseteq \Per_\delta(S)\), it follows that \(r(a) 
  \in \Per_\ve(S)\), and hence \(r(a) \in \bigcap_{\ve \in \cmplO} 
  \Per_\ve(S) =  \Per(S)\) by the arbitrariness of $\ve$.
\end{proof}

\begin{definition}
  Let \(\ve \in \cmplO\) be arbitrary. An \(a \in \cmplM\) is said to 
  be \DefOrd[{uniquely reducible}*]{$\ve$-uniquely reducible} under 
  $T(S)$ if, for any \(t_1,t_2 \in T(S)\) and \(b_1,b_2 \in \Irr(S)\) 
  such that $t_1(a)$ is stuck in $b_1+\ve$ and $t_2(a)$ is stuck in 
  $b_2+\ve$, it holds that \(b_1+\ve = b_2+\ve\). The set of all 
  elements in $\cmplM$ which are both persistently and 
  $\ve$-uniquely reducible under $T(S)$ is denoted 
  $\Red_\ve(S)$\index{Red epsilon S@$\Red_\ve(S)$}.
\end{definition}


\begin{lemma} \label{L: Red_ve modul}
  For every \(\ve \in \cmplO\), the set $\Red_\ve(S)$ is an 
  $R$-module that furthermore is mapped into itself by every 
  reduction.
\end{lemma}
\begin{proof}
  Let \(\ve \in \cmplO\) be given. Let \(a_1,a_2 \in \Red_\ve(S)\) be 
  arbitrary. 
  It follows from Lemma~\ref{L:Per modul} that \(a_1 - a_2 \in \Per(S)\) 
  and hence there exists \(t_1 \in T(S)\) and \(b \in \Irr(S)\) such that 
  \(t_1(a_1\mbin{-}a_2)\) is stuck in $b+\ve$. Since \(a_1,a_2 \in 
  \Per(S)\) there furthermore exist \(t_2,t_3 \in T(S)\) and \(b_1,b_2 
  \in \Irr(S)\) such that \(t_2\bigl(t_1(a_1)\bigr)\) is stuck in 
  $b_1+\ve$ and \(t_3\bigl( (t_2\mbin{\circ}t_1)(a_2) \bigr)\) is stuck 
  in $b_2+\ve$. This implies that, for \(t = t_3 \circ t_2 \circ 
  t_1\),
  \[
    b - (b_1-b_2) =
    b - t(a_1-a_2) - \bigl( b_1 - t(a_1) \bigr) + 
      \bigr( b_2 - t(a_2) \bigr) \in
    \ve - \ve + \ve = \ve
    \text{.}
  \]
  Starting from some other \(t_1' \in T(S)\) and \(b' \in \Irr(S)\) 
  such that \(t_1'(a_1\mbin{-}a_2)\) is stuck in $b'+\ve$, one 
  similarly gets the existence of \(t_2',t_3' \in T(S)\) and \(b_1',b_2' 
  \in \Irr(S)\) such that \(t_2'\bigl(t_1'(a_1)\bigr)\) is stuck in 
  $b_1'+\ve$ and \(t_3'\bigl( (t_2'\mbin{\circ}t_1')(a_2) \bigr)\) is 
  stuck in $b_2'+\ve$; in precisely the same way one furthermore 
  shows that \(b' - (b_1'-b_2') \in \ve\). The $\ve$-unique 
  reducibility of $a_1$ and $a_2$ does however imply that 
  \(b_1-b_1' \in \ve\) and \(b_2-b_2' \in \ve\). From this 
  follows that \(b-b' \in \ve\) and hence \(a_1 - a_2 \in 
  \Red_\ve(S)\). This has shown that $\Red_\ve(S)$ is a group.
  
  It must also be shown that elements of $R$ map $\Red_\ve(S)$ into 
  itself. Let \(r \in R\), \(\ve \in \cmplO\), and \(a \in \Red_\ve(S)\) 
  be given. Let \(t_1,t_2 \in T(S)\) and \(b_1,b_2 \in \Irr(S)\) be 
  arbitrary such that $t_i\bigl( r(a) \bigr)$ is stuck in $b_i+\ve$ 
  for \(i=1,2\). Since \(a \in \Per(S)\), there exist \(t_1',t_2' \in 
  T(S)\) and \(b_1',b_2' \in \Irr(S)\) such that $t_i'\bigl( t_i(a) 
  \bigr)$ is stuck in $b_i'+\ve$ for \(i=1,2\). It follows that
  \begin{equation*}
    b_i - r(b_i') = 
    \bigl( b_i - (t_i' \circ t_i \circ r)(a) \bigr) +
      r\bigl( (t_i' \circ t_i)(a) - b_i' \bigr) \in
    \ve + r(\ve) = \ve
  \end{equation*}
  for \(i=1,2\). Hence
  \[
    b_1-b_2 =
    b_1 - r(b_1') + r(b_1'-b_2') + r(b_2') - b_2 \in
    \ve + r(\ve) + \ve = \ve
  \]
  and thus $r(a)$ is $\ve$-uniquely reducible by the arbitrariness of 
  $t_1$ and $t_2$. By Lemma~\ref{L:Per modul}, $r(a)$ is also 
  persistently reducible, and so \(r(a) \in \Red_\ve(S)\).
  
  The corresponding property for reductions is more trivial. 
  $\ve$-unique reducibility of $t(a)$ for \(a \in \Red_\ve(S)\) and 
  \(t \in T(S)\) is the claim that any \(t_1,t_2 \in T(S)\) and 
  \(b_1,b_2 \in \Irr(S)\) such that $t_1\bigl( t(a) \bigr)$ is stuck 
  in $b_1+\ve$ and $t_2\bigl( t(a) \bigr)$ is stuck in $b_2+\ve$ 
  satisfy \(b_1+\ve = b_2+\ve\), but that is just a special case of 
  the $\ve$-unique reducibility of $a$.
\end{proof}

\begin{definition} \label{Def:Red(S)}
  An element in $\cmplM$ is said to be \DefOrd{uniquely reducible} 
  if it is $\ve$-uniquely reducible for all \(\ve \in \cmplO\). The 
  set of those element which are both persistently and uniquely 
  reducible under $T(S)$ is denoted \index{Red@$\Red(S)$}$\Red(S)$. 
  Define the map \(\tS\colon \Red(S) \Fpil \Irr(S)\) by letting 
  $\tS(a)$ be the unique element of $\Irr(S)$ with the property that 
  there for every \(\ve \in \cmplO\) exists some \(t \in T(S)\) such 
  that $t(a)$ is stuck in $\tS(a)+\ve$.
\end{definition}

The chain of sets defined in this section is thus that \(\Red(S) 
\subseteq \Red_\ve(S) \subseteq \Per(S) \subseteq \Per_\ve(S) 
\subseteq \cmplM\), with \(\Red(S) = \bigcap_{\ve\in\cmplO} 
\Red_\ve(S)\) and \(\Per(S) = \bigcap_{\ve\in\cmplO} \Per_\ve(S)\). 
The end one wants to see is that all of these are equal, and in 
Lemmas~\ref{L: Persistently red.} and~\ref{L:Red=cmplM} this is taken 
care of by giving sufficient conditions for \(\Per_\ve(S)=\cmplM\) 
and \(\Red_\ve(S)=\cmplM\) respectively. A more immediate goal is 
however to establish that circumstances inside $\Red(S)$ are good.

\begin{lemma} \label{L: Red modul}
  The set $\Red(S)$ is an $R$-module that is mapped into itself by 
  every reduction. 
  The map \(\tS\colon \Red(S) \Fpil \Irr(S)\) is well-defined and a 
  group homomorphism. More generally, every continuous group 
  homomorphism \(r\colon \cmplM \Fpil \cmplM\) which commutes with 
  all reductions maps $\Red(S)$ into itself and commutes with $\tS$. 
  In addition, \(\tS(b) = b\) for all \(b \in \Irr(S)\), \(\ker \tS 
  \subseteq \mc{I}(S)\), and \(\tS\bigl( t(a) \bigr) = \tS(a)\) for 
  all \(a \in \Red(S)\) and \(t \in T(S)\).
\end{lemma}
\begin{proof}
  Since \(\Red(S) = \bigcap_{\ve\in\cmplO} \Red_\ve(S)\) is an 
  intersection of sets which by Lemma~\ref{L: Red_ve modul} are 
  $R$-modules that are mapped into themselves by every reduction, it 
  follows that $\Red(S)$ shares these properties.
  
  Next consider $\tS$. It was shown in Lemma~\ref{L:Per modul} that 
  there for every \(a \in \Per(S)\) exists some \(b \in \Irr(S)\) 
  which is a candidate for being $\tS(a)$, but what about uniqueness? 
  One may observe that if \(a \in \Red(S)\) and \(b_1,b_2 \in \Irr(S)\) 
  are such that there for every \(\ve \in \cmplO\) exist \(t_1,t_2 \in 
  T(S)\) such that $t_1(a)$ is stuck in $b_1+\ve$ and $t_2(a)$ is stuck 
  in $b_2+\ve$, then by $\ve$-unique reducibility \(b_1-b_2 \in \ve\) 
  for every \(\ve \in \cmplO\). Hence \(b_1=b_2\) as 
  \(\bigcap_{\ve\in\cmplO} \ve = \{0\}\) and thus $\tS$ is 
  well-defined. The same argument for $(t_1 \circ\nobreak t)(a)$ and 
  $t_2(a)$ being stuck in $b_1+\ve$ and $b_2+\ve$ respectively 
  demonstrates that \((\tS \circ\nobreak t)(a) = \tS(a)\) for all \(t 
  \in T(S)\) and \(a \in \Red(S)\). Since an irreducible element $b$ 
  is always stuck in every neighbourhood of itself, it follows that 
  \(\tS(b)=b\) for all \(b \in \Irr(S)\).
  
  Now let \(a_1,a_2 \in \Red(S)\) be given and consider the matter of 
  whether \(\tS(a_1 +\nobreak a_2) = \tS(a_1) + \tS(a_2)\). Let \(\ve 
  \in \cmplO\) be arbitrary. By definition of $\tS$ there exists some 
  \(t_1 \in T(S)\) such that $t_1(a_1)$ is stuck in $\tS(a_1) + \ve$. 
  Since \(a_2 \in \Per(S)\) there exists some \(t_2 \in T(S)\) and 
  \(b_2 \in \Irr(S)\) such that $(t_2 \circ\nobreak t_1)(a_2)$ is stuck 
  in $b_2+\ve$, and by unique reducibility of $a_2$ it follows that 
  \(b_2+\ve = \tS(a_2)+\ve\). Hence $(t_2 \circ\nobreak t_1)(a_1 
  +\nobreak a_2)$ is stuck in $\tS(a_1)+\tS(a_2)+\ve$, and then by 
  unique reducibility of $a_1+a_2$ it follows that \(\tS(a_1 
  +\nobreak a_2) + \ve = \tS(a_1)+\tS(a_2)+\ve\). Thus \(\tS(a_1 
  +\nobreak a_2) = \tS(a_1) + \tS(a_2)\) by the arbitrariness of 
  $\ve$.
  
  Next consider the matter of whether a continuous group homomorphism 
  \(r\colon \cmplM \Fpil \cmplM\) that commutes with all reductions 
  will map $\Red(S)$ into itself and commute with $\tS$. Let \(a \in 
  \Red(S)\) be given and \(\ve\in\cmplO\) be arbitrary. Let \(b \in 
  \Irr(S)\) and \(t_1 \in T(S)\) such that \(t_1\bigl( r(a) \bigr)\) 
  is stuck in $b+\ve$ be arbitrary. By the continuity of $r$ there 
  exists some \(\delta\in\cmplO\) such that \(r(\delta) \subseteq 
  \ve\); let \(t_2 \in T(S)\) be such that $t_2\bigl( t_1(a) \bigr)$ 
  is stuck in $\tS(a)+\delta$. Since
  \begin{equation*}
    b - r\bigl( \tS(a) \bigr) =
    b - (t_2 \circ t_1 \circ r)(a) + 
      r\bigl( (t_2 \circ t_1)(a) - \tS(a) \bigr) \in
    \ve + r(\delta) = \ve
  \end{equation*}
  it follows that $r(a)$ is $\ve$-uniquely reducible, and by the 
  arbitrariness of $\ve$ that it is uniquely reducible. It is 
  furthermore persistently reducible by Lemma~\ref{L:Per modul}, and 
  hence an element of $\Red(S)$. Finally \(\tS\bigl( r(a) \bigr) = 
  r\bigl( \tS(a) \bigr)\) since it was in neighbourhoods of 
  $r\bigl( \tS(a) \bigr)$ that images of $r(a)$ could get stuck.
  
  Last, it should be verified that \(\ker\tS \subseteq \mc{I}(S)\). 
  Let \(a \in \ker\tS\) be given. Let \(\ve \in \cmplO\) be 
  arbitrary. There exists some \(t \in T(S)\) such that $t(a)$ 
  is stuck in \(\tS(a)+\ve = \ve\), and hence \(a - t(a) \in a - 
  \ve\) on one hand and \(a - t(a) \in \setOf[\big]{ b-t(b) }{ b \in 
  \cmplM} \subseteq \mc{I}(S)\) on the other, i.e., $a$ is a limit 
  point of $\mc{I}(S)$. Since $\mc{I}(S)$ is topologically closed 
  by definition, \(a \in \mc{I}(S)\).
\end{proof}

A third property that $\Per_\ve(S)$ and $\Red_\ve(S)$ should possess 
is to be topologically closed, but this does \emph{not} happen 
automatically. A minimal example of a situation where $\Per_\ve(S)$ is 
not topologically closed can be constructed on the formal power series 
foundation \(\mc{M} = \Z[\ssa]\), \(\cmplM = \Z [\![ \ssa ]\!]\), and 
\(\cmpl{B_n} = \ssa^n \cmplM\), if one picks as \(T_1(S) = 
\{t_n\}_{n=1}^\infty\) where
\begin{equation*}
  t_n(\ssa^m) = \begin{cases}
    \ssa^{m-1}& \text{if \(m=n\),}\\
    \ssa^m& \text{otherwise}
  \end{cases}
\end{equation*}
for all \(n \geqslant 1\) and \(m \geqslant 0\). The problematic 
trait of this set of reductions is that \((t_1 \circ\nobreak \dotsc 
\circ\nobreak t_n)(\ssa^n) = \ssI \notin \cmpl{B_1}\) for any \(n 
\geqslant 1\), although the initial \(\ssa^n \in \cmpl{B_n}\). This 
means no proper series \(a \in \cmplM \setminus \mc{M}\) is ever 
stuck in any set of the form \(b + \ve\), and consequently no such 
element can ever be persistently reducible. It follows 
that in this case \(\Red(S) = \Per(S) = \mc{M}\), which is rather 
small compared to the closure $\cmplM$.

On a conceptual level, what breaks down in this example is the 
principle that series truncation produces a useful approximation. 
Truncation works for a fixed 
reduction \(t \in T(S)\)\Ldash in order to determine $t(a)$ up to 
a certain number $n$ of terms (i.e., in order to identify $t(a) + 
\cmpl{B_n}$) it is sufficient to determine $t(b)$ where $b$ is 
truncated to some number $m$ of terms; this is the claim that $t$ is 
continuous\Dash but it may fail when the reduction is not fixed. What 
one would want is therefore a bound $m$ on the number of terms that 
must be taken into account that works for all reductions, and as it 
happens the property that such a bound exists has a name that is well 
known in analysis.

\begin{definition} \label{Def:Likgradig-kont.}
  A set $F$ of group homomorphisms \(\cmplM \Fpil \cmplM\) is said to 
  be \DefOrd{equicontinuous} if there for every \(\ve \in \cmplO\) 
  exists some \(\delta \in \cmplO\) such that for all \(f \in F\) it 
  holds that \(f(\delta) \subseteq \ve\).
\end{definition}

Returning to the example, one may observe that \(T_1(S) = 
\{t_n\}_{n=1}^\infty\) actually is equicontinuous (for \(\ve = 
\cmpl{B_n}\) take \(\delta = \cmpl{B_{n+1}}\)), but what matters is 
that $T(S)$ is not: no matter how small an $\ssa^m$ may be, there is 
always a composition of simple reductions that magnifies it to 
something outside every \(\ve\in\cmplO\). On the contrapositive side, 
when $T(S)$ is equicontinuous then all the sets defined in this 
section become topologically closed.

\begin{lemma} \label{L:SlutenhetPer}
  If $T(S)$ is equicontinuous then \(\tS\colon \Red(S) \Fpil \Irr(S)\) 
  is continuous and furthermore the sets $\Per(S)$, $\Per_\ve(S)$, 
  $\Red(S)$, and $\Red_\ve(S)$ are, for all $\ve\in\cmplO$, 
  topologically closed in $\cmplM$.
\end{lemma}
\begin{proof}
  Since $\tS$ by Lemma~\ref{L: Red modul} is a group homomorphism, it 
  suffices to show that it is continuous at $0$. Let \(\ve\in\cmplO\) 
  be arbitrary. Let \(\delta \in \cmplO\) be such that \(t(\delta) 
  \subseteq \ve\) for all \(t \in T(S)\); in other words every \(a 
  \in \delta\) is stuck in $\ve$. Any \(a \in \Red(S) \cap \delta\) 
  thus satisfies \(\tS(a) \in \ve\), and hence $\tS$ is continuous at 
  $0$.
  
  Now let \(\ve\in\cmplO\) be given, and let \(\delta \in \cmplO\) 
  be such that \(t(\delta) \subseteq \ve\) for all \(t \in T(S)\).
  Let \(a \in \cmplM\) be an arbitrary limit point of $\Per_\ve(S)$, 
  and let \(c \in \Per_\ve(S)\) be such that \(c-a \in \delta\). 
  Let \(t_1 \in T(S)\) be arbitrary. By persistent reducibility of 
  $c$ there exists some \(t_2 \in T(S)\) and \(b \in \Irr(S)\) such 
  that \((t_2 \circ\nobreak t_1)(c)\) is stuck in \(b+\ve\), i.e., 
  \((t_3 \circ\nobreak t_2 \circ\nobreak t_1)(c) \in b+\ve\) for all 
  \(t_3 \in T(S)\). By equicontinuity
  \(
    (t_3 \circ\nobreak t_2 \circ\nobreak t_1)(c) - 
    (t_3 \circ\nobreak t_2 \circ\nobreak t_1)(a) 
    \in \ve
  \)
  and hence \((t_3 \circ\nobreak t_2 \circ\nobreak t_1)(a) \in 
  b+\ve\) as well, which means \((t_2 \circ\nobreak t_1)(a)\) is 
  stuck in \(b+\ve\). By the arbitrariness of $t_1$ it follows 
  that \(a \in \Per_\ve(S)\), and hence that set must be topologically 
  closed. $\Per(S)$ is thus known to be the intersection of a family of 
  topologically closed sets, which implies that it too is closed.
  
  To show that $\Red_\ve(S)$ is topologically closed, let \(a \in 
  \Per(S)\) be a limit point of $\Red_\ve(S)$. Let \(t_1,t_2 \in T(S)\) 
  and \(b_1,b_2 \in \Irr(S)\) be arbitrary such that \(t_i(a)\) is 
  stuck in $b_i+\ve$ for \(i=1,2\). Let \(c \in \Red_\ve(S)\) be 
  such that \(c-a \in \delta\). Since \(t(c) - t(a) \in \ve\) for all 
  \(t \in T(S)\), it follows that $t_i(c)$ is also stuck in $b_i+\ve$ 
  for \(i=1,2\), and hence \(b_1-b_2 \in \ve\) by the $\ve$-unique 
  reducibility of $c$, whence $a$ is $\ve$-uniquely reducible. All 
  limit points of $\Red_\ve(S)$ are in $\Per(S)$ and thus 
  $\Red_\ve(S)$ is topologically closed. $\Red(S)$ is similarly now 
  known to be the intersection of a family of topologically closed 
  sets, which implies that it is closed as well.
\end{proof}

From an analytical perspective, the effect of equicontinuity of 
$T(S)$ is rather drastic\Ldash $\Irr(S)$ becomes sticky, in the sense 
that any \(a \in \Irr(S) + \delta\) is stuck in the corresponding 
$\Irr(S) + \ve$ (and even in $a+\ve$)\Dash so in view of the 
ruggedness of the proof of Lemma~\ref{L:SlutenhetPer}, one might 
wonder whether equicontinuity really is The Right Condition for 
reaching the end that $\Red(S)$ is closed, but the jury is still out 
on that one. Looking at the proofs certainly suggests that it should 
be possible to make do with something weaker, but concrete 
applications rather tend to end up satisfying the stronger condition 
that \(t(\ve) \subseteq \ve\) for all \(\ve \in \cmplO\) and \(t \in 
T(S)\). Right now, the best reason for using equicontinuity is 
probably that it is well established and fully general; many other 
conditions which at first may seem to give finer control or be easier 
to verify are only defined with respect to some additional structure, 
such as a metric.

The next lemma is the first step towards the Diamond Lemma. In 
natural language, the first of the two equivalent claims is that all 
elements of $\cmplM$ are uniquely reducible, whereas the second claim 
is that every element has a unique normal form.

\begin{lemma} \label{L:(b)<=>(c)}
  If $T(S)$ is equicontinuous and all elements of $\cmplM$ are 
  persistently reducible then the following claims are equivalent:
  \begin{itemize}
    \item \(\Red(S) = \cmplM\).
    \item \(\cmplM = \Irr(S) \oplus \mc{I}(S)\).
  \end{itemize}
\end{lemma}
\begin{proof}
  First assume \(\Red(S) = \cmplM\). By Lemma~\ref{L: Red modul}, 
  $\tS$ is a projection of $\cmplM$ onto $\Irr(S)$, hence \(\cmplM = 
  \Irr(S) \oplus \ker\tS\). By the same lemma, \(\ker\tS \subseteq 
  \mc{I}(S)\). Hence \(\cmplM = \Irr(S) \oplus \mc{I}(S)\) will 
  follow if it can be shown that \(\mc{I}(S) \subseteq \ker\tS\). For 
  any \(t \in T(S)\) and \(a \in \cmplM\) it follows from 
  Lemma~\ref{L: Red modul} that \(\tS\bigl(a -\nobreak t(a) \bigr) = 
  0\) and hence \(\setOf[\big]{ a - t(a) }{ a \in \cmplM } \subseteq 
  \ker\tS\) for any \(t \in T(S)\). Since $\ker\tS$ is closed under 
  addition,
  \[
    \ker\tS \supseteq 
    \sum_{t \in T(S)} \setOf[\big]{ a - t(a) }{ a \in \cmplM }
    \text{,}
  \]
  and since it by Lemma~\ref{L:SlutenhetPer} also is topologically 
  closed, the wanted \(\mc{I}(S) \subseteq \ker\tS\) has been 
  established. This has proved one half of the equivalence.
  
  For the other half, assume \(\cmplM = \Irr(S) 
  \oplus \mc{I}(S)\). Let \(a \in \cmplM\) and \(\ve \in \cmplO\) be 
  arbitrary; it will be shown that $a$ is $\ve$-uniquely reducible. 
  Let \(t_1,t_2 \in T(S)\) and \(b_1,b_2 \in \Irr(S)\) such that 
  $t_i(a)$ is stuck in $b_i+\ve$ for \(i=1,2\) be arbitrary. By 
  Lemma~\ref{L:PerTillIrr} and the persistent reducibility of $t_1(a)$ 
  and $t_2(a)$, there exist \(c_1,c_2 \in \Irr(S)\) such that there 
  for every \(\delta \in \cmplO\) exist \(u_1,u_2 \in T(S)\) such 
  that \(u_i\bigl( t_i(a) \bigr)\) is stuck in $c_i+\delta$ for 
  \(i=1,2\). Clearly \(u_1\bigl( t_1(a) \bigr) - u_2\bigl( t_2(a) 
  \bigr) \in \mc{I}(S)\), and since such $u_1$ and $u_2$ exist for 
  all \(\delta \in \cmplO\) it follows that the limit \(c_1-c_2 \in 
  \mc{I}(S)\) as well, but since also \(c_1-c_2 \in \Irr(S)\) and 
  \(\mc{I}(S) \cap \Irr(S) = \{0\}\) it just so happens that 
  \(c_1=c_2\). For either \(i=1,2\) this common value is a limit 
  point of a sequence of elements that are stuck in the topologically 
  closed set \(b_i+\ve\), and hence \(c_1 - b_i \in \ve\). Therefore 
  \(b_1-b_2 \in \ve\) and the two neighbourhoods are the same. By the 
  arbitrariness of $b_1$ and $b_2$, the element $a$ is $\ve$-uniquely 
  reducible.
\end{proof}

The final lemma in this section explores a slightly different aspect of 
the machinery: how the sets change if some simple reductions are 
removed. It sometimes happens when one is preparing a presentation of 
an argument involving the diamond lemma that some of the reductions 
turn out to be redundant, but not all ways of verifying this redundancy 
are as easy as they may seem. The key condition for establishing 
equivalence of $T_1(S)$ to $T_1(S')$\Ldash i.e., that all things 
constructed from the set of simple reductions are the same when the 
set of simple reductions is $T_1(S')$ as when it is $T_1(S)$\Dash is 
that their respective sets $\Irr(S)$ and $\Irr(S')$ of irreducible 
elements are the same; see Theorem~\ref{S:Konstr.Irr(S)} for a method 
of characterising the irreducible elements.

\begin{lemma} \label{L:Utgl.T(S)}
  If \(T_1(S') \subseteq T_1(S)\) are such that \(\Irr(S') = 
  \Irr(S)\), \(\Per(S') = \cmplM\), \(\Red(S) = \cmplM\), and $T(S)$ 
  is equicontinuous, then \(\Red(S') = \cmplM\) as well and \(t^{S'} 
  = \tS\).
\end{lemma}
\begin{proof}
  Let \(a \in \cmplM\) and \(\ve \in \cmplO\) be arbitrary. What 
  needs to be shown is that $a$ is $\ve$-uniquely reducible under 
  $T(S')$. Therefore let \(b_1,b_2 \in \Irr(S')\) and \(t_1,t_2 \in 
  T(S')\) be such that $t_1(a)$ is stuck in $b_1+\ve$ under $T(S')$ 
  and $t_2(a)$ is stuck in $b_2+\ve$ under $T(S')$. Let \(\delta \in 
  \cmplO\) be such that \(t(\delta) \subseteq \ve\) for all \(t \in 
  T(S)\). By persistent $\delta$-reducibility of $a$ under $T(S')$ 
  there exist \(t_1',t_2' \in T(S')\) and \(b_1',b_2' \in \Irr(S')\) 
  such that \(t_1'\bigl( t_1(a) \bigr) \in b_1'+\delta\) and 
  \(t_2'\bigl( t_2(a) \bigr) \in b_2'+\delta\). By equicontinuity this 
  implies that $t_1'\bigl( t_1(a) \bigr)$ and $t_2'\bigl( t_2(a) 
  \bigr)$ are stuck under $T(S)$ in $b_1'+\ve$ and $b_2'+\ve$ 
  respectively. By $\ve$-unique reducibility under $T(S)$ of $a$ this 
  implies that \(b_1'-b_2' \in \ve\). Furthermore \(b_i - b_i' = 
  b_i - t_i'\bigl( t_i(a) \bigr) + t_i'\bigl( t_i(a) \bigr) - b_i' 
  \in \ve + \delta \subseteq \ve\) for \(i=1,2\) and thus \(b_1-b_2 
  \in \ve\) as well. Hence $a$ is indeed $\ve$-uniquely reducible 
  under $T(S')$, and it follows that \(\Red(S') = \cmplM\).
  
  Now let \(a \in \cmplM\) be given and consider the matter of 
  whether \(t^{S'}(a) = \tS(a)\). Let \(\ve \in \cmplO\) be arbitrary 
  and let \(\delta \in \cmplO\) be such that \(t(\delta) \subseteq 
  \ve\) for all \(t \in T(S)\). There exists some \(t \in T(S') 
  \subseteq T(S)\) such that \(t(a) \in t^{S'}(a) + \delta\), hence 
  $t(a)$ is stuck in $t^{S'}(a) + \ve$ under $T(S)$, and consequently 
  \(\tS(a) \in t^{S'}(a) + \ve\). It follows from the arbitrariness 
  of $\ve$ that \(\tS(a) = t^{S'}(a)\).
\end{proof}

\section{The core theorem}
\label{Sec:Monomordning}

At the heart of every diamond lemma lies an induction, and the role 
that $\mc{Y}$ will play is as the domain of that induction, to which 
end it is necessary to order $\mc{Y}$. 
For many novices, the need to systematically order the 
monomials is by far the most unfamiliar aspect of working with the 
diamond lemma (and\slash or Gr\"obner basis theory), and the problem 
of constructing a suitable order can be quite baffling. While this is 
not the place to give advice on how to attack that problem\Ldash 
see instead~\cite{Avhandlingen} for tips on this, in particular for 
issues regarding how the order interacts with the topology\Dash it 
will still become necessary to reason about orders and their 
relations to other structures. For that end, it helps to introduce a 
bit of notation for order relations, which will facilitate discussions 
that simultaneously involve several orders. The aim is to allow the 
order to be an ordinary mathematical letter $P$ (or more generally 
an expression), rather than a fancy symbol like $\succ$.

The basic claim one can make (with respect to an order relation $P$) 
about a pair $(\mu,\nu)$ of elements is that they are related by this 
relation. The usual formal interpretation of this is that \((\mu,\nu) 
\in P\), but notationally it is more convenient to write something 
like `\(\mu \leqslant \nu \pin{P}\)', to clarify that this is a 
non-strict inequality and that $\mu$ is on the ``small side''. Using 
that one then defines
\begin{alignat*}{3}
  \mu \geqslant{}& \nu \pin{P} &\quad&\Epil& \quad
    \nu \leqslant{}& \mu \pin{P} \text{,}\\
  \mu <{}& \nu \pin{P}  &&\Epil& 
    \mu \leqslant{}& \nu \pin{P} \text{ and } 
    \nu \not\leqslant \mu \pin{P} \text{,}\\
  \mu >{}& \nu \pin{P} &&\Epil& 
    \nu <{}& \mu \pin{P} \text{,}\\
  \mu \sim{}& \nu \pin{P}  &&\Epil& 
    \mu \leqslant{}& \nu \pin{P} \text{ and } 
    \nu \leqslant \mu \pin{P} \text{.}
\end{alignat*}
If $P$ is a partial order then \(\mu \sim \nu \pin{P}\) is the same 
thing as \(\mu=\nu\), but if $P$ is a more general quasi-order then 
this need not be the case. It is often convenient to construct 
a complex partial order by a sequence of refinements of some simpler 
quasi-order.
Not every partial order will usefully support inductions 
however, so an additional property is needed.

\begin{definition} \label{Def:TDCC}
  A binary relation $P$ on $\mc{Y}$ is said to satisfy the 
  \DefOrd{topological descending chain condition} (or to be 
  \DefOrd{TDCC} for short) if \(\lim_{n\Lpil\infty} \mu_n = 0\) for 
  every infinite sequence \(\{\mu_n\}_{n=0}^\infty \subseteq \mc{Y}\) 
  such that \(\mu_n > \mu_{n+1} \pin{P}\) for all \(n\in\N\).
\end{definition}

An informal phrasing of the ordinary descending chain condition~(DCC) 
is `there is no infinite strictly descending chain', although it is 
important to observe that `descending' here implicitly requires that 
the chain elements are indexed. An index-free formulation of the DCC 
is `every nonempty subset has a minimal element', and when this 
definition is given one usually speaks about the order being 
\emDefOrd{well-founded} (which is thus a synonym of DCC). Note that 
asking for a \emph{minimal} element is weaker than asking for a 
\emph{minimum} element; the latter would give rise to a well-order, 
which in particular is always a total order.

The next lemma gives the precise form of an induction over 
$\mc{Y}$; condition~\parenthetic{i} provides the induction base, 
whereas the verification of condition~\parenthetic{ii} is the 
induction step.

\begin{lemma} \label{L: Induktion i Y}
  Let $P$ be a partial order on $\mc{Y}$ that is TDCC. If 
  \(Z \subseteq \mc{Y}\) is such that:
  \begin{enumerate}
    \item[\parenthetic{i}]
      there exists an \(\ve\in\cmplO\) such that \(\ve \cap \mc{Y} 
      \subseteq Z\), and
    \item[\parenthetic{ii}]
      if \(\mu \in \mc{Y}\) is such that 
      \(\setOf{ \nu \in \mc{Y} }{ \nu < \mu \pin{P} } \subseteq Z\)
      then \(\mu \in Z\);
  \end{enumerate}
  then \(Z=\mc{Y}\).
\end{lemma}
\begin{proof}
  Let $Z$ be an arbitrary proper subset of $\mc{Y}$ which satisfies 
  \parenthetic{ii}; it will be shown that $Z$ does not satisfy 
  \parenthetic{i}. To see this, let \(\mu_0 \in \mc{Y} \setminus Z\). 
  For any \(\mu_n \in \mc{Y} \setminus Z\) there must exist a 
  \(\mu_{n+1} \in \mc{Y} \setminus Z\) such that \(\mu_{n+1} < 
  \mu_n \pin{P}\), because if that was not the case then all 
  \(\nu \in \mc{Y}\) which satisfy \(\nu < \mu_n \pin{P}\) 
  would also satisfy 
  \(\nu \in Z\), and hence by \parenthetic{ii} \(\mu_n \in Z\), 
  which would be a contradiction. Thus there exists an infinite 
  $P$-descending sequence \(\{\mu_n\}_{n=1}^\infty \subseteq 
  \mc{Y} \setminus Z\), and hence by TDCC \(\lim_{n \Lpil \infty} \mu_n 
  = 0\). In other words there exists for each \(\ve\in\cmplO\) an 
  integer $N$ such that \(\mu_n \in \ve\) for all \(n \geqslant N\), 
  and thus \(\ve \cap \mc{Y} \ni \mu_N \notin Z\). Hence $Z$ does 
  not satisfy \parenthetic{i}.
\end{proof}

For such inductions to be useful in the present context, it is 
however necessary that the reductions used comply with the order.

\begin{definition} \label{Def:Kompatibel}
  If $P$ is a binary relation on $\mc{Y}$ and \(\mu \in \mc{Y}\) then 
  $\DSM(\mu,P)$ \index{DSM mu P@$\DSM(\mu,P)$} denotes the least 
  topologically closed $R$-module of $\cmplM$ which contains all 
  \(\nu \in \mc{Y}\) such that \(\nu < \mu \pin{P}\), i.e.,
  \begin{equation} \label{Eq:Def:Kompatibel}
    \DSM(\mu,P) = \Cspan\bigl( 
      \setOf{ \nu \in \mc{Y} }{ \nu < \mu \pin{P} }
    \bigr) \text{.}
  \end{equation}
  The set $\DSM(\mu,P)$ is called the 
  \DefOrd[*{down-set!module}]{down-set module} of $\mu$ with respect 
  to $P$.
  
  A reduction \(t \in T(S)\) is said to be 
  \DefOrd[*{compatible!reduction}]{compatible} with the relation 
  $P$ if \(t(\mu) \in \{\mu\} \cup \DSM(\mu,P)\) for all \(\mu \in 
  \mc{Y}\). A set of reductions is said to be compatible with $P$ if 
  all its elements are compatible with $P$.
  
  An element \(g \in \cmplM\) is said to be 
  \DefOrd[*{monic}]{$P$-monic} if there exists some \(\mu \in 
  \mc{Y}\) such that \(g - \mu \in \DSM(\mu,P)\). A subset of 
  $\cmplM$ is said to be $P$-monic if all elements in it are 
  $P$-monic.
\end{definition}

\emph{Down-set}\index{down-set} is a term from poset theory, but the 
standard name there for this concept is \emph{ideal} rather than 
down-set. That terminology has however been avoided so that no 
confusion with the ring-theoretic ideal concept will arise.

That an arbitrary map \(\cmplM \Fpil \cmplM\) should be compatible 
with some order relation $P$ is a rather strong condition, but for 
reductions it is often something which comes naturally. 
For a simple reduction $t_{\mu \mapsto a}$ satisfying 
\eqref{Eq:Reduktion-monom-def}, it boils down to the condition that 
\(a \in \DSM(\mu,P)\), and when that is the case then $\mu-a$ is 
$P$-monic. Conversely, if \eqref{Eq:Reduktion-monom-def} for any 
\((\mu,a) \in \mc{Y} \times \cmplM\) defines a continuous 
homomorphism $t_{\mu \mapsto a}$ that commutes with elements of $R$, 
then every $P$-monic \(g \in \cmplM\) gives rise to some map $t_{\mu 
\mapsto \mu-g}$ that is compatible with $P$. Gr\"obner basis theory 
preaches an extreme form of this, where the normal state of things is 
that the leading monomial $\mu$ is split off from a 
basis element $g$ every time something is to be reduced modulo $g$ 
(although it is recognised that caching $\mu$ with $g$ can improve 
performance). This approach is facilitated by the fact that Gr\"obner 
basis theory normally only considers total orders, as that guarantees 
that there always is a unique leading monomial to split off.

In more general cases, compatibility is often something one arrives 
at indirectly. When preparing to apply the diamond lemma, one often 
starts with some reductions $t_{\mu \mapsto a}$ that one wants to 
use, and faces the task of constructing some $P$ with which these 
would be compatible. (This $P$ will also have to satisfy some other 
conditions, in particular the TDCC and in most cases some variant of 
\eqref{Eq:OrdnFramflyttKompatibel}, which constrains the 
possibilities quite a lot.) If the given reductions do not generate 
all the wanted congruences, then the next step is to find a set of 
$P$-monic generators which cover the rest, and then make additional 
simple reductions from these. At each step one's choices are 
restricted by the need to ensure compatibility further on, but when 
the set-up is complete it is usually a trivial matter to verify the 
compatibility of simple reductions. The next lemma then extends this 
result to general reductions.

\begin{lemma} \label{L:Kompabilitet}
  If $P$ is a partial order on $\mc{Y}$ with which \(t \in T(S)\) is 
  compatible, then for all \(\mu \in \mc{Y}\) and \(b \in \DSM(\mu,P)\) 
  it holds that \(t(b) \in \DSM(\mu,P)\). If $P$ is a partial order on 
  $\mc{Y}$ with which \(t_1,t_2 \in T(S)\) are compatible, then $t_2 
  \circ t_1$ is compatible with $P$ as well. If $T_1(S)$ is compatible 
  with a partial order $P$ on $\mc{Y}$ then the whole of $T(S)$ is 
  compatible with $P$.
\end{lemma}
\begin{proof}
  For the first claim, consider first some special $b$, and then 
  generalise following the characterisation \eqref{Eq:Def:Kompatibel} 
  of $\DSM(\mu,P)$. If \(b=\nu \in \mc{Y}\) satisfies \(\nu < \mu 
  \pin{P}\) then \(\DSM(\nu,P) \subseteq \DSM(\mu,P)\) and 
  consequently \(t(\nu) \in \{\nu\} \cup \DSM(\nu,P) \subseteq 
  \DSM(\mu,P)\) as claimed. If \(r \in R^*\) is any finite 
  composition of elements of $R$ then \(t\bigl( r(\nu) \bigr) = 
  r\bigl( t(\nu) \bigr) \in r\bigl( \DSM(\mu,P) \bigr) \subseteq 
  \DSM(\mu,P)\), thus extending the result to $b$ on the form $r(\nu)$. 
  If \(b_1,b_2 \in \DSM(\mu,P)\) are such that \(t(b_1),t(b_2) \in 
  \DSM(\mu,P)\) then clearly \(t(b_1 -\nobreak b_2) \in \DSM(\mu,P)\) 
  as well, and since this establishes that the set of $b$ for which 
  the result holds is a group it follows that the result holds for 
  arbitrary \(b \in \Span\bigl( \setOf{\nu \in \mc{Y}}{\nu<\mu\pin{P}} 
  \bigr)\). 
  Finally if \(\{b_n\}_{n=1}^\infty \subseteq \DSM(\mu,P)\) are such 
  that \(t(b_n) \in \DSM(\mu,P)\) then \(t(b) \in \DSM(\mu,P)\) also 
  for \(b=\lim_{n\Lpil\infty}b_n\) by continuity, and thus the result 
  holds for all \(b \in \DSM(\mu,P)\).
  
  For the second claim, let \(\mu \in \mc{Y}\) be arbitrary. If 
  \(t_1(\mu) = \mu\) then \((t_2 \circ\nobreak t_1)(\mu) = t_2(\mu) 
  \in \{\mu\} \cup \DSM(\mu,P)\) by the compatibility of $t_2$ with 
  $P$. Otherwise \(t_1(\mu) \in \DSM(\mu,P)\) and thus \((t_2 
  \circ\nobreak t_1)(\mu) \in \DSM(\mu,P)\) by the first claim. 
  Hence $t_2 \circ t_1$ is compatible with $P$. The third claim 
  immediately follows from the second and the observation that the 
  identity map $\id$ is compatible with all relations.
\end{proof}

An intuitive picture which might be useful is to think of the down-set 
module of $\mu$ as a sort of cone with $\mu$ at the apex. This 
picture is deceiving insofar as it represents entire $R$-modules of 
the form \(\Span\bigl(\{\mu\}\bigr)\) as single points and does not 
even begin to consider the topological structure, but it is nonetheless 
very much to the point. In that picture, one might interpret the above 
lemma as saying compatible reductions cannot map elements inside a cone 
to elements outside it. This is similar to how the sets $\Per_\ve(S)$ of 
persistently $\ve$-reducible elements behave with respect to reductions, 
and indeed the next lemma makes use of down-set modules in showing 
that $\Per_\ve(S)$ is the whole of~$\cmplM$.

\begin{lemma} \label{L: Persistently red.}
  Assume $T(S)$ is equicontinuous and compatible with some partial 
  order $P$ on $\mc{Y}$. If $P$ satisfies the topological descending 
  chain condition then \(\Per(S) = \cmplM\).
\end{lemma}
\begin{proof}
  Let \(\ve\in\cmplO\) be arbitrary. Let $Z$ be the set of all elements 
  of $\mc{Y}$ which are persistently $\ve$-reducible. By equicontinuity 
  there exists some \(\delta \in \cmplO\) such that \(t(\delta) 
  \subseteq \ve\) for all \(t \in T(S)\). Hence all \(\mu \in \mc{Y} 
  \cap \delta\) belong to $Z$, since these satisfy \(t(\mu) \in \ve\) 
  for all \(t \in T(S)\). These elements constitute the base for the 
  induction, fulfilling condition~\parenthetic{i} of 
  Lemma~\ref{L: Induktion i Y}.
  
  For the induction step, consider some arbitrary \(\mu \in \mc{Y}\). 
  Assume that all \(\nu \in \mc{Y}\) such that \(\nu < \mu \pin{P}\) 
  satisfy $\nu \in Z$; it will now be shown that this implies \(\mu 
  \in Z\). To that end, let \(t_1 \in T(S)\) be given and try to find 
  some \(t_2 \in T(S)\) and \(a \in \Irr(S)\) such that $t_2\bigl( 
  t_1(\mu) \bigr)$ is stuck in $a+\ve$. It is useful to observe 
  that the induction hypothesis, by Lemmas~\ref{L:Per modul} 
  and~\ref{L:SlutenhetPer}, implies \(\DSM(\mu,P) \subseteq 
  \Per_\ve(S)\).
  
  Depending on $\mu$ and $t_1$, there are three cases that can occur. 
  If \(\mu \in \Irr(S)\) then \(t_1(\mu) = \mu\) is stuck in 
  \(\mu+\ve\) and hence \(\mu \in Z\). If \(t_1(\mu) \neq \mu\) 
  (and hence \(\mu \notin \Irr(S)\)) then by compatibility of $t_1$ 
  with $P$ it follows that \(t_1(\mu) \in \DSM(\mu,P) \subseteq 
  \Per_\ve(S)\) and consequently by this persistent 
  $\ve$-reducibility of $t_1(\mu)$ there exist \(a \in \Irr(S)\) and 
  \(t_2 \in T(S)\) such that \(t_2\bigl( t_1(\mu) \bigr)\) is stuck 
  in $a+\ve$. Finally, if \(\mu \notin \Irr(S)\) but 
  \(t_1(\mu) = \mu\) then there still exists some 
  \(t_2' \in T(S)\) such that \(t_2'(\mu) \neq \mu\) and thus 
  \(t_2'(\mu) \in \DSM(\mu,P) \subseteq \Per_\ve(S)\). As before 
  there now exist \(t_2'' \in T(S)\) and \(a \in \Irr(S)\) such that 
  \(t_2''\bigl( t_2'(\mu) \bigr)\) is stuck in $a+\ve$, whence for 
  \(t_2 = t_2'' \circ t_2'\) one finds that \(t_2\bigl( t_1(\mu) 
  \bigr)\) is stuck in $a+\ve$. Either way, \(\mu \in Z\) by the 
  arbitrariness of $t_1$, which completes the induction step.
  
  All conditions for Lemma~\ref{L: Induktion i Y} are now fulfilled 
  and hence \(Z = \mc{Y}\). By Lemmas~\ref{L:Per modul} and 
  Lemma~\ref{L:SlutenhetPer}, \(\Per_\ve(S) = \cmplM\). By the 
  arbitrariness of $\ve$, it then follows that \(\Per(S) = 
  \bigcap_{\ve\in\cmplO} \Per_\ve(S) = \cmplM\) as well.
\end{proof}

The same conditions also suffice for giving an explicit description 
of $\Irr(S)$. It is not unusual that one can quickly establish this 
result also through more elementary arguments, but for complicated 
set-ups it is convenient to have a proof relying on (a subset of) the 
conditions of Theorem~\ref{S:CDL}.

\begin{theorem} \label{S:Konstr.Irr(S)}
  Assume $T(S)$ is equicontinuous and compatible with some partial 
  order $P$ on $\mc{Y}$. If $P$ satisfies the topological descending 
  chain condition then
  \begin{equation} \label{Eq:Konstr.Irr(S)}
    \Irr(S) = \Cspan\Bigl( 
      \setOf[\big]{ \mu \in \mc{Y} }{ 
        \text{\(t(\mu) = \mu\) for all \(t \in T_1(S)\)}
      }
    \Bigr)\text{.}
  \end{equation}
\end{theorem}
\begin{proof}
  Let $W$ be the set of irreducible elements of $\mc{Y}$.
  It follows from Lemma~\ref{L: Irr modul} that the left hand side 
  $\Irr(S)$ of \eqref{Eq:Konstr.Irr(S)} contains the right hand side 
  $\Cspan(W)$. The reverse inclusion will be established by 
  demonstrating that \(\Irr(S) \subseteq \Span(W) + \ve\) for all 
  \(\ve \in \cmplO\).
  
  Let \(\ve \in \cmplO\) be given. Let \(\delta \in \cmplO\) be such 
  that \(t(\delta) \subseteq \ve\) for all \(t \in T(S)\). Let $N$ be 
  the set of those \(a \in \cmplM\) such that there for every \(t_1 
  \in T(S)\) exists some \(t_2 \in T(S)\) such that for every \(t_3 
  \in T(S)\) it holds that \((t_3 \circ\nobreak t_2 \circ\nobreak 
  t_1)(a) \in \Span(W) + \ve\). It will now be shown that \(N = 
  \cmplM\).
  
  To see that $N$ is a subgroup of $\cmplM$, let \(a_1,a_2 \in N\) be 
  given. Let \(t_1 \in T(S)\) be arbitrary. 
  There exists some \(t_2 \in T(S)\) such that 
  \((t_3 \circ\nobreak t_2 \circ\nobreak t_1)(a_1) \in \Span(W) + \ve\) 
  for every \(t_3 \in T(S)\). There also exists some \(t_4 \in T(S)\) 
  such that \(\bigl( t_3 \circ\nobreak t_4 \circ\nobreak 
  (t_2 \circ\nobreak t_1) \bigr) (a_2) \in \Span(W) + \ve\) for every 
  \(t_3 \in T(S)\). Hence \(\bigl( t_3 \circ\nobreak (t_4 \circ\nobreak 
  t_2) \circ\nobreak t_1 \bigr) (a_1 - \nobreak a_2) = 
  \bigl( (t_3 \circ\nobreak t_4) \circ\nobreak t_2 \circ\nobreak t_1 
  \bigr) (a_1) - \bigl( t_3 \circ\nobreak t_4 \circ\nobreak 
  (t_2 \circ\nobreak t_1) \bigr) (a_2) \in \Span(W) + \ve\) for every 
  \(t_3 \in T(S)\), and thus \(a_1-a_2 \in N\) by the arbitrariness 
  of $t_1$. Elements of $R$ map $N$ into itself because they commute 
  with all reductions and map $\Span(W) + \ve$ into itself, hence $N$ 
  is an $R$-module. $N$ is topologically closed because any \(t \in 
  T(S)\) maps $N + \delta$ into $t(N) + \ve$; if \(a \in \cmpl{N}\) 
  then there is some \(a' \in N \cap (a +\nobreak \delta)\), hence 
  for any \(t_1 \in T(S)\) there exists some \(t_2 \in T(S)\) such 
  that \((t_3 \circ\nobreak t_2 \circ\nobreak t_1)(a') \in \Span(W) + 
  \ve\) for all \(t_3 \in T(S)\), and thus \((t_3 \circ\nobreak t_2 
  \circ\nobreak t_1)(a) \in (t_3 \circ\nobreak t_2 \circ\nobreak 
  t_1)(a') + \ve \subseteq \Span(W) + \ve + \ve\).
  
  The proof that \(\mc{Y} \subseteq N\) is done by induction. As 
  usual, \(\mc{Y} \cap \delta \subseteq \delta \subseteq N\) since 
  \(t(\delta) \subseteq 0 + \ve\) for all \(t \in T(S)\). For the 
  induction step, let \(\mu \in \mc{Y}\) be arbitrary and assume 
  \(\nu \in N\) for all \(\nu \in \mc{Y}\) such that \(\nu < \mu 
  \pin{P}\). If \(\mu \in W\) then \(t(\mu) \in W\) for all \(t \in 
  T(S)\) and hence \(\mu \in N\). Otherwise let \(t_1 \in T(S)\) be 
  given. If \(t_1(\mu) = \mu\) then let \(t_2' \in T_1(S)\) be such 
  that \(t_2'(\mu) \neq \mu\), otherwise let \(t_2' = \id\). Since 
  \((t_1' \circ\nobreak t_1)(\mu) \neq \mu\) it follows from 
  compatibility, the induction hypothesis, and the previous paragraph 
  that \((t_2' \circ\nobreak t_1)(\mu) \in \DSM(\mu,P) \subseteq N\). 
  Hence there exists some \(t_2 \in T(S)\) such that any \(t_3 \in 
  T(S)\) satisfies \((t_3 \circ\nobreak t_2 \circ\nobreak t_2' 
  \circ\nobreak t_1)(\mu) \in \Span(W) + \ve\) and thus \(\mu \in 
  N\). Since $P$ is TDCC, the conclusions \(\mc{Y} \subseteq N\) and 
  \(N = \cmplM\) follow.
  
  Finally, let \(a \in \Irr(S)\) be arbitrary. Since \(a \in N\) 
  there exists some \(t \in T(S)\) such that \(t(a) \in \Span(W) + 
  \ve\), but \(t(a)=a\) by irreducibility. Hence \(\Irr(S) \subseteq 
  \Span(W) + \ve\) by the arbitrariness of $a$, and it follows that 
  \(\Irr(S) \subseteq \bigcap_{\ve \in \cmplO} 
  \bigl( \Span(W) +\nobreak \ve \bigr) = \Cspan(W)\).
\end{proof}

The \(\Red(S) = \cmplM\) counterpart of Lemma~\ref{L: Persistently red.} 
is Lemma~\ref{L:Red=cmplM}, but unique reducibility requires another 
condition, wherein the following definition is handy.

\begin{definition} \label{Def:DIS}
  Let $P$ be a binary relation on $\mc{Y}$. Then for any 
  \(\mu \in \mc{Y}\), define
  \begin{equation}
    \DIS(\mu,P,S) = \Cspan\Bigl(
      \setOf[\big]{ \nu - t(\nu) }{ \nu < \mu \pin{P}, t \in T_1(S) }
    \Bigr) \text{.}
  \end{equation}
  Also write \(a \equiv b \pmod{S < \mu \pin{P}}\)
  \index{= mod S mu P@$\equiv\pmod{S < \mu \pin{P}}$} 
  (read ``$a$ is congruent to $b$ mod $S$ less $\mu$ in $P$'')
  as a shorthand for \(a-b \in \DIS(\mu,P,S)\).
\end{definition}

The etymology of this $\DIS$ notation is ``Down-set $\mc{I}(S)$ 
Section'', even though one cannot in general interpret $\DIS(\mu,P,S)$ 
as being synonymous to \(\DSM(\mu,P) \cap \mc{I}(S)\). It 
is clear that \(\DIS(\mu,P,S) \subseteq \mc{I}(S)\), and if $T(S)$ 
is compatible with $P$ then \(\DIS(\mu,P,S) \subseteq \DSM(\mu,P)\) 
by Lemma~\ref{L:Kompabilitet}, but for equality with the intersection 
to hold one pretty much have to fulfil the conditions of 
Theorem~\ref{S:CDL}. It is however not so interesting to exactly map 
the extent of this module; one rather seeks to prove that particular 
elements of $\cmplM$ belong to it by exhibiting explicit expressions 
for them. Calculations are often convenient to express in the 
form \(a \equiv a_1 \equiv \dotsb \equiv a_n \equiv 0 
\pmod{S < \mu \pin{P}}\).

\begin{lemma} \label{L:Red=cmplM}
  Let $P$ be a partial order on $\mc{Y}$. Assume $T_1(S)$ is such that 
  for all \(\mu \in \mc{Y}\) and simple reductions \(t_1,t_2 \in 
  T_1(S)\) that act nontrivially on $\mu$ it holds that
  \begin{equation} \label{Eq:Red=cmplM}
    t_1(\mu) - t_2(\mu) \in \DIS(\mu,P,S) \text{.}
  \end{equation}
  If furthermore $P$ satisfies the topological descending chain 
  condition and $T(S)$ is equicontinuous and compatible with $P$, 
  then \(\Red(S) = \cmplM\).
\end{lemma}
\begin{proof}
  It will be shown by induction over $\mc{Y}$ that \(\Red_\ve(S) = 
  \cmplM\) for all \(\ve \in \cmplO\). Observe that \(\Per(S) = 
  \cmplM\) by Lemma~\ref{L: Persistently red.}; hence it is sufficient 
  to prove that all elements of $\cmplM$ are $\ve$-uniquely reducible. 
  Let \(\ve \in \cmplO\) be given and let \(\delta \in \cmplO\) be such 
  that \(t(\delta) \subseteq \ve\) for all \(t \in T(S)\). The 
  induction hypothesis is that the \(\mu \in \mc{Y}\) under 
  consideration satisfies \(\mu \in \Red_\ve(S)\). The induction 
  hypothesis clearly holds for all \(\mu \in \mc{Y}\) such that 
  \(\mu \in \delta\), since \(t(\mu) \in \ve\) for all such 
  $\mu$ and all \(t \in T(S)\). This has laid the base for the 
  induction.
  
  For the induction step, assume that \(\nu \in \Red_\ve(S)\) for 
  all \(\nu \in \mc{Y}\) such that \(\nu < \mu \pin{P}\), and 
  consider $\mu$. By Lemmas~\ref{L: Red_ve modul} 
  and~\ref{L:SlutenhetPer}, this assumption is equivalent to 
  \(\DSM(\mu,P) \subseteq \Red_\ve(S)\). Let \(t_1,t_2 \in T(S)\) and 
  \(b_1,b_2 \in \Irr(S)\) be arbitrary elements such that 
  \(t_1(\mu)\) is stuck in $b_1+\ve$ and \(t_2(\mu)\) is stuck in 
  $b_2+\ve$. The problem now is to show that \(b_1+\ve = b_2+\ve\).
  A trivial case occurs if $t_1$ or $t_2$ acts trivially on $\mu$; it 
  can be assumed without loss of generality that $t_2$ acts trivially. 
  In this case \(\mu = t_2(\mu)\) is already known to be stuck in 
  $b_2+\ve$, and hence $t_1(\mu)$ is stuck there as well.
  
  With that taken care of, it can be assumed that $t_1$ and $t_2$ both 
  act nontrivially on $\mu$. Thus there exist \(t_{1a},t_{2a} \in 
  T_1(S)\) and \(t_{1b},t_{2b} \in T(S)\) such that: \(t_1(\mu) = 
  t_{1b}\bigl( t_{1a}(\mu) \bigr)\), \(t_2(\mu) = t_{2b}\bigl( 
  t_{2a}(\mu) \bigr)\), \(t_{1a}(\mu) \neq \mu\) and \(t_{2a}(\mu) \neq 
  \mu\). By persistent reducibility there also exists some \(t_3 \in 
  T(S)\) and \(b_3 \in \Irr(S)\) such that $t_3\bigl( t_{1a}(\mu) 
  -\nobreak t_{2a}(\mu) \bigr)$ is stuck in $b_3+\ve$. By 
  $\ve$-unique reducibility of $t_{1a}(\mu)$ there exists some 
  \(t_{1c} \in T(S)\) such that $(t_{1c} \circ\nobreak t_3 
  \circ\nobreak t_{1a})(\mu)$ is stuck in $b_1+\ve$, and similarly 
  there exists some \(t_{2c} \in T(S)\) such that $(t_{2c} 
  \circ\nobreak t_{1c} \circ\nobreak t_3 \circ\nobreak t_{2a})(\mu)$ 
  is stuck in $b_2+\ve$. Let \(t_4 = t_{2c} \circ t_{1c}\). Since
  \begin{align*}
    b_1-b_2 ={}&
    \bigl( b_1 - (t_4 \circ t_3 \circ t_{1a})(\mu) \bigr) +
      \\ &\qquad {}+
      (t_4 \circ t_3)\bigl( t_{1a}(\mu) - t_{2a}(\mu) \bigr) +
      \bigl( (t_4 \circ t_3 \circ t_{2a})(\mu) - b_2 \bigr) 
      \in \\ \in{}&
    \ve + (b_3+\ve) + \ve
    \text{,}
  \end{align*}
  it would follow that \(b_1+\ve = b_2+\ve\) if \(b_3 \in \ve\).
  
  By assumption \(t_{1a}(\mu) - t_{2a}(\mu) \in \DIS(\mu,P,S)\). 
  Thus there exist \(\{\rho_i\}_{i=1}^n \subseteq \mc{Y}\), 
  reductions \(\{u_i\}_{i=1}^n \subseteq T_1(S)\), and
  \(a_i \in \Span\bigl( \{\rho_i\} \bigr)\) for each 
  \(i=1,\dotsc,n\), such that \(\rho_i < \mu \pin{P}\) for 
  \(i=1,\dotsc,n\) and
  \begin{equation*}
    \bigl( t_{1a}(\mu) -\nobreak t_{2a}(\mu) \bigr) - 
    \sum_{i=1}^n \bigl(a_i -\nobreak u_i(a_i) \bigr) \in \delta 
    \text{.}
  \end{equation*}
  Now the idea is to construct and consider a reduction that takes 
  each term of this expression to an $\ve$-neighbourhood of its 
  normal form. Let \(w_0 = t_3\) and for each \(i=1,\dotsc,n\): let 
  \(v_i,v_i' \in T(S)\) and \(c_i,c_i' \in \Irr(S)\) be such that 
  $v_i\bigl( w_{i-1}(a_i) \bigr)$ is stuck in some $c_i+\ve$ and 
  $v_i'\bigl( (v_i \circ\nobreak w_{i-1} \circ\nobreak u_i)(a_i) 
  \bigr)$ is stuck in $c_i'+\ve$, then define \(w_i = v_i' \circ v_i 
  \circ w_{i-1}\). By $\ve$-unique reducibility of $a_i$ it follows 
  that \(c_i+\ve = c_i'+\ve\) and thus \(w_n\bigl( a_i -\nobreak 
  u_i(a_i)\bigr) \in \ve\). Hence
  \begin{multline*}
    w_n\bigl( t_{1a}(\mu) - t_{2a}(\mu) \bigr) \in
    w_n\biggl(
      \sum_{i=1}^n \bigl(a_i - u_i(a_i) \bigr) + \delta
    \biggr) = \\ =
    \sum_{i=1}^n w_n\bigl(a_i - u_i(a_i) \bigr) + w_n(\delta) 
    \subseteq \ve
    \text{.}
  \end{multline*}
  Since also \(w_n\bigl( t_{1a}(\mu) -\nobreak t_{2a}(\mu) \bigr) \in 
  b_3+\ve\), it follows that \(b_3\in\ve\) and \(b_1+\ve=b_2+\ve\). 
  This completes the induction step.
  
  By Lemma~\ref{L: Induktion i Y}, the induction hypothesis holds for 
  all \(\mu \in \mc{Y}\), and hence \(\mc{Y} \subseteq \Red_\ve(S)\), 
  which implies \(\Red_\ve(S) = \cmplM\). Since $\ve$ was arbitrary, 
  \(\Red(S) = \bigcap_{\ve>0} \Red_\ve(S) = \cmplM\) as well.
\end{proof}

The next definition helps simplify the main condition 
\eqref{Eq:Red=cmplM} of Lemma~\ref{L:Red=cmplM} to ``assume all 
ambiguities of $T_1(S)$ are resolvable relative to $P$'', which is 
one of the main equivalent conditions in the diamond lemma.

\begin{definition} \label{Def:Tvetydighet}
  An \DefOrd{ambiguity} of $T_1(S)$ is a triplet $(t_1,\mu,t_2)$, where 
  \(t_1,t_2 \in T_1(S)\) act nontrivially on \(\mu \in \mc{Y}\); the 
  ambiguities $(t_1,\mu,t_2)$ and $(t_2,\mu,t_1)$ are considered 
  equivalent. An ambiguity $(t_1,\mu,t_2)$ is said to be 
  \DefOrd[*{ambiguity!resolvable}]{resolvable} 
  if there, for every \(\ve \in \cmplO\), exists 
  reductions \(t_3,t_4 \in T(S)\) such that 
  \(t_3\bigl( t_1(\mu) \bigr) - t_4\bigl( t_2(\mu) \bigr) \in \ve\). 
  An ambiguity $(t_1,\mu,t_2)$ is said to be 
  \DefOrd[*{ambiguity!resolvable relative to}]{resolvable relative to} 
  a binary relation $P$ on $\mc{Y}$ if \(t_1(\mu) - 
  t_2(\mu) \in \DIS(\mu,P,S)\).
\end{definition}

The essential content of the ambiguity concept has been given a 
bewildering variety of names, where `ambiguity' is that used by 
Bergman~\cite{Bergman}. The most common term is rather 
\emDefOrd{critical pair}, but there appears to be no consensus on 
what the elements of the critical pair are. 
Baader--Nipkow~\cite[Def.~6.2.1]{BaaderNipkow} effectively defines a 
critical pair to be some $\bigl( t_1(\mu), t_2(\mu) \bigr)$ and 
informally speaks of the ambiguity $(t_1,\mu,t_2)$ from which it came 
as a \emDefOrd{fork}. This critical pair terminology would make sense 
within the present framework, but it cannot completely replace 
ambiguities, as there is not enough information in the critical pair 
to define relative resolvability. In contrast, the definition of 
`critical pair' in Gr\"obner basis theory (a pair of Gr\"obner basis 
elements) is technically quite different and cannot be stated in the 
generic framework, although the essential content is still the same.

The second most common name is probably `overlap', but although 
overlap ambiguities are by far the most important ones, there are also 
important ambiguities which aren't overlaps; the taxonomy of 
ambiguities is a subject of Section~\ref{Sec:Tvetydigheter}. Rarer 
names still are `composition' (Shirshov~\cite{Shirshov} and 
Bokut~\cite{Bokut}, hence the alternative name \emDefOrd{composition 
lemma} for the diamond lemma) and `superposition' 
(Knuth--Bendix~\cite{KnuthBendix}), both of which refer primarily to 
the $\mu$ part of an ambiguity $(t_1,\mu,t_2)$.

It should also be pointed out that many of the above concepts 
presume a certain minimality\Ldash the \emph{critical} of `critical 
pair' refers to that these are the ones that really need to be 
checked\Dash whereas the above ambiguity concept has no such 
restriction. This is because the mechanisms traditionally used to 
discard some ambiguities as redundant rely on structures not apparent 
in the basic framework $\bigl( \mc{M},R,\mc{Y},\mc{O},T_1(S) \bigr)$, 
and therefore not available in this generality. Corresponding results 
for the present setting can be found in 
Section~\ref{Sec:Tvetydigheter}.

It is common to say that a rewriting system $S$ is 
\emDefOrd{confluent} if everything has a unique normal form, but 
since this by Theorem~\ref{S:CDL} is equivalent to a number of 
quite different conditions, one shouldn't be surprised if different 
authors define it differently. Taking~\cite{BaaderNipkow} as 
authority, where confluence is defined for reduction relations, one 
may call $T_1(S)$ \emDefOrd{locally confluent} if all ambiguities of 
$T_1(S)$ are resolvable; this adjusts the traditional definition to 
allow for topology and take advantage of the $R$-module structure, but 
is otherwise a strict interpretation. Global confluence is, assuming 
persistent reducibility, more directly equivalent to unique 
reducibility: $T_1(S)$ is globally confluent if there for every \(\mu 
\in \mc{Y}\), all \(t_1,t_2 \in T(S)\) (\emph{not} only simple 
reductions), and every \(\ve \in \cmplO\) exist \(t_3,t_4 \in T(S)\) 
such that \(t_3\bigl( t_1(\mu) \bigr) - t_4\bigl( t_2(\mu) \bigr) \in 
\ve\). The two may seem similar, but a proof that local confluence 
implies global confluence (which essentially is what the original 
diamond lemma of Newman~\cite{Newman} was all about) requires 
something like an induction over $\mc{Y}$ to go through.

As will become clear in the next section, relative resolvability is 
more important in the theoretical machinery than plain resolvability, 
since it more easily lends itself to reasoning about elements of 
$\cmplM \setminus \mc{Y}$. On the other hand, plain resolvability is 
usually a more natural goal to aim for in practical calculations. The 
next lemma says that it is a sufficient condition also for relative 
resolvability.

\begin{lemma} \label{L:RelResolv}
  If $T(S)$ is compatible with the partial order $P$ on $\mc{Y}$, 
  then 
  \begin{align} 
    \DIS(\mu,P,S) ={}& \overline{ 
      \sum_{t \in T_1(S)} \setOf[\big]{ a - t(a) }{ a\in\DSM(\mu,P) }
    } \label{Eq1:Kar.DIS}\\ ={}&
    \overline{ 
      \sum_{t \in T(S)} \setOf[\big]{ a - t(a) }{ a\in\DSM(\mu,P) }
    } \label{Eq2:Kar.DIS}
  \end{align}
  for all \(\mu \in \mc{Y}\) and each resolvable ambiguity of 
  $T_1(S)$ is also resolvable relative to $P$.
\end{lemma}
\begin{proof}
  Let \(\mu \in \mc{Y}\) be given.
  It is clear that \eqref{Eq1:Kar.DIS} and \eqref{Eq2:Kar.DIS} hold 
  with $\subseteq$ inclusions, so what needs to be shown are 
  the $\supseteq$ inclusions. In \eqref{Eq2:Kar.DIS}, one may observe 
  that any \(t \in T(S)\setminus\{\id\}\) decomposes as \(t = t_n \circ 
  \dotsb \circ t_1\) for \(t_1,\dotsc,t_n \in T_1(S)\) and that 
  \(a - t(a) = \sum_{i=1}^n \bigl( a_i -\nobreak t_i(a_i) \bigr)\) for 
  \(a_1 = a\) and \(a_{i+1} = (t_i \circ\nobreak \dotsb \circ\nobreak 
  t_1)(a)\) for \(i=1,\dotsc,n-1\). By Lemma~\ref{L:Kompabilitet}, 
  \(a_i \in \DSM(\mu,P)\) for \(i=1,\dotsc,n\), and hence
  \[
    \setOf[\big]{ a - t(a) }{ a\in\DSM(\mu,P), t \in T(S) } 
    \subseteq
    \! \sum_{t \in T_1(S)} \!\!\! 
      \setOf[\big]{ a - t(a) }{ a\in\DSM(\mu,P) }
    \text{.}
  \]
  
  In \eqref{Eq1:Kar.DIS}, one must instead decompose the elements of 
  $\DSM(\mu,P)$. Let \(t \in T_1(S)\) and \(a \in \DSM(\mu,P)\) be 
  arbitrary. Let \(\ve \in \cmplO\) be arbitrary. Let \(\delta \in 
  \cmplO\) be such that \(\delta \subseteq \ve\) and \(t(\delta) 
  \subseteq \ve\). By Lemma~\ref{L:Cspan-uppdelning}, there exist 
  \(\nu_1,\dotsc,\nu_n \in \mc{Y}\) and \(r_1,\dotsc,r_n \in \pm R^*\) 
  such that \(\nu_i < \mu \pin{P}\) for \(i=1,\dotsc,n\) and 
  \(b := \sum_{i=1}^n r_i(\nu_i) \in a + \delta\). Clearly
  \[
    b - t(b) = 
    \sum_{i=1}^n r_i(\nu_i) - 
      t\biggl( \sum_{i=1}^n r_i(\nu_i) \biggr) =
    \sum_{i=1}^n r_i\bigl( \nu_i - t(\nu_i) \bigr) \in
    \DIS(\mu,P,S)
  \]
  and since \(\bigl(b -\nobreak t(b) \bigr) - \bigl(a -\nobreak t(a) 
  \bigr) \in \ve\) it follows, by the arbitrariness of $\ve$, that 
  $a-t(a)$ is a limit point of $\DIS(\mu,P,S)$. Since this set is 
  topologically closed by definition, \(a - t(a) \in 
  \DIS(\mu,P,S)\) and hence 
  \[
    \sum_{t \in T_1(S)} \!
      \setOf[\big]{ a - t(a) }{ a\in\DSM(\mu,P)} \subseteq
    \DIS(\mu,P,S)
  \]
  by the arbitrariness of $a$ and $t$. \eqref{Eq1:Kar.DIS} follows.
  
  Let $(t_1,\mu,t_2)$ be a resolvable ambiguity of $T_1(S)$, let 
  \(a_1 = t_1(\mu)\), and let \(a_2 = t_2(\mu)\). Let \(\ve \in 
  \cmplO\) be arbitrary. Let \(t_3,t_4 \in T(S)\) be such that 
  \(t_3(a_1) - t_4(a_2) \in \ve\). By Lemma~\ref{L:Kompabilitet}, 
  \(a_1,a_2 \in \DSM(\mu,P)\). Hence \(b = \bigl( a_1 -\nobreak 
  t_3(a_1)\bigr) + \bigl( -a_2 - t_4(-a_2) \bigr) \in \DIS(\mu,P,S)\) 
  and \((a_1 -\nobreak a_2) - b = t_3(a_1) - t_4(a_2) \in \ve\). In 
  other words, $a_1-a_2$ is a limit point of $\DIS(\mu,P,S)$. 
  As above, it follows that \(a_1-a_2 \in \DIS(\mu,P,S)\).
\end{proof}

With that final implication, the big equivalence in the generic 
diamond lemma is now apparent:

\begin{theorem} \label{S:CDL}
  If $T(S)$ is equicontinuous and compatible with a partial order $P$ 
  on $\mc{Y}$ that furthermore satisfies the topological descending 
  chain condition, then the following claims are equivalent:
  \begin{enumerate}
    \item[\parenthetic{a}] 
      Every ambiguity of $T_1(S)$ is resolvable.
    \item[\parenthetic{a\textprime}] 
      Every ambiguity of $T_1(S)$ is resolvable relative to $P$.
    \item[\parenthetic{b}]
      Every element of $\cmplM$ is persistently and uniquely 
      reducible, i.e., \(\cmplM = \Red(S)\).
    \item[\parenthetic{c}]
      Every element of $\cmplM$ has a unique normal form, i.e., 
      \(\cmplM = \Irr(S) \oplus \mc{I}(S)\).
  \end{enumerate}
\end{theorem}
\begin{proof}
  By Lemma~\ref{L: Persistently red.}, \(\Per(S) = \cmplM\). Hence 
  (b) and (c) are equivalent by Lemma~\ref{L:(b)<=>(c)}. (a)~implies 
  (a\textprime) by Lemma~\ref{L:RelResolv} and (a\textprime)~implies 
  (b) by Lemma~\ref{L:Red=cmplM}. Hence the only thing left to prove 
  is that (b)~implies~(a).
  
  Let an ambiguity $(t_1,\mu,t_2)$ be given. Let \(\ve \in \cmplO\) be 
  arbitrary. Since \(\mu \in \Red(S)\) there exist \(t_3,t_4 \in 
  T(S)\) such that $t_3\bigl( t_1(\mu) \bigr)$ and $t_4\bigl( 
  t_2(\mu) \bigr)$ are stuck in $\tS(\mu)+\ve$. Hence \(t_3\bigl( 
  t_1(\mu) \bigr) - t_4\bigl( t_2(\mu) \bigr) \in \ve\), and thus the 
  ambiguity is resolvable.
\end{proof}

\section{Ambiguities}
\label{Sec:Tvetydigheter}

Most applications of a diamond lemma has as one of its main steps the 
calculations for checking that the ambiguities are resolvable. In 
Gr\"obner basis theory this is even more central, with various 
completion algorithms being driven by these ambiguity resolution 
calculations (and adding the twist of modifying the set of reductions 
whenever it is found that an ambiguity fails to resolve). From the 
theoretical foundations these algorithms above all else need some 
criteria for discarding ambiguities that don't need to be checked, as 
there typically are infinitely many triplets $(t,\mu,u)$ which 
qualify as ambiguities under Definition~\ref{Def:Tvetydighet}. 
In order to formulate such criteria one needs some extra structure 
however, and one that performs very well is to have a family of 
advanceable maps.

What makes advanceable maps useful for structuring a set of 
ambiguities is primarily that one ambiguity can be the image of 
another ambiguity; indeed, for e.g.~the family \eqref{Eq:2side-mult.} 
of advanceable maps \emph{every} image of an ambiguity is another 
ambiguity. Such images are however never more than shadows of the 
original ambiguities, since also the ambiguity resolutions can be 
transported to the image by the advanceable map.
This argument (see Lemma~\ref{L:AbsolutSkugga} for the formal claim) 
is classically used to prove that it is only necessary 
to check resolvability for minimal ambiguities (since any non-minimal 
ambiguity would be a shadow of a smaller ambiguity), but it can also 
be used to prove that it is sufficient to 
check one labelling of an expression, since relabelling maps are 
often advanceable. A catch is however that even if one can skip 
checking any particular shadow ambiguity, it does not necessarily 
follow that all shadow ambiguities can be skipped\Dash the 
``original'' of which an ambiguity is a shadow can itself be a shadow 
of the shadow (e.g.~relabellings are typically invertible). The 
`shadow-critical' concept of Definition~\ref{Def:V-kritisk} is one 
way around this catch, even though it in general doesn't discard 
everything that might be skipped.

A major complication when considering shadow ambiguities is 
that the original ambiguity will often have a different 
sort\Ldash reside in a different base set\Dash than the shadow that 
one wishes to resolve. Therefore it is in this section necessary to 
make the multiplicity in the basic framework 
$\bigl( \mc{M}, R, \mc{Y}, \mc{O}, T_1(S) \bigr)$ explicit, and 
think in terms of a family of such frameworks. Thus there is 
a set $I$\index{I@$I$ (set of sorts)} (the set of ``sorts'') which 
serves as the index set for the family of frameworks, and for each 
\(i \in I\) there is a quintuplet 
$\bigl( \mc{M}(i), R(i), \mc{Y}(i), \mc{O}(i), T_1(S)(i) \bigr)$ where:
\begin{itemize}
  \item
    $\mc{M}(i)$\index{M i@$\mc{M}(i)$} is a topological abelian group 
    (and $\cmplM(i)$\index{M-bar i@$\cmplM(i)$} is its completion).
  \item 
    $R(i)$\index{R i@$R(i)$} is a set of continuous group endomorphisms 
    on $\mc{M}(i)$, and hence on $\cmplM(i)$.
  \item
    $\mc{Y}(i)$\index{Y i@$\mc{Y}(i)$} is a spanning subset of 
    $\mc{M}(i)$.
  \item 
    \leavevmode\index{O i@$\mc{O}(i)$}\(\mc{O}(i) = 
    \bigl\{ B_n(i) \bigr\}_{n=1}^\infty\)\index{B n i@$B_n(i)$} is a 
    family of $R(i)$-modules (hence subsets of $\mc{M}(i)$) that 
    constitutes a family of fundamental neighbourhoods of $0$.
  \item
    $T_1(S)(i)$\index{T 1 S i@$T_1(S)(i)$} is a set of continuous 
    group endomorphisms on $\cmplM(i)$ that commute with elements 
    of $R(i)$.
\end{itemize}
As with the five main pieces,
\begin{itemize}
  \item
    the particular partial order on $\mc{Y}(i)$ will be denoted $P(i)$, 
    but all assumptions on this are explicit, like they were in the 
    previous section.
\end{itemize}
Other things defined from the basic framework are similarly 
specialised to a sort $i$ by appending an `$(i)$\index{(i)@\dots$(i)$}' 
to the symbol; the parenthesis notation for indexing may seem a bit 
peculiar, but it is traditional for operads, which have inspired much 
of the multi-sorted extensions to this formalism. Operads have \(I = \N\) 
with the index $i$ being the arity of the elements concerned, so 
there is one set $\Irr(S)(0)$ of irreducible constants, another set 
$\Irr(S)(1)$ of irreducible unary operations, yet another set 
$\Irr(S)(2)$ of irreducible binary operations, etc.; other types of 
algebraic structures typically require other index sets. What happens 
with respect to ambiguities is that each $T_1(S)(i)$ has its own 
ambiguities, but it is frequently the case that they turn out to be 
shadows of ambiguities in some $T_1(S)(i')$.

The family $V$ of maps that one wants to have advanceable reacts 
differently to the introduction of several sorts: it acquires two 
sort indices, since there is no reason the codomain $\cmplM(i)$ 
should have the same sort as the domain $\cmplM(i')$. It is however 
not until Definition~\ref{Def:V-kritisk} that this has to be made 
explicit; before that it is sufficient to reason about specific 
advanceable maps that relate to specific ambiguities. 
An underlying set $S$ of rewrite rules will typically not carry sort 
indices, as every element of it contributes to every $T_1(S)(i)$.

\begin{definition} \label{Def:Framflyttbar2}
  Let \(i,i' \in I\) be given. A map \(v\colon \cmplM(i') \Fpil 
  \cmplM(i)\) is said to be \DefOrd{advanceable} with respect to 
  $T_1(S)(i')$ and $T_1(S)(i)$ if there for every \(t' \in T_1(S)(i')\) 
  and \(a \in \RstarY(i')\) exists some \(t \in T(S)(i)\) such that 
  \(
    t\bigl( v(a) \bigr) = v\bigl( t'(a) \bigr)
  \).
  The map $v$ is said to be 
  \DefOrd[*{advanceable!absolutely}]{absolutely advanceable} with 
  respect to $T(S)(i')$ and $T(S)(i)$ if there for every \(t' \in 
  T(S)(i')\) exists some \(t \in T(S)(i)\) such that 
  \(v \circ t' = t \circ v\).
  
  An ambiguity $(t,\mu,u)$ of $T_1(S)(i)$ is said to be a 
  \DefOrd[*{ambiguity!shadow}]{shadow} of the ambiguity $(t',\mu',u')$ 
  of $T_1(S)(i')$ if there exists an advanceable continuous 
  homomorphism \(v\colon \cmplM(i') \Fpil \cmplM(i)\) such that 
  \(\mu = v(\mu')\), \(t(\mu) = v\bigl( t'(\mu') \bigr)\), and 
  \(u(\mu) = v\bigl( u'(\mu') \bigr)\). The ambiguity $(t,\mu,u)$ is 
  said to be an \DefOrd[*{ambiguity!absolute shadow}]{absolute 
  shadow} if $v$ is absolutely advanceable.
\end{definition}

In the classical case of Bergman's diamond lemma, there is only one 
sort\Ldash preferably denoted $1$ for consistency with the operad 
generalisation\Dash and hence the indices could be dropped. Not 
dropping indices, one would have \(\mc{M}(1) = \RavX\), \(\mc{Y}(1) = 
X^*\), and $T_1(S)(1)$ being the set of all maps $t_{\nu_1 s \nu_2}$ 
on the form \eqref{Eq1:Trad.def.T_1(S)}; an ambiguity is thus some 
$(t_{\lambda_1 s_1 \nu_1}, \mu, t_{\lambda_2 s_2 \nu_2})$ where 
\(\lambda_1 \mu_{s_1} \nu_1 = \mu = \lambda_2 \mu_{s_2} \nu_2\). 
However, if $\lambda_1$ and $\lambda_2$ have some common prefix 
$\kappa$ (i.e., \(\lambda_1 = \kappa \lambda_1'\) and \(\lambda_2 = 
\kappa \lambda_2'\) for some \(\kappa,\lambda_1',\lambda_2' \in 
X^*\)) and\slash or $\nu_1$ and $\nu_2$ have some common suffix 
$\rho$ (i.e., \(\nu_1 = \nu_1' \rho\) and \(\nu_2 = \nu_2' \rho\) for 
some \(\nu_1',\nu_2',\rho \in X^*\)) then for \(\mu' = 
\lambda_1' \mu_{s_1} \nu_1' = \lambda_2' \mu_{s_2} \nu_2'\) one finds 
that $(t_{\lambda_1 s_1 \nu_1}, \mu, t_{\lambda_2 s_2 \nu_2})$ is a 
shadow under the absolutely advanceable map \(v(b) = \kappa b \rho\) 
of the ambiguity $(t_{\lambda_1' s_1 \nu_1'}, \mu', 
t_{\lambda_2' s_2 \nu_2'})$. From the unique factorisation in 
the free monoid $X^*$, it follows that the only ambiguities that are 
not such shadows are those where at least one of $\lambda_1$ and 
$\lambda_2$, and at least one of $\nu_1$ and $\nu_2$, are equal to 
the identity $\ssI$. This, with the help of the next two lemmas, 
cuts down the number of ambiguities that explicitly need to 
be resolved quite considerably.

\begin{lemma} \label{L:AbsolutSkugga}
  If an ambiguity is resolvable then all its absolute shadows are 
  resolvable as well.
\end{lemma}
\begin{proof}
  Let the indices \(i,i' \in I\) be given. Let $(t_1',\mu',t_2')$ 
  be a resolvable ambiguity of $T_1(S)(i')$. Let $(t_1,\mu,t_2)$ 
  be an arbitrary ambiguity of $T_1(S)(i)$ that is an absolute shadow 
  of $(t_1',\mu',t_2')$, and let \(v\colon \cmplM(i') \Fpil \cmplM(i)\) 
  be the absolutely advanceable continuous homomorphism which links the 
  two ambiguities. Let \(\ve \in \cmplO(i)\) be arbitrary, and let 
  \(\delta \in \cmplO(i')\) be such that \(v(\delta) \subseteq \ve\). 
  Since $(t_1',\mu',t_2')$ is resolvable there exists \(t_3',t_4' \in 
  T(S)(i')\) such that \(t_3'\bigl( t_1'(\mu') \bigr) - 
  t_4'\bigl( t_2'(\mu') \bigr) \in \delta\). By absolute advanceability 
  of $v$ there exists \(t_3,t_4 \in T(S)(i)\) such that \(t_3 \circ v 
  = v \circ t_3'\) and \(t_4 \circ v = v \circ t_4'\). Then
  \begin{align*}
    t_3\bigl( t_1(\mu) \bigr) - t_4\bigl( t_2(\mu) \bigr) ={}& 
    (t_3 \circ t_1)\bigl( v(\mu') \bigr) - 
      (t_4 \circ t_2)\bigl( v(\mu') \bigr) 
      = \\ ={}&
    (v \circ t_3' \circ t_1')(\mu') - (v \circ t_4' \circ t_2')(\mu) 
      = \\ ={}&
    v\bigl( (t_3' \circ t_1')(\mu') - (t_4' \circ t_2')(\mu') \bigr) \in 
    v(\delta) \subseteq \ve \text{,}
  \end{align*}
  and since $\ve$ was arbitrary it follows that $(t_1,\mu,t_2)$ is 
  resolvable.
\end{proof}

In this lemma, it would not have been sufficient to assume 
conditional advanceability, and it is instructive to consider why. 
Suppose $t_1(\mu')$ is not a single element of $\RstarY(i')$, but is 
instead the sum $\lambda + \nu$ of two different elements of 
$\mc{Y}(i')$. 
Suppose further that \(t_3' \in T_1(S)(i')\), because 
the extension to non-simple reductions in absolute advanceability is 
not the main issue. If $v$ is advanceable then there certainly exist 
\(t_{3\lambda},t_{3\nu} \in T(S)(i)\) such that \(t_{3\lambda}\bigl( 
v(\lambda) \bigr) = v\bigl( t_3'(\lambda) \bigr)\) and 
\(t_{3\nu}\bigl( v(\nu) \bigr) = v\bigl( t_3'(\nu) \bigr)\), but 
there is no guarantee that there is some \(t_3 \in T(S)(i)\) such 
that \(t_3\bigl( v( \lambda +\nobreak \nu ) \bigr) = v\bigl( t_3'( 
\lambda +\nobreak \nu ) \bigr)\). In very many cases it would 
probably turn out that something like $t_{3\lambda} \circ t_{3\nu}$ 
acts exactly as the $t_3$ one needs, because good choices of simple 
reductions tend to act trivially on large subsets of $\cmplM(i)$, but 
since it cannot in general be assumed that $t_{3\nu}$ acts trivially 
on $v(\lambda)$, that composition will sometimes fail. 

The idea to first reduce one term, and then the next, is basically 
good but requires some kind of book-keeping device to work. Provided 
that simple reductions are of the $t_{\mu \mapsto a}$ kind (i.e., 
each only acts nontrivially on one element of $\mc{Y}$), a 
possibility would be to use the order on $\mc{Y}(i)$ and start with 
the smallest terms, but a more powerful solution is to go for 
relative resolvability instead; this provides for reducing different 
terms independently of each other. The small price one has to pay is 
a condition on how the advanceable map behaves with respect to 
the partial orderings.

\begin{lemma} \label{L:Rel.Res.Shadow}
  Assume the ambiguity $(t,\mu,u)$ of $T_1(S)(i)$ is a shadow of 
  the ambiguity $(t',\mu',u')$ of $T_1(S)(i')$, that the 
  corresponding advanceable map \(v\colon \cmplM(i') \Fpil 
  \cmplM(i)\) satisfies
  \begin{equation} \label{Eq:Rel.Res.Shadow}
    v\Bigl( \DSM\bigl( \mu', P(i') \bigr) \Bigr) \subseteq
    \DSM\bigl( \mu, P(i) \bigr)
    \text{,}
  \end{equation}
  and that $T(S)(i)$ is compatible with the partial order $P(i)$. 
  If $(t',\mu',u')$ is resolvable relative to $P(i')$, then 
  $(t,\mu,u)$ is resolvable relative to $P(i)$.
  
  More generally, an advanceable continuous homomorphism 
  \(v\colon \cmplM(i') \Fpil \cmplM(i)\) satisfying 
  \eqref{Eq:Rel.Res.Shadow} also maps $\DIS\bigl( \mu', P(i'), S 
  \bigr)$ into $\DIS\bigl( \mu, P(i), S \bigr)$ if $T_1(S)(i)$ is 
  compatible with the partial order $P(i)$.
\end{lemma}
\begin{proof}
  Let \(D = \DIS\bigl( \mu, P(i), S \bigr)\) and \(D' = 
  \DIS\bigl( \mu', P(i'), S \bigr)\). That $(t,\mu,u)$ is resolvable 
  relative to $P(i)$ is by definition that \(t(\mu) - u(\mu) \in D\), 
  and since \(t(\mu) - u(\mu) = v\bigl( t'(\mu') -\nobreak u'(\mu') 
  \bigr)\) where \(t'(\mu') -\nobreak u'(\mu') \in D'\), the first 
  claim follows from the second: that \(v(D') \subseteq D\).
  
  Let \(a \in D'\) be given. Since $D$ is topologically closed, it 
  follows that \(v(a) \in D\) if it can be shown that \(v(a) \in D + 
  \ve\) for every \(\ve \in \cmplO(i)\). 
  Let \(\ve \in \cmplO(i)\) be arbitrary. Let \(\delta \in 
  \cmplO(i')\) be such that \(v(\delta) \subseteq \ve\). 
  Let \(\{\nu_j\}_{j=1}^m \subseteq \mc{Y}(i')\), \(\{r_j\}_{j=1}^m 
  \subseteq \pm R^*(i')\), and \(\{t_j'\}_{j=1}^m \subseteq 
  T_1(S)(i')\) be such that 
  \begin{equation*}
    \sum_{j=1}^m r_j\bigl( \nu_j - t_j'(\nu_j) \bigr) \in 
    a + \delta 
  \end{equation*}
  and \(\nu_j < \mu' \pin{P(i')}\) for \(j=1,\dotsc,m\). 
  Let \(b_j = r_j(\nu_j)\) for \(j=1,\dotsc,m\) and let 
  \(b = \sum_{j=1}^m \bigl( b_j -\nobreak t_j'(b_j)  \bigr)\). 
  
  By advanceability of $v$ there exist reductions \(\{t_j\}_{j=1}^m 
  \subseteq T(S)(i)\) such that \((t_j \circ\nobreak v)(b_j) = 
  (v \circ\nobreak t_j')(b_j)\) for \(j=1,\dotsc,m\). Since \(b_j 
  \in \DSM\bigl( \mu', P(i') \bigr)\) it follows that \(v(b_j) \in 
  \DSM\bigl( \mu, P(i) \bigr)\) for \(j=1,\dotsc,m\), and hence
  \begin{equation*}
    v(b) = 
    v\biggl( \sum_{j=1}^m \Bigl( b_j - t_j'(b_j) \Bigr) \biggr) =
    \sum_{j=1}^m \Bigl( v(b_j) - t_j\bigl( v(b_j) \bigr) \Bigr) \in
    D
  \end{equation*}
  by Lemma~\ref{L:RelResolv}. Furthermore 
  \(
    v(a) - v(b) = 
    v( a -\nobreak b ) \in 
    v(\delta) \subseteq \ve 
  \),
  thus \(v(a) \in D\) by the arbitrariness of $\ve$.
\end{proof}

The following concepts are useful when one seeks to prove that an 
advanceable map (or family of advanceable maps) satisfies 
\eqref{Eq:Rel.Res.Shadow}.

\begin{definition}
  A map \(v\colon \cmplM(i') \Fpil \cmplM(i)\) is said to 
  \DefOrd{correlate} $P(i')$ to $P(i)$ if
  \begin{equation} 
    v(\mu) \in \mc{Y}(i)
    \Ipil
    v\Bigl( \DSM\bigl( \mu, P(i') \bigr) \Bigr) \subseteq
      \DSM\bigl( v(\mu), P(i) \bigr)
  \end{equation}
  for all \(\mu \in \mc{Y}(i')\). 
  A map \(v\colon \mc{Y}(i') \Fpil \mc{Y}(i)\) is said to be 
  \DefOrd{monotone} with respect to $P(i')$ and $P(i)$ if
  \begin{equation} \label{Eq:Monoton}
    \nu \leqslant \mu \pin{P}(i')
    \Ipil
    v(\nu) \leqslant v(\mu) \pin{P}(i)
  \end{equation}
  for all \(\mu,\nu \in \mc{Y}(i')\). The map $v$ is 
  \DefOrd{strictly monotone} if
  \begin{equation}  \label{Eq:OrdnFramflyttKompatibel}
    \nu < \mu \pin{P}(i')
    \Ipil
    v(\nu) < v(\mu) \pin{P}(i)
  \end{equation}
  for all \(\mu,\nu \in \mc{Y}(i')\).
\end{definition}

It needs to be pointed out that there, particularly for functions 
with domain and codomain in $\R$, exists a conflicting terminology 
which calls the property defined in \eqref{Eq:Monoton} `increasing' 
and instead defines `monotone' as `increasing or decreasing'; 
preferences vary. When `monotone' as here means ``preserves 
inequalities'' then the corresponding name for ``reverses 
inequalities'' is \emDefOrd{antitone}. Yet another name that might be 
seen for a map having properties like these is that it is 
`compatible' with the order, but here Definition~\ref{Def:Kompatibel} 
has already given that name to a different relation between maps and 
binary relations.

The applied concept in this trio is that of a map which correlates 
$P(i')$ to $P(i)$: it covers the condition \eqref{Eq:Rel.Res.Shadow} of 
Lemma~\ref{L:Rel.Res.Shadow} and it blends nicely with 
Construction~\ref{Konstr:t_v,s-reduktion} in that it 
reduces compatibility of simple reductions made from a rule 
$(\mu,a)$ to the matter of whether \(a \in \DSM(\mu,P)\). On 
the other hand, it is usually monotonicity that is the goal when one 
constructs the relations $\bigl\{ P(i) \bigr\}_{i \in I}$, so a small 
lemma bridging the gap may be in order.

\begin{lemma} \label{L:OrdnFramflyttKompatibel}
  Let \(i,i' \in I\) be sorts. 
  Let \(v\colon \cmplM(i') \Fpil \cmplM(i)\) be a continuous 
  homomorphism such that \(v\bigl( r(\mu) \bigr) \in
  \Cspan\bigl(\bigl\{ v(\mu) \bigr\} \bigr)\) for all \(\mu \in 
  \mc{Y}(i')\) and \(r \in R^*(i')\). Let $P(i)$ be a binary relation 
  on $\mc{Y}(i)$ and let $P(i')$ be a binary relation on 
  $\mc{Y}(i')$. If \(v(\nu) \in \DSM\bigl( v(\mu), P(i) \bigr)\) for 
  all \(\mu,\nu \in \mc{Y}\) such that \(\nu < \mu \pin{P(i')}\) and 
  \(v(\mu) \in \mc{Y}\), then $v$ correlates $P(i')$ to $P(i)$.
\end{lemma}
\begin{proof}
  Let \(\mu \in \mc{Y}(i')\) such that \(v(\mu) \in \mc{Y}\) be given. 
  Let \(D = \DSM\bigl( v(\mu), P(i) \bigr)\). If \(\nu \in \mc{Y}(i')\) 
  satisfies \(\nu < \mu \pin{P(i')}\) then by assumption \(v(\nu) \in 
  D\) and hence \(\Cspan\bigl(\bigl\{ v(\nu) \bigr\}\bigr) \subseteq 
  D\), which implies \(v\bigl( r(\nu) \bigr) \in D\) for all \(r \in 
  R^*(i')\). Since $v$ is a continuous homomorphism and $D$ is a 
  topologically closed group, it now follows that
  \begin{equation*}
    v\Bigl( \Cspan\bigl( 
      \setOf[\big]{ \nu \in \mc{Y}(i') }{ \nu < \mu \pin{P(i')} }
    \bigr) \Bigr) \subseteq D \text{,}
  \end{equation*}
  i.e., $v$ satisfies the condition at $\mu$ for correlating $P(i')$ 
  to $P(i)$.
\end{proof}

\begin{remark}
  The meaning of \eqref{Eq:OrdnFramflyttKompatibel} if $v$ ranges 
  over all maps in the family \eqref{Eq:2side-mult.}, as would be the 
  setting for Bergman's diamond lemma, is that
  \begin{equation} \label{Eq:Semigruppsordning}
    \nu < \mu \pin{P(1)} \Ipil 
    \lambda\nu\rho < \lambda\mu\rho \pin{P(1)}
    \qquad\text{for all \(\mu,\nu,\lambda,\rho \in X^*\),}
  \end{equation}
  i.e., $P(1)$ must be a (strictly compatible) monoid partial 
  order\index{compatible!partial order}. 
  The need for this very classical condition can thus in the 
  generic theory be found only in the resolvability of ambiguities! 
  It would be possible to apply the \emph{generic} diamond lemma with 
  a partial order that violates \eqref{Eq:Semigruppsordning}, but for 
  that one would then pay the price that ambiguity resolution gets 
  more complicated. Furthermore compatibility of reductions interacts 
  with advanceability so that one anyway comes pretty close to needing 
  correlation; in practice the choice one has is one of how many 
  advanceable maps one will have rather than whether these will 
  correlate the partial orders.
  
  The condition that advanceable maps should be continuous can 
  similarly be regarded as a condition on how the multiplication 
  operation on \(\mc{M}(1) = \RavX\) should relate to the topology, 
  and if for example \eqref{Eq:Norm&Span} holds then continuity of 
  maps on the form \eqref{Eq:2side-mult.} can also be reduced to a 
  condition on multiplication of monomials, but this point of view is 
  not as striking as it is for the order on $\mc{Y}$, since continuity 
  to the average mathematician is more of an everyday condition.
\end{remark}

In most classical cases, Lemma~\ref{L:OrdnFramflyttKompatibel} would 
be applied to maps $v$ which map $\mc{Y}(i')$ into $\mc{Y}(i)$, in 
which case correlation implies strict monotonicity. Some algebraic 
structures will however give rise to ``degenerate''Êmaps which cannot 
be strictly monotone on the whole of $\mc{Y}(i')$, and at least in the 
notable case of path algebras (where the product of two monomials can 
be zero, see Subsection~\ref{Ssec:Stigalgebra}) the additional 
precondition that \(v(\mu) \in \mc{Y}(i)\) provides a convenient 
loophole to avoid getting caught by this technicality.

The final nontrivial condition in 
Lemma~\ref{L:OrdnFramflyttKompatibel} is that $v$ should map elements 
on the form $r(\mu)$ into $\Cspan\bigl( \bigl\{ v(\mu) \bigr\} \bigr)$. 
The most common reason this condition would be fulfilled is that 
$\mc{M}(i')$ and $\mc{M}(i)$ are both $\mc{R}$-modules for some ring 
$\mc{R}$ such that the advanceable map \(v\colon \mc{M}(i') \Fpil 
\mc{M}(i)\) is $\mc{R}$-linear, while $R(i')$ and $R(i)$ are the sets 
of actions of elements of $\mc{R}$ on $\mc{M}(i')$ and $\mc{M}(i)$ 
respectively. Note, however, that the general framework makes no 
assumption that maps in $R(i')$ should have counterparts in $R(i)$, 
or even that there should be a corresponding endomorphism on 
$\cmplM(i)$. This is a reason why advanceability is a condition on 
how arbitrary elements of $\RstarY$ are treated; for suitably linear 
maps it is sufficient to check advanceability on elements of $\mc{Y}$.

Whether it could be useful to have $R(i')$ and $R(i)$ generate 
nonisomorphic rings of endomorphisms on $\mc{M}(i')$ and $\mc{M}(i)$ 
respectively remains to be seen, but not requiring 
$\mc{R}$-linearity turns out to be advantageous for the related 
concept of \emph{biadvanceability}. 
(Shadow ambiguities are not the only ones that are traditionally 
discarded; equally important are ambiguities where the parts 
are disjoint. Biadvanceability provides a 
way to give an abstract definition of this.)
Recall that an $\mc{R}$-bilinear 
map $w$ satisfies \(w(ra,b) = r w(a,b) = w(a,rb)\), from which 
follows \(rs w(a,b) = r w(a,sb) = w(ra,sb) = sw(ra,b) = sr w(a,b)\), 
for all \(r,s \in \mc{R}\). If $\mc{R}$ is commutative this is only 
natural, but if $\mc{R}$ is a noncommutative ring then it places a 
rather severe restriction on the range of $w$: only elements at which 
all $\mc{R}$-module actions commute are allowed! This would often be 
insufficient for the intended uses of biadvanceable maps.

A practical compromise that is sometimes available is to request 
some weak form of bilinearity. It might for example be the 
case that the identity \(w(ra,b) = r w(a,b) = w(a,rb)\) only holds for 
monomial $a$ and $b$. It could also be the case that the identity is 
relaxed to \(w(ra,b) = r' w(a,b)\) (and similarly for moving out from 
the second position), where \(r' \in \mc{R}\) need not be equal to $r$ 
and may depend on $a$ or $b$. The condition needed for 
Lemma~\ref{L:OrdnFramflyttKompatibel} is weaker still\Ldash roughly 
that \(w(ra,b), w(a,rb) \in \Cspan\bigl(\bigl\{ w(a,b) 
\bigr\}\bigr)\) for \(a,b \in \mc{Y}\); 
see~\eqref{Eq:MontageAmbiguity}\Dash and should therefore not be a 
problem to fulfil when necessary.

\begin{definition} \label{Def:Montage}
  Let \(i,i_1,i_2 \in I\) be sorts. A map \(w\colon \cmplM(i_1) 
  \times \cmplM(i_2) \Fpil \cmplM(i)\) is said to be a 
  \DefOrd{bihomomorphism} if
  \begin{equation*}
    w(a_1-b_1,a_2-b_2) = 
    w(a_1,a_2) - w(b_1,a_2) - w(a_1,b_2) + w(b_1,b_2)
  \end{equation*}
  for all \(a_1,b_1 \in \cmplM(i_1)\) and \(a_2,b_2 \in 
  \cmplM(i_2)\). 
  A bihomomorphism \(w\colon \cmplM(i_1) \times \cmplM(i_2) \Fpil 
  \cmplM(i)\) is said to be 
  \DefOrd{biadvanceable}\index{advanceable!bi-}
  (with respect to $T(S)(i)$, $T(S)(i_1)$, and $T(S)(i_2)$) if 
  \begin{itemize}
    \item
      there for every \(t_1 \in T_1(S)(i_1)\), \(b_1 \in 
      \RstarY(i_1)\), and \(b_2 \in \RstarY(i_2)\) exists some 
      \(t \in T(S)(i)\) such that
      \(
        t\bigl( w( b_1, b_2) \bigr) = w\bigl( t_1(b_1), b_2 \bigr)
      \),
      and
    \item
      there for every \(t_2 \in T(S)(i_2)\), \(b_1 \in 
      \RstarY(i_1)\), and \(b_2 \in \RstarY(i_2)\) exists some \(t \in 
      T(S)(i)\) such that
      \(
        t\bigl( w( b_1, b_2) \bigr) = w\bigl( b_1, t_2(b_2) \bigr)
      \).
  \end{itemize}
  
  An ambiguity $(t,\mu,u)$ of $T_1(S)(i)$ is said to be a 
  \DefOrd{montage}\index{ambiguity!montage} of the 
  \DefOrd[*{piece}]{pieces} 
  \((\lambda,t') \in \mc{Y}(i_1) \times T(S)(i_1)\) and \((\nu,u') 
  \in \mc{Y}(i_2) \times T(S)(i_2)\) if there exists a continuous 
  biadvanceable map \(w\colon \mc{M}(i_1) \times \mc{M}(i_2) \Fpil 
  \mc{M}(i)\) such that
  \begin{align*}
    \mu ={}& w(\lambda,\nu)\text{,}\\ 
    t\bigl( w(\lambda,\nu) \bigr) ={}& w\bigl( t'(\lambda), \nu\bigr)
      \\
    u\bigl( w(\lambda,\nu) \bigr) ={}& w\bigl( \lambda, u'(\nu) \bigr)
  \end{align*}
  The map $w$ is called the \DefOrd{composition map} of this montage. 
  
  Let $V_1$ be a set of maps \(\cmplM(i_1) \Fpil \cmplM(i)\) and let 
  $V_2$ be a set of maps \(\cmplM(i_2) \Fpil \cmplM(i)\). A 
  biadvanceable map \(w\colon \cmplM(i_1) \times \cmplM(i_2) \Fpil 
  \cmplM(i)\) is said to be 
  \DefOrd[*{biadvanceable}]{$(V_1,V_2)$-biadvanceable} if 
  \(w(\cdot,\rho) \in V_1\) for all \(\rho \in \mc{Y}(i_2)\) and 
  \(w(\rho,\cdot) \in V_2\) for all \(\rho \in \mc{Y}(i_1)\).
  A \DefOrd[*{ambiguity!montage}]{$(V_1,V_2)$-montage ambiguity} is a 
  montage ambiguity where the composition map is 
  $(V_1,V_2)$-biadvanceable.
\end{definition}

The idea formalised by the montage ambiguity concept is to recognise 
the situation that the two pieces act on disjoint parts of 
$\mu$\Dash the pieces are like two small windows to completely 
different gardens that have been embedded into a large mural provided 
by the composition map, and the presence of an embedding cannot change 
the fact that the games that can be played in one garden are quite 
independent of what happens in the other. In Gr\"obner basis theory 
this idea~\cite{Buchberger-avh} is known as \emph{Buchberger's First 
Criterion} for eliminating useless critical 
pairs, although the identification is perhaps not obvious at this 
stage. The correspondence will be made clear in 
Corollary~\ref{Kor:KomGrobner}, however.

In the \(i=i_1=i_2=1\) case of \(\mc{M}(1) = \RavX\) and 
\(\mc{Y}(1) = X^*\), with all maps of the form \(b \mapsto 
\nu_1 b \nu_2\) for some \(\nu_1,\nu_2 \in X^*\) being 
advanceable, the typical form of a biadvanceable map is \(w(a,b) = 
\nu_1 a \nu_2 b \nu_3\) for some \(\nu_1,\nu_2,\nu_3 \in X^*\). 
Such `multiplication with fixed extra factors' maps can be used to 
produce a great variety of biadvanceable maps, and the construction 
does not require the multiplication operation to be associative, 
or even binary; pretty much anything that can be composed from 
multilinear operations on and fixed elements of $\mc{Y}$ will probably 
turn out to be biadvanceable, if simple reductions are constructed 
by putting every rule in every possible context. The underlying idea 
for making a biadvanceable map is however to take an element of 
$\mc{Y}$ and cut out two disjoint pieces from it\Dash the biadvanceable 
map then consists of inserting the two arguments into these two holes. 
How such a map may be interpreted depends very much on the underlying 
algebraic structure, but for the diamond lemma machinery it is 
sufficient that the biadvanceable maps exist.

\begin{lemma} \label{L:MontageAmbiguity}
  Let $(t,\mu,u)$ be an ambiguity of $T_1(S)(i)$ that is a montage of 
  the pieces \((\lambda,t') \in \mc{Y}(i_1) \times T_1(S)(i_1)\) and 
  \((\nu,u') \in  \mc{Y}(i_2) \times T_1(S)(i_2)\).  If $T(S)(j)$ is 
  compatible with some partial order $P(j)$ on $\mc{Y}(j)$ for all 
  \(j \in \{i,i_1,i_2\}\) and the composition map \(w\colon \mc{M}(i_1) 
  \times \mc{M}(i_2) \Fpil \mc{M}(i)\) of the montage satisfies
  \begin{equation} \label{Eq:MontageAmbiguity}
    w\Bigl( \DSM\bigl( \lambda, P(i_1) \bigr), \nu \Bigr) \cup
    w\Bigl( \lambda, \DSM\bigl( \nu, P(i_2) \bigr) \Bigr) \subseteq
    \DSM\bigl( w(\lambda,\nu), P(i) \bigr)
  \end{equation}
  then $(t,\mu,u)$ is resolvable relative to $P(i)$.
\end{lemma}
\begin{proof}
  The problem is to prove that
  \begin{multline*}
    t(\mu) - u(\mu) = 
    w\bigl( t'(\lambda), \nu \bigr) - w\bigl( \lambda, u'(\nu) \bigr) 
      = \\ =
    w\bigl( t'(\lambda), \nu -\nobreak u'(\nu) \bigr) + 
      w\bigl( t'(\lambda) -\nobreak \lambda, u'(\nu) \bigr)
    \in \DIS\bigl( \mu, P(i), S \bigr)
    \text{,}
  \end{multline*}
  and by symmetry it is sufficient to do the first of 
  $w\bigl( t'(\lambda), \nu -\nobreak u'(\nu) \bigr)$ and 
  $w\bigl( t'(\lambda) -\nobreak \lambda, u'(\nu) \bigr)$, as the 
  other is completely analogous.
  
  Let \(\ve \in \cmplO(i)\) be arbitrary. Let \(\ve_1 \in \cmplO(i_1)\) 
  be such that \(w\bigl( \ve_1, \nu -\nobreak u'(\nu) \bigr) \subseteq 
  \ve\). Since \(t'(\lambda) \in \DSM\bigl( \lambda, 
  P(i_1) \bigr)\) there exist \(\{\lambda_j\}_{j=1}^m \subseteq 
  \mc{Y}(i_1)\) and \(\{r_j\}_{j=1}^m \subseteq \pm R^*(i_1)\) such that 
  \(\lambda_j < \lambda \pin{P(i_1)}\) for all \(j=1,\dotsc,m\) and 
  \(\sum_{j=1}^m r_j(\lambda_j) \in t'(\lambda) + \ve_1\). Let \(a_j 
  = r_j(\lambda_j)\) for \(j=1,\dotsc,m\). By biadvanceability there 
  exist \(\{u_j\}_{j=1}^m \subseteq T(S)(i)\) such that \(u_j\bigl( 
  w(a_j,\nu) \bigr) = w\bigl( a_j, u'(\nu) \bigr)\) for all 
  \(j=1,\dotsc,m\). Since \(w(a_j,\nu) \in \DSM\bigl( \mu, P(i) 
  \bigr)\) by \eqref{Eq:MontageAmbiguity}, these satisfy
  \begin{align*}
    w\biggl( \sum_{j=1}^m a_j, \nu - u'(\nu) \biggr) ={}&
    \sum_{j=1}^m \Bigl( w(a_j,\nu) - w\bigl( a_j, u'(\nu) \bigr)
      \Bigr) = \\ ={}&
    \sum_{j=1}^m \Bigl( w(a_j,\nu) - u_j\bigl( w(a_j,\nu) \bigr)
      \Bigr) \in \displaybreak[0]\\ \in{}&
    \sum_{j=1}^m \setOf[\Big]{ a - u_j(a) }{ 
      a \in \DSM\bigl( \mu, P(i) \bigr)
    } \subseteq \\ \subseteq{}&
    \DIS\bigl( \mu, P(i), S \bigr)
  \end{align*}
  by Lemma~\ref{L:RelResolv}.
  Therefore \(w\bigl( t'(\lambda), \nu - u'(\nu) \bigr) \in 
  \DIS\bigl( \mu, P(i), S \bigr) + \ve\), and by the arbitrariness of 
  $\ve$ thus \(w\bigl( t'(\lambda), \nu - u'(\nu) \bigr) \in 
  \DIS\bigl( \mu, P(i), S \bigr)\), as claimed.
\end{proof}

The pieces are now in place for a definition of \emph{critical} as in 
`critical pair', i.e., ``member of a (small) set of ambiguities that 
together cover all ways in which things can fail to resolve''. 
The definition given is with respect to a particular family of 
advanceable maps, since this is how it will typically be applied: 
when someone considers only \emph{those} advanceable maps, then 
\emph{these} are the ambiguities that need to be explicitly checked. 
It is often natural to let the family of advanceable maps be (the set 
of morphisms in) a category, but there is no technical need for this.

\begin{definition} \label{Def:V-kritisk}
  Let a family \(V = \bigcup_{i,j \in I} V(i,j)\) of maps such that 
  every \(v \in V(i,j)\) is an advanceable continuous homomorphism 
  \(\cmplM(j) \Fpil \cmplM(i)\) be given.
  
  The family $V$ is said to be a \DefOrd{category} if \(V(i,i) \owns 
  \id\colon \cmplM(i) \Fpil \cmplM(i)\) and \(v_2 \circ v_1 \in 
  V(i,k)\) for all \(v_2 \in V(i,j)\), \(v_1 \in V(j,k)\), and 
  \(i,j,k \in I\). The family \(V = \bigcup_{i,j \in I} V(i,j)\) is the 
  \DefOrd[*{category!generated by}]{category generated by} 
  \(V_1 = \bigcup_{i,j \in I} V_1(i,j)\) if it is the smallest 
  category that satisfies \(V_1(i,j) \subseteq V(i,j)\) for all 
  \(i,j \in I\).
  
  An ambiguity $(t,\mu,u)$ of $T_1(S)(i)$ is said to be a 
  \DefOrd[*{ambiguity!shadow}]{$V$-shadow} of the ambiguity 
  $(t',\mu',u')$ of $T_1(S)(i')$ if there exists some \(v \in 
  V(i,i')\) such that \(\mu = v(\mu')\), \(t(\mu) = v\bigl( t'(\mu') 
  \bigr)\), and \(u(\mu) = v\bigl( u'(\mu') \bigr)\). If in addition 
  $(t',\mu',u')$ is not a $V$-shadow of $(t,\mu,u)$ then $(t,\mu,u)$ 
  is a \DefOrd[*{ambiguity!proper shadow}]{proper $V$-shadow} of 
  $(t',\mu',u')$. An ambiguity of $T_1(S)$ is said to be 
  \DefOrd[*{ambiguity!shadow-minimal}]{$V$-shadow-minimal} 
  if it is not a proper $V$-shadow of any ambiguity of $T_1(S)$. 
  An ambiguity of $T_1(S)$ is said to be 
  \DefOrd[*{ambiguity!shadow-critical}]{$V$-shadow-critical} if it 
  is not a proper $V$-shadow of any $V$-shadow-minimal ambiguity of 
  $T_1(S)$.
  
  An ambiguity $(t,\mu,u)$ of $T_1(S)(i)$ is said to be 
  \DefOrd[*{ambiguity!critical}]{$V$-critical} if it is 
  $V$-shadow-critical and is not a $\bigl( V(i,i_1), V(i,i_2) 
  \bigr)$-montage ambiguity for any \(i_1,i_2 \in I\).
\end{definition}

If $V$ is a category then the $V$-shadow relation $Q_V$\Ldash 
formally defined by \((t,\mu,u) \geqslant (t',\mu',u') \pin{Q_V}\) 
iff $(t,\mu,u)$ is a $V$-shadow of $(t',\mu',u')$\Dash is a 
quasi-order on the set of ambiguities. This point of view is 
instructive for understanding the definition of $V$-shadow-critical; 
$(t,\mu,u)$ is a proper $V$-shadow of $(t',\mu',u')$ iff
\((t,\mu,u) > (t',\mu',u') \pin{Q_V}\) and $(t,\mu,u)$ is 
$V$-shadow-minimal iff it is $Q_V$-minimal. A first stab at defining 
$V$-shadow-critical would be to use $V$-shadow-minimal, reasoning 
that non-minimal ambiguities need not be considered critical as 
there is always some smaller ambiguity of which they are a shadow, 
but this fails if $Q_V$ is not DCC; a simple example of a family $V$ 
for which this might occur is \(V = \{D^n\}_{n=0}^\infty\), where 
\(D(\ssI) = 0\) and \(D(\ssx^{n+1}) = \ssx^n\) for all \(n \in \N\). 
By only discarding those ambiguities which are proper shadows of a 
minimal ambiguity, one arrives at a concept which is as strong as 
minimality in the nice cases but is sufficient also in the strange 
cases.

\begin{theorem} \label{S:V-kritisk}
  Let \(V = \bigcup_{i,j \in I} V(i,j)\) be a family of maps such 
  that every \(v \in V(i,j)\) is an advanceable continuous 
  homomorphism \(\cmplM(j) \Fpil \cmplM(i)\). For each \(i \in I\), 
  let $P(i)$ be a partial order on $\mc{Y}(i)$ with which $T_1(S)(i)$ 
  is compatible. If every \(v \in V(i,j)\), for all \(i,j \in I\), 
  correlates $P(j)$ to $P(i)$ then the following claims are 
  equivalent:
  \begin{enumerate}
    \item[\parenthetic{a\textprime}] 
      Every ambiguity of $T_1(S)(i)$ is resolvable relative to 
      $P(i)$, for all \(i \in I\).
    \item[\parenthetic{a\textbis}] 
      Every $V$-critical ambiguity of $T_1(S)(i)$ is resolvable 
      relative to $P(i)$, for all \(i \in I\).
  \end{enumerate}
\end{theorem}
\begin{proof}
  Since the $V$-critical ambiguities of \parenthetic{a\textbis} are 
  included among the ambiguities of \parenthetic{a\textprime}, all 
  that needs to be shown is that the non-$V$-critical ambiguities are 
  resolvable relative to $P$ whenever the $V$-critical ambiguities are 
  so resolvable. Hence assume \parenthetic{a\textbis}.
  
  If an ambiguity $(t,\mu,u)$ of $T_1(S)(i)$ is not 
  $V$-shadow-critical, then by definition there exists some 
  $V$-shadow-minimal ambiguity, say $(t',\mu',u')$ of $T_1(S)(i')$, 
  of which $(t,\mu,u)$ is a proper $V$-shadow. 
  Since $(t',\mu',u')$ is minimal it is not a proper $V$-shadow of 
  any ambiguity, and hence $(t',\mu',u')$ is $V$-critical. By 
  \parenthetic{a\textbis}, $(t',\mu',u')$ is resolvable 
  relative to $P(i')$, which by Lemma~\ref{L:Rel.Res.Shadow} implies 
  that $(t,\mu,u)$ is resolvable relative to $P(i)$, as claimed.
  
  If an ambiguity $(t,\mu,u)$ of $T_1(S)(i)$ is a $\bigl( V(i,i_1), 
  V(i,i_2) \bigr)$-montage ambiguity for some \(i_1,i_2 \in I\) then 
  it is resolvable relative to $P(i)$ by 
  Lemma~\ref{L:MontageAmbiguity}; \eqref{Eq:MontageAmbiguity} holds 
  because $V(i,i_1)$ is a set of maps correlating $P(i_1)$ to $P(i)$, 
  $V(i,i_2)$ is a set of maps correlating $P(i_2)$ to $P(i)$, and
  the composition map $w$ of the ambiguity $(t,\mu,u)$ is 
  $\bigl( V(i,i_1), V(i,i_2) \bigr)$-biadvanceable.
\end{proof}

It should be observed that the set of $V$-critical ambiguities 
is not always the smallest set of ambiguities with which one can make 
do; if $(t,\mu,u)$ is $V$-critical then every $(t',\mu',u')$ such that 
\((t,\mu,u) \sim (t',\mu',u') \pin{Q_V}\) is $V$-critical as well, 
even though it is clearly sufficient to check one ambiguity in each 
$Q_V$-equivalence class. In actual calculations this often 
corresponds to being able to pick one labelling of an ambiguity and 
resolve it in that context, instead of having to write a resolution 
proof for arbitrary labellings. 

Another labour-saving trick which 
goes beyond the definition of $V$-critical ambiguity is 
\emph{Buchberger's Second Criterion}, which in its raw form is simply 
the observation that if three simple reductions $t_1$, $t_2$, and 
$t_3$ act nontrivially on the same $\mu$, then relative resolvability 
of two of the resulting ambiguities $(t_1,\mu,t_2)$, $(t_1,\mu,t_3)$, 
and $(t_2,\mu,t_3)$ implies the same for the third. Under mild extra 
assumptions on $V$, this criterion can be given the more traditional 
form that $(t_1,\mu,t_2)$ can be skipped if there exists some simple 
reduction $t_3$ such that $(t_1,\mu,t_3)$ and $(t_2,\mu,t_3)$ are both 
non-shadow-critical, since the latter two can then be assumed 
relatively resolvable on account of being shadows of other ambiguities. 
This criterion is of practical interest because it is often far less 
work to perform an explicit search for a matching $t_3$ than it is to 
explicitly resolve $(t_1,\mu,t_2)$, but it is not as theoretically 
important as the recognition of montage ambiguities (Buchberger's 
first criterion).

\begin{example} \label{Ex:BergmanskTvetydighet}
  Let \(I=\{1\}\) and drop sort indices. Let \(\mc{M} = \RavX\), 
  \(\mc{Y} = X^*\), \(R = \mc{R}\), the topology be discrete, \(V = 
  \left\{ b \mapsto \nu_1 b \nu_2 \right\}_{\nu_1,\nu_2 \in X^*}\), 
  and the simple reductions be defined as in 
  Corollary~\ref{Kor:t_v,s-reduktion} (or equivalently 
  Construction~\ref{Konstr:t_v,s-reduktion}); this is the setting 
  for Bergman's diamond lemma. Which are then the $V$-critical 
  ambiguities?
  
  Using the notation of \eqref{Eq1:Trad.def.T_1(S)}, an ambiguity has 
  the form $(t_{\lambda_1 s_1 \nu_1}, \mu, t_{\lambda_2 s_2 \nu_2})$ 
  where \(\lambda_1 \mu_{s_1} \nu_1 = \mu = \lambda_2 \mu_{s_2} \nu_2\). 
  By unique factorisation of $\mu$ in $X^*$, $\lambda_1$ is a prefix 
  (left divisor) of $\lambda_2$ or vice versa\Dash hence there 
  exist \(\kappa \in \{\lambda_1,\lambda_2\}\) and 
  \(\lambda_1',\lambda_2' \in X^*\) such that \(\lambda_1 = \kappa 
  \lambda_1'\) and \(\lambda_2 = \kappa \lambda_2'\). Similarly 
  $\nu_1$ is a suffix (right divisor) of $\nu_2$ or vice versa, 
  whence there exist \(\rho \in \{\nu_1,\nu_2\}\) and \(\nu_1',\nu_2' 
  \in X^*\) such that \(\nu_1 = \nu_1' \rho\) and \(\nu_2 = \nu_2' 
  \rho\). It follows that $(t_{\lambda_1 s_1 \nu_1}, \mu, 
  t_{\lambda_2 s_2 \nu_2})$ is a shadow under \(v(b) = \kappa b 
  \rho\) of $(t_{\lambda_1' s_1 \nu_1'}, \mu', 
  t_{\lambda_2' s_2 \nu_2'})$ where \(\mu' = 
  \lambda_1' \mu_{s_1} \nu_1' = \lambda_2' \mu_{s_2} \nu_2'\), and 
  this shadow is proper unless \(\kappa = \rho = \ssI\). Conversely 
  any ambiguity $(t_{\lambda_1 s_1 \nu_1}, \mu, 
  t_{\lambda_2 s_2 \nu_2})$ which is a proper $V$-shadow must have 
  \(\lambda_1,\lambda_2 \neq \ssI\) or \(\nu_1,\nu_2 \neq \ssI\), so 
  it follows that the constructed $(t_{\lambda_1' s_1 \nu_1'}, \mu', 
  t_{\lambda_2' s_2 \nu_2'})$ is $V$-shadow-minimal. Hence a 
  $V$-shadow-critical ambiguity $(t_{\lambda_1 s_1 \nu_1}, \mu, 
  t_{\lambda_2 s_2 \nu_2})$ has \(\ssI \in \{\lambda_1,\lambda_2\}\) 
  and \(\ssI \in \{\nu_1,\nu_2\}\). Without loss of generality it may 
  be assumed that \(\lambda_1 = \ssI\).
  
  If the ambiguity $(t_{\ssI s_1 \nu_1}, \mu, t_{\lambda_2 s_2 
  \nu_2})$ is such that $\mu_{s_1}$ is a prefix of $\lambda_2$, say 
  \(\lambda_2 = \mu_{s_1} \tau\), then conversely $\mu_{s_2} \nu_2$ 
  is a suffix of \(\nu_1 = \tau \mu_{s_2} \nu_2\), and consequently 
  $(t_{\ssI s_1 \nu_1}, \mu, t_{\lambda_2 s_2 \nu_2})$ is a montage 
  with the composition map \(w(b_1,b_2) = b_1 \tau b_2\) of the pieces 
  $(\mu_{s_1}, t_{\ssI s_1 \ssI})$ and $(\mu_{s_2}\nu_2, 
  t_{\ssI s_2 \nu_2})$. Hence a $V$-critical ambiguity 
  $(t_{\ssI s_1 \nu_1}, \mu, t_{\lambda_2 s_2 \nu_2})$ rather has 
  \(\mu = \lambda_2 \tau \nu\), where 
  either (if \(\nu_1 = \ssI\): an \emph{inclusion} ambiguity) \(\tau = 
  \mu_{s_2}\), \(\nu = \nu_2\), and \(\mu_{s_1} = \lambda_2 \tau \nu\) 
  or (if \(\nu_1 \neq \ssI\): an \emph{overlap} ambiguity) 
  \(\mu_{s_1} = \lambda_2 \tau\), \(\mu_{s_2} = \tau \nu\), and \(\nu 
  = \nu_1\). Either way, the ambiguity is uniquely identified by the 
  quintuplet $(s_1,s_2,\lambda_2,\tau,\nu)$, which (as it happens) 
  is the definition of ambiguity that was used in~\cite{Bergman}. 
  This has shown that all $V$-critical ambiguities are among those 
  specified by Bergman and consequently Bergman's diamond lemma 
  follows from combining Theorems~\ref{S:CDL} and~\ref{S:V-kritisk}.
  
  It may also be observed that if the set $S$ of rules is finite then 
  the set of $V$-critical ambiguities is finite as well; for any 
  given pair \((s_1,s_2) \in S^2\), a non-montage ambiguity on the 
  form $(t_{\ssI s_1 \nu_1}, \mu, t_{\lambda_2 s_2 \nu_2})$ where 
  at least one of $\nu_1$ and $\nu_2$ is equal to $\ssI$ must satisfy 
  \(\deg \lambda_2 < \deg \mu_{s_1}\), and hence the number of 
  $V$-critical ambiguities on this form can be at most 
  $\deg \mu_{s_1}$. This gives the overall bound $\norm{S}^2 
  \max_{s \in S} \deg \mu_s$, or sharper $\norm{S} \sum_{s \in S} 
  \deg\mu_s$, for the number of $V$-critical ambiguities, but the 
  actual number is often much lower.
\end{example}

Note how the finiteness of the set of $V$-critical ambiguities 
whenever $S$ is finite requires that the montage ambiguities are 
discarded. This is not the case in the commutative counterpart 
\(\mc{M} = \mc{R}[X]\), since in that case all montage ambiguities 
are shadows of the one with \(w(a,b) = ab\). This is probably the 
reason that this important principle in the commutative theory is 
merely known as the ``first criterion''.

As was parenthetically remarked in the example, Bergman distinguishes 
between \emph{inclusion} and \emph{overlap} ambiguities, where the 
former have the property that $\mu_{s_2}$ divides (is a subword of) 
$\mu_{s_1}$. Since the property of being a divisor can be expressed 
in terms of advanceable maps, these classes may be defined also in 
the more abstract setting.

\begin{definition}
  An ambiguity $(t,\mu,u)$ of $T_1(S)(i)$ is said to be an 
  \DefOrd[*{ambiguity!inclusion}]{inclusion ambiguity}, 
  where $t$ is called the \DefOrd{inner reduction} and $u$ is called 
  the \DefOrd{outer reduction}, if there for every advanceable map 
  \(v\colon \cmplM(i') \Fpil \cmplM(i)\) and \((\mu',u') \in 
  \mc{Y}(i') \times T_1(S)(i')\) such that 
  \(\mu = v(\mu')\) and \(u(\mu) = v\bigl( u'(\mu') \bigr)\) exists 
  some \(t' \in T_1(S)(i')\) such that \(t(\mu) = v\bigl( t'(\mu') 
  \bigr)\). The inclusion is said to be 
  \DefOrd[*{ambiguity!proper inclusion}]{proper} if only one of 
  the reductions fit the definition for being the inner reduction.
  
  An ambiguity is said to be an 
  \DefOrd[*{ambiguity!overlap}]{overlap ambiguity} if it is 
  neither a montage ambiguity nor an inclusion ambiguity.
\end{definition}

For handmade sets of simple reductions $T_1(S)$ (or ditto sets of 
rules $S$ from which they are made), inclusion ambiguities are rare, 
because they typically mean the outer reduction is redundant and can 
be dropped without changing $\Irr(S)$ (by Theorem~\ref{S:Konstr.Irr(S)}), 
$\Red(S)$, or $\tS$ (by Lemma~\ref{L:Utgl.T(S)}). The situation is a 
bit different in sets of reductions that are automatically generated 
by some completion procedure, since it is very common that special 
cases of a rule are derived before (and even used in deriving) the 
more general rule that one may find in the literature. Relying on 
Lemma~\ref{L:Utgl.T(S)} for simplifying the set of reductions would 
require keeping all reductions until a complete set is found and only 
then drop those which are redundant, but it is usually more practical 
to drop them as soon as the inclusion is discovered. The next theorem 
gives conditions for this.

\begin{theorem} \label{S:DropRule}
  For every \(i \in I\), let \(T_1(S')(i) \subseteq T_1(S)(i)\) and a 
  partial order $P(i)$ on $\mc{Y}(i)$ with which $T_1(S)(i)$ is 
  compatible be given. Assume there is some \(i_0 \in I\) and \(t_0 
  \in T_1(S)(i_0) \setminus T_1(S')(i_0)\) such that there for every 
  \(i \in I\), \(t \in T_1(S)(i) \setminus T_1(S')(i)\), and \(\mu \in 
  \mc{Y}(i)\) on which $t$ acts nontrivially exists a continuous 
  homomorphism \(v\colon \cmplM(i_0) \Fpil \cmplM(i)\) which is 
  advanceable with respect to $T_1(S')(i_0)$ and $T_1(S')(i)$ and 
  also some \(\mu_0 \in \mc{Y}(i_0)\) such that \(v(\mu_0)=\mu\), 
  \(t(\mu) = v\bigl( t_0(\mu_0) \bigr)\), and \(v\bigl( 
  \DSM(\mu_0,P(i_0)) \bigr) \subseteq \DSM(\mu,P(i))\).
  
  If for all \(i \in I\) all ambiguities of $T_1(S')(i)$ are resolvable 
  with respect to $P(i)$ and there for every \(\mu_0 \in \mc{Y}(i_0)\) 
  on which $t_0$ acts nontrivially exists some \(u_0 \in T_1(S')(i_0)\) 
  such that \(t_0(\mu_0) - u_0(\mu_0) \in \DIS\bigl( \mu_0, P(i_0), 
  S' \bigr)\), then for all \(i \in I\) all ambiguities of $T_1(S)(i)$ 
  are resolvable with respect to $P(i)$ and \(\mc{I}(S')(i) = 
  \mc{I}(S)(i)\).
\end{theorem}
\begin{proof}
  Let \(i \in I\) be arbitrary. Let $(t,\mu,u)$ be an arbitrary 
  ambiguity of $T_1(S)(i)$. If \(t,u \in T_1(S')(i)\) then \(t(\mu) - 
  u(\mu) \in \DIS\bigl(\mu,P(i),S'\bigr) \subseteq 
  \DIS\bigl(\mu,P(i),S\bigr)\) by assumption. 
  
  If \(t \in T_1(S)(i) \setminus T_1(S')(i)\) and \(u \in T_1(S')(i)\) 
  then by assumption there exists some \(\mu_0 \in \mc{Y}(i_0)\) and 
  an advanceable continuous homomorphism \(v\colon \cmplM(i_0) \Fpil 
  \cmplM(i)\) such that \(v\bigl( \DSM(\mu_0,P(i_0)) \bigr) \subseteq 
  \DSM(\mu,P(i))\), \(\mu = v(\mu_0)\), and \(t(\mu) = v\bigl( 
  t_0(\mu_0) \bigr)\). Since \(v\bigl( t_0(\mu_0) \bigr) = t(\mu) \neq 
  \mu = v(\mu_0)\) it follows that $t_0$ acts nontrivially on 
  $\mu_0$, and hence there exists some \(u_0 \in T_1(S')(i_0)\) such 
  that \(t_0(\mu_0) \equiv u_0(\mu_0) \pmod{S' < \mu_0 \pin{P(i_0)}}\). 
  By the second claim of Lemma~\ref{L:Rel.Res.Shadow}, 
  \(v\bigl( t_0(\mu_0)\bigr) \equiv v\bigl( u_0(\mu_0) \bigr) 
  \pmod{S' < \mu \pin{P(i)}}\), and by advanceability there exists 
  some \(u_1 \in T(S')(i)\) such that \(u_1(\mu) = v\bigl( u_0(\mu_0) 
  \bigr)\); in other words \(t(\mu) \equiv u_1(\mu) 
  \pmod{S' < \mu \pin{P(i)}}\).
  
  It need not be the case that \(u_1 = u\), or even that \(u_1 \in 
  T_1(S')(i)\), but typically \(u_1(\mu) \neq \mu\) and then there 
  exist \(u_{1a} \in T_1(S')(i)\) and \(u_{1b} \in T(S')(i)\) such 
  that \(u_1(\mu) = u_{1b}\bigl( u_{1a}(\mu) \bigr)\) and 
  \(u_{1a}(\mu) \neq \mu\). In this case, $(u,\mu,u_{1a})$ is an 
  ambiguity of $T_1(S')(i)$ and \(u(\mu) \equiv u_{1a}(\mu) 
  \pmod{S' < \mu \pin{P(i)}}\) since it is resolvable. Furthermore 
  \(u_{1a}(\mu) \equiv u_1(\mu) \pmod{S' < \mu \pin{P(i)}}\) 
  by Lemma~\ref{L:RelResolv}, and it follows that $(t,\mu,u)$ is 
  resolvable relative to $P(i)$. In the degenerate case that 
  \(u_1(\mu) = \mu\), one has the curious situation that \(\mu = 
  v\bigl( u_0(\mu_0) \bigr) \in v\bigl( \DSM(\mu_0,P(i_0)) \bigr) 
  \subseteq \DSM\bigl( \mu, P(i) \bigr)\); there must be an 
  alternative expression for $\mu$ as a linear combination of 
  strictly smaller elements of $\mc{Y}(i)$. Hence \(\mu - u(\mu) 
  \in \DIS\bigl( \mu, P(i), S' \bigr)\) by definition and therefore 
  \(t(\mu) \equiv u_1(\mu) = \mu \equiv u(\mu) 
  \pmod{S' < \mu \pin{P(i)}}\).
  
  The case \(t \in T_1(S')(i)\) and \(u \in T_1(S)(i) \setminus 
  T_1(S')(i)\) is handled similarly. The case \(t,u \in T_1(S)(i) 
  \setminus T_1(S')(i)\) is handled by combining the two previous 
  cases\Dash having two advanceable maps \(v,v'\colon \cmplM(i_0) \Fpil 
  \cmplM(i)\) and \(\mu_0,\mu_0' \in \mc{Y}(i_0) \) such that 
  \(v(\mu_0) = \mu = v'(\mu_0')\), \(t(\mu) = v\bigl( t_0(\mu_0) 
  \bigr)\), and \(u(\mu) = v'\bigl( t_0(\mu_0') \bigr)\).
  
  As for the claim that \(\mc{I}(S')(i) = \mc{I}(S)(i)\), it follows 
  from Lemma~\ref{L:Slutenhet,I'(S)} that
  \begin{align*}
    \mc{I}(S')(i) ={}&
    \Cspan\Bigl( \setOf[\big]{ \mu - t(\mu) }{ 
      \mu \in \mc{Y}(i), t \in T_1(S')(i)} \Bigr) 
      \subseteq \\ \subseteq{}&
    \Cspan\Bigl( \setOf[\big]{ \mu - t(\mu) }{ 
      \mu \in \mc{Y}(i), t \in T_1(S)(i)} \Bigr) = 
    \mc{I}(S)(i) \text{.}
  \end{align*}
  Furthermore, if \(t \in T_1(S)(i) \setminus T_1(S')(i)\) and \(\mu 
  \in \mc{Y}(i)\) can give a nonzero contribution to the second 
  $\Cspan$, i.e., if they are such that \(\mu - t(\mu) \neq 0\), then 
  by assumption there exist \(\mu_0 \in \mc{Y}(i_0)\) and an 
  advanceable continuous homomorphism \(v\colon \cmplM(i_0) \Fpil 
  \cmplM(i)\) such that \(\mu = v(\mu_0)\), \(t(\mu) = 
  v\bigl( t_0(\mu_0) \bigr)\), and \(v\bigl( \DSM(\mu_0,P(i_0)) \bigr) 
  \subseteq \DSM(\mu,P(i))\). There also exists some \(u_0 \in 
  T_1(S')(i_0)\) such that \(t_0(\mu_0) \equiv u_0(\mu_0) \pmod{S' < 
  \mu_0 \pin{P(i_0)}}\) and some \(u \in T(S')(i)\) such that 
  \(u\bigl( v(\mu_0) \bigr) = v\bigl( u_0(\mu_0) \bigr)\). Hence
  \begin{align*}
    \mu - t(\mu) = 
    v\bigl( \mu_0 - t_0(\mu_0) \bigr) ={}&
    v\bigl( \mu_0 - u_0(\mu_0) \bigr) + 
      v\bigl( u_0(\mu_0) - t_0(\mu_0) \bigr) \in \\ \in{}&
    \bigl( \mu - u(\mu) \bigr) + 
      v\Bigl( \DIS\bigl( \mu_0,P(i_0), S'\bigl) \Bigr) 
      \subseteq \\ \subseteq{}&
    \mc{I}(S')(i) + \DIS\bigl( \mu,P(i), S'\bigl) =
    \mc{I}(S')(i)
  \end{align*}
  and thus \(\mc{I}(S)(i) \subseteq \mc{I}(S')(i)\).
\end{proof}

\section{A framework construction}
\label{Sec:Konstruktion}

In the last couple of sections, the generic theory has been developed 
to a point where it is comparable to the ring theory diamond lemma 
\emph{provided} that one can set up the necessary framework. 
The ideas behind the standard 
construction have already been presented, but it is convenient to 
collect everything in a formal statement to facilitate citations 
in other papers. Furthermore the last couple of sections have 
demonstrated that one typically wants a bit more than just the basic 
framework assumptions, so in support of the framework construction 
there are also some lemmas which give more elementary conditions that 
suffice for establishing the advanceability, compatibility, and 
equicontinuity properties.

The first lemma concerns the basic construction of a topology. It 
serves mainly as a preparation for Lemma~\ref{L2:Likgradig} and may 
certainly be skipped if one is only interested in a discrete 
topology. For simplicity, it is stated in single-sorted notation.

\begin{lemma} \label{L:Modul-ultranorm}
  Let $\mc{R}$ be an associative unital ring with ring ultranorm 
  $\norm{\cdot}$ which is complete in the topology induced by this 
  norm. Let some nonempty set $Y$ and a function \(U\colon Y \Fpil \Rp\) 
  be given. Let $\mc{M}$ be the free $\mc{R}$-module with basis $Y$. 
  For every \(\mu \in Y\), let \(f_\mu\colon \mc{M} \Fpil \mc{R}\) be 
  the coefficient-of-$\mu$ homomorphism, i.e., \(f_\mu(\mu)=1\) for 
  all \(\mu \in Y\) and \(f_\mu(\nu) = 0\) for all \(\mu,\nu \in Y\) 
  such that \(\mu \neq \nu\). Define
  \begin{equation} \label{Eq:Modul-ultranorm}
    \Norm{a} := \max_{\mu \in Y} \norm[\big]{f_\mu(a)} U(\mu)
  \end{equation}
  for all \(a \in \mc{M}\). Then $\Norm{\cdot}$ is an $\mc{R}$-module 
  ultranorm on $\mc{M}$. Let
  \begin{align*}
    R ={}& \setOf[\big]{ a \mapsto r \cdot a : \mc{M} \Fpil \mc{M} }{ 
    r \in \mc{R}, \norm{r} \leqslant 1 } \text{,}\\
    B_n ={}& \setOf{ a \in \mc{M} }{ \Norm{a} < 2^{1-n} }
    \quad\text{for \(n \geqslant 1\).}
  \end{align*}
  Then \(\mc{O} = \{B_n\}_{n=1}^\infty\) satisfies 
  Assumption~\ref{Ant:B_n} and every $f_\mu$ for \(\mu \in Y\) is 
  continuous. The extensions of $\{f_\mu\}_{\mu \in Y}$ and 
  $\Norm{\cdot}$ to $\cmplM$ by continuity satisfy 
  \eqref{Eq:Modul-ultranorm} and
  \begin{equation} \label{Eq2:Modul-ultranorm}
    \Norm[\big]{ f_\mu(a) \cdot \mu } \leqslant \Norm{a}
    \qquad\text{for all \(\mu \in Y\),}
  \end{equation}
  for all \(a \in \cmplM\).
\end{lemma}
\begin{proof}
  That the right hand side of \eqref{Eq:Modul-ultranorm} exists for 
  all \(b \in \mc{M}\) follows from the fact that $\setOf[\big]{ \mu 
  \in Y }{ f_\mu(b) \neq 0 }$ is finite for every \(b \in \mc{M}\). 
  The claim that $\Norm{\cdot}$ is an $\mc{R}$-module ultranorm is 
  shown by verifying the conditions in 
  Definition~\ref{Def:Ultranorm}. $\Norm{a}$ is nonnegative because 
  it is a maximum of nonnegative numbers. Since $\norm{\cdot}$ 
  satisfies the strong triangle inequality, 
  \begin{multline*}
    \Norm{a-b} = 
    \max_{\mu \in Y} \norm[\big]{ f_\mu(a-b) } U(\mu) \leqslant
    \max_{\mu \in Y} \max\Bigl\{ \norm[\big]{f_\mu(a)}, 
      \norm[\big]{f_\mu(b)} \Bigr\} U(\mu) = \\ =
    \max\Bigl\{  \max_{\mu \in Y} \norm[\big]{f_\mu(a)} U(\mu), 
      \max_{\mu \in Y} \norm[\big]{f_\mu(b)} U(\mu) \Bigr\} =
    \max\bigl\{ \Norm{a}, \Norm{b} \bigr\}
  \end{multline*}
  for all \(a,b \in \mc{M}\). \(\Norm{a} = 0\) iff 
  \(\norm[\big]{f_\mu(a)} = 0\) for all \(\mu \in \mc{Y}\), which 
  holds iff \(f_\mu(a) = 0\) for all \(\mu \in \mc{Y}\), which in 
  turn is true iff \(a=0\). For every \(r \in \mc{R}\) and \(a \in 
  \mc{M}\),
  \begin{multline} \label{Eq3:Modul-ultranorm}
    \Norm{r \cdot a} =
    \max_{\mu \in Y} \norm[\big]{ f_\mu(r \cdot a) } U(\mu) = 
    \max_{\mu \in Y} \norm[\big]{ r \cdot f_\mu(a) } U(\mu) 
      \leqslant \\ \leqslant 
    \max_{\mu \in Y} \norm{r} \norm[\big]{ f_\mu(a) } U(\mu) =
    \norm{r} \Norm{a}
    \text{.}
  \end{multline}
  Hence $\Norm{\cdot}$ is an $\mc{R}$-module ultranorm. As such, it 
  is also a uniformly continuous function \(\mc{M} \Fpil 
  [0,\infty\mathclose{[} \subset \R\) since the absolute value of 
  $\Norm{a}-\Norm{b}$ by the triangle inequality is bounded from above 
  by $\Norm{a-b}$. By the completeness of the codomain, it follows that 
  $\Norm{\cdot}$ extends by continuity to a function \(\cmplM \Fpil 
  [0,\infty\mathclose{[}\), and this extended map will also be an 
  $\mc{R}$-module ultranorm because left and right hand sides in the 
  axioms for this are all continuous and thus the axioms are preserved 
  under taking limits.
  
  It follows from \eqref{Eq3:Modul-ultranorm} that any set of all \(a 
  \in \mc{M}\) such that \(\Norm{a} < e\) for some \(e \in \Rp\) is 
  an $R$-module. That \(\bigcap_{n=1}^\infty B_n = \{0\}\) is because 
  \(\Norm{a}=0\) implies \(a=0\). Let \(e \in \Rp\) and \(\mu \in Y\) 
  be given. Since \(\norm[\big]{f_\mu(a)} U(\mu) \leqslant \Norm{a}\) 
  for all \(a \in \mc{M}\), it follows that any \(a \in \mc{M}\) such 
  that \(\Norm{a} < e U(\mu)\) has \(\norm[\big]{f_\mu(a)} < e\), and 
  hence $f_\mu$ is continuous. By this continuity and the completeness 
  of $\mc{R}$, it extends to a map \(\cmplM \Fpil \mc{R}\). Furthermore
  \begin{multline*}
    \Norm[\big]{ f_\mu(a) \cdot \mu } = 
    \max_{\nu \in Y} \norm[\Big]{ f_\nu\bigl( f_\mu(a) \cdot \mu 
      \bigr) } U(\nu) =
    \max_{\nu \in Y} \norm[\Big]{ f_\mu(a) \cdot f_\nu(\mu) } U(\nu) 
      = \\ =
    \max\Bigl\{ 0, \norm[\big]{ f_\mu(a) } U(\mu) \Bigr\} =
    \norm[\big]{ f_\mu(a) } U(\mu) \leqslant \Norm{a}
  \end{multline*}
  which demonstrates \eqref{Eq2:Modul-ultranorm} for \(a \in \mc{M}\). 
  By continuity of left and right hand sides it continues to hold for 
  arbitrary \(a \in \cmplM\).
  
  On the matter of \eqref{Eq:Modul-ultranorm} for \(a \in \cmplM 
  \setminus \mc{M}\), one may first observe that \(\Norm{a} > 0\) and 
  thus there exists some \(b \in \mc{M}\) such that \(\Norm{a-b} < 
  \Norm{a}\). It follows from the strong triangle inequality that 
  \(\Norm{a} = \Norm{b} = \max_{\mu \in Y} 
  \norm[\big]{f_\mu(b)}U(\mu)\); let \(\mu_0 \in Y\) be an element in 
  which this maximum is attained. Since \(\norm[\big]{f_\mu(c)} 
  U(\mu) \leqslant \Norm{c}\) for all \(c \in \mc{M}\) and \(\mu \in 
  Y\), this inequality by continuity holds for all \(c \in \cmplM\). 
  Hence \(\norm[\big]{f_\mu(a-b)} U(\mu) \leqslant \Norm{a-b} < 
  \Norm{a}\) and thus
  \begin{multline*}
    \norm[\big]{f_\mu(a)} U(\mu) =
    \norm[\big]{f_\mu(b) + f_\mu(a-b)} U(\mu) 
      \leqslant \\ \leqslant
    \max\Bigl\{ \norm[\big]{f_\mu(b)} U(\mu), 
      \norm[\big]{f_\mu(a-b)} U(\mu) \Bigr\} \leqslant
    \Norm{a} \text{.}
  \end{multline*}
  On the other hand,
  \begin{multline*}
    \Norm{a} = 
    \norm[\big]{f_{\mu_0}(b)} U({\mu_0}) =
    \norm[\big]{ f_{\mu_0}(a) - f_{\mu_0}(a-b)} U({\mu_0}) 
      \leqslant \\ \leqslant
    \max\Bigl\{ \norm[\big]{f_{\mu_0}(a)} U(\mu_0), 
      \norm[\big]{f_{\mu_0}(a-b)} U(\mu_0) \Bigr\}
  \end{multline*}
  and since \(\norm[\big]{f_{\mu_0}(a-b)} U(\mu_0) < \Norm{a}\) it 
  follows that \(\norm[\big]{f_{\mu_0}(a)} U(\mu_0) = \Norm{a}\). 
  This has verified not only that the maximum in the right hand side 
  of \eqref{Eq:Modul-ultranorm} exists, but also that it equals the 
  left hand side.
\end{proof}

The main specialisation made in this construction is that every 
$\mc{M}(i)$ is a free $\mc{R}$-module, for some fixed ring $\mc{R}$. 
The topological conditions may seem extensive, but they are all 
fulfilled in the case considered in Lemma~\ref{L:Modul-ultranorm}, and 
they are of course void in the case of a discrete topology. The 
classical case is furthermore that \(R(i) = \mc{R}(i)\) and 
\(\mc{Y}(i) = Y(i)\), but as discussed in Section~\ref{Sec2:Basics}, 
things aren't always that simple.

\begin{construction} \label{Konstr:t_v,s-reduktion}
  Let $\mc{R}$ be a unital associative topologically complete ring. 
  For every \(i \in I\), let $Y(i)$ be an arbitrary set and let 
  $\mc{M}(i)$ be the free $\mc{R}$-module with basis $Y(i)$. For 
  every \(\mu \in Y(i)\), denote by $f_\mu$ the coefficient-of-$\mu$ 
  homomorphism \(\mc{M}(i) \Fpil \mc{R}\). Also let 
  $\mc{R}(i)$ be the ring of $\mc{R}$-actions on $\mc{M}(i)$, let 
  $R(i)$ be a subring of $\mc{R}(i)$, and let \(R^\bot(i) \subseteq 
  \mc{R}(i)\) be a set such that \(\sum_{r \in R^\bot(i)} R(i) \circ r = 
  \mc{R}(i)\) and \(\id \in R^\bot(i)\). Let \(\mc{Y}(i) = 
  \setOf[\big]{ r(\mu) }{ r \in R^\bot(i), \mu \in Y(i) }\); this 
  ensures Assumption~\ref{Ant:YSpan} is fulfilled.
  For every \(i \in I\), let \(\mc{O}(i) = \bigl\{ B_n(i) \bigr\}\) 
  be a family of $R(i)$-modules satisfying Assumption~\ref{Ant:B_n} 
  and in addition being such that 
  the $\mc{R}$-module multiplication \(\mc{R} \times \cmplM(i) \Fpil 
  \cmplM(i) : (r,b) \mapsto r \cdot b\) 
  and all maps \(\{f_\mu\}_{\mu \in Y(i)}\) are continuous. 
  
  Let \(V = \bigcup_{i,j \in I} V(i,j)\) be such that every 
  \(v \in V(i,j)\) is a continuous $\mc{R}$-linear map \(\cmplM(j) 
  \Fpil \cmplM(i)\). Let \(S = \bigcup_{i \in I} S(i)\) be arbitrary 
  such that \(S(i) \subseteq Y(i) \times \cmplM(i)\).
  For all \(i,j \in I\), let
  \begin{equation*}
    W(i,j) = \setOf[\Big]{ \bigl( v,(\mu,a) \bigr) \in V(i,j) 
    \times S(j) }{ v(\mu) \in Y(i) }
  \end{equation*}
  and define \(T_1(S)(i) = \bigcup_{j \in I} \{t_{v,s}\}_{(v,s) \in 
  W(i,j)}\)\index{t v s@$t_{v,s}$}, where 
  \begin{equation} \label{KonEq:t_v,s-reduktion}
    t_{v,(\mu,a)}(b) = b - f_{v(\mu)}(b) \cdot v(\mu-a)
    \qquad\text{for all \(b \in \cmplM(i)\).}
  \end{equation}
  This $T_1(S)(i)$ satisfies Assumption~\ref{Ant:T(S)-kont-hom} and 
  every \(r \in \mc{R}(i)\) is absolutely advanceable with respect 
  to $T_1(S)(i)$, for all \(i \in I\). Furthermore every ambiguity 
  $\bigl( t_1, r(\mu), t_2 \bigr)$ of $T_1(S)(i)$ where \(\mu \in 
  Y(i)\) and \(r \in R^\bot(i) \setminus \{\id\}\) is an absolute 
  shadow of the ambiguity $(t_1,\mu,t_2)$.
\end{construction}
\begin{proof}
  Let \(i \in I\) be given. That any \(t \in T_1(S)(i)\) is continuous 
  and $\mc{R}$-linear follows from \eqref{KonEq:t_v,s-reduktion} 
  since this formula is a composition of maps with these properties, 
  and thus $t$ is a continuous homomorphism satisfying \(t \circ r = 
  r \circ t\) for all \(r \in \mc{R}(i)\). Not only does this satisfy 
  Assumption~\ref{Ant:T(S)-kont-hom}, but it also means every \(r \in 
  \mc{R}(i)\) is absolutely advanceable.
\end{proof}

\begin{lemma} \label{L:Framflyttbar}
  Let everything be as in the construction. Assume in addition that 
  $V$ is closed under composition and satisfies \(v\bigl( Y(j) \bigr) 
  \subseteq Y(i) \cup \{0\}\) for all \(v \in V(i,j)\) and \(i,j \in I\). 
  If $S$ is such that \(v(\mu)=0\) implies \(v(a) = 0\) for \(v \in V(i,j)\), 
  \((\mu,a) \in S(j)\), and \(i,j \in I\) then for all \(i,j \in I\) 
  every element of $V(i,j)$ is advanceable with respect to $T_1(S)(j)$ 
  and $T_1(S)(i)$.
\end{lemma}
\begin{proof}
  Let \(i,j \in I\), \(w \in V(i,j)\) such that \(w\bigl( Y(j) 
  \bigr) \subseteq Y(i) \cup \{0\}\), \(t' \in T_1(S)(j)\), and \(b 
  \in \RstarY(j)\) be given. By definition of $T_1(S)(j)$ there 
  exists some \(k \in I\) and \(\bigl( v, (\mu,a) \bigr) \in W(j,k)\) 
  such that \(t' = t_{v,(\mu,a)}\). Furthermore there exist \(\nu \in 
  Y(j)\) and \(r \in \mc{R}\) such that \(b = r \cdot \nu\). There 
  are three cases for \(v\bigl( t'(b) \bigr)\).
  \begin{enumerate}
    \item
      If \(v(\mu) \neq \nu\) then \(f_{v(\mu)}(b) = 0\) and hence 
      \(t'(b) = b\), in which case \(t\bigl( w(b) \bigr) = w\bigl( 
      t'(b) \bigr)\) for \(t = \id\).
    \item
      If \(v(\mu) = \nu\) and \(w(\nu) \in Y(i)\) then 
      \(\bigl( w \circ\nobreak v, (\mu,a) \bigr) \in W(i,k)\) and 
      hence one can consider \(t = t_{w \circ v,(\mu,a)}\), which is 
      the most interesting case. \(f_{v(\mu)}(b) = r\) and hence 
      \(t'(b) = r \cdot v(a)\), so that \(w\bigl( t'(b) \bigr) = 
      r \cdot w\bigl( v(a) \bigr) = 
      r \cdot t\bigl( (w \circ\nobreak v)(\mu) \bigr) = 
      t\bigl( r \cdot\nobreak w(\nu) \bigr) = t\bigl( w(b) \bigr)\), 
      as claimed.
    \item
      If \(v(\mu) = \nu\) and \(w(\nu) = 0\) then 
      \((w \circ\nobreak v)(\mu) = 0\) and hence 
      \((w \circ\nobreak v)(a) = 0\), which means \(w\bigl( t'(b) 
      \bigr) = w\bigl( r \cdot\nobreak v(a) \bigr) = r \cdot 
      (w \circ\nobreak v)(a) = r \cdot (w \circ\nobreak v)(\mu) = 
      w\bigl( r \cdot\nobreak v(\mu) \bigr) = w(b)\), and thus 
      \(t\bigl( w(b) \bigr) = w\bigl( t'(b) \bigr)\) for \(t = \id\).
  \end{enumerate}
  Either way, there exists some \(t \in T(S)(i)\) such that 
  \(t\bigl( w(b) \bigr) = w\bigl( t'(b) \bigr)\).
\end{proof}

\begin{lemma} \label{L:Kompatibilitet}
  Let everything be as in the construction. Let a partial order 
  $P(i)$ on $\mc{Y}(i)$ be given for every \(i \in I\). Assume 
  every \(v \in V(i,j)\) correlates $P(j)$ to $P(i)$ and every 
  \(r \in R^\bot(i)\) correlates $P(i)$ to itself. If \(a \in 
  \DSM\bigl( \mu, P(j) \bigr)\) for all \((\mu,a) \in S(j)\) and 
  \(j \in I\), then $T(S)(i)$ is compatible with $P(i)$ for all 
  \(i \in I\).
\end{lemma}
\begin{proof}
  Let \(i,j \in I\), \(\bigl(v, (\mu,a) \bigr) \in W(i,j)\), and 
  \(\nu \in \mc{Y}(i)\) be given. It must be shown that 
  \(t_{v,(\mu,a)}(\nu) \in \{\nu\} \cup \DSM\bigl( \nu, P(i) \bigr)\). 
  If \(f_{v(\mu)}(\nu) = 0\) then \(t_{v,(\mu,a)}(\nu) = \nu\) and 
  all is well. Otherwise \(\nu = r \cdot v(\mu)\) for 
  \(r = f_{v(\mu)}(\nu)\) and \(t_{v,(\mu,a)}(\nu) = r \cdot v(a)\). 
  Since \(a \in \DSM\bigl( \mu, P(j) \bigr)\) it follows that \(v(a) 
  \in v\bigl( \DSM\bigl( \mu, P(j) \bigr)\bigr) \subseteq 
  \DSM\bigl( v(\mu), P(i) \bigr)\) and hence \(r \cdot v(a) \in 
  r \cdot \DSM\bigl( v(\mu), P(i) \bigr) \subseteq 
  \DSM\bigl( \nu, P(i) \bigr)\). Therefore all \(t \in T_1(S)(i)\) 
  are compatible with $P(i)$. By Lemma~\ref{L:Kompabilitet}, this 
  extends to the whole of $T(S)(i)$.
\end{proof}

\begin{lemma} \label{L1:Likgradig}
  Let everything be as in the construction. Let \(i \in I\) be given 
  and let \(\mc{N} \supseteq \cmplO(i)\) be a family of topologically 
  open subgroups of $\cmplM(i)$ such that \(f_\nu(a) \cdot \nu \in 
  \ve\) for all \(a \in \ve\), \(\ve \in \mc{N}\), and \(\nu \in Y(i)\). 
  If all \(j \in I\), \(\bigl( v, (\mu,a) \bigr) \in W(i,j)\), \(r \in 
  \mc{R}\), and \(\ve \in \mc{N}\) such that \(r \cdot v(\mu) \in \ve\) 
  also satisfy \(r \cdot v(a) \in \ve\) then:
  \begin{enumerate}
    \item
      \(t(\ve) \subseteq \ve\) for any \(\ve \in \mc{N}\) and \(t \in 
      T(S)(i)\).
    \item
      $T(S)(i)$ is equicontinuous.
  \end{enumerate}
\end{lemma}
\begin{proof}
  The second claim is an immediate consequence of the first 
  (\(\delta=\ve\) works for all reductions). If the first claim holds 
  for two particular reductions, then it also holds for their 
  composition; hence it is sufficient to verify it for simple 
  reductions. Let \(t \in T_1(S)(i)\) be given. By definition there 
  is some \(j \in I\) and \(\bigl( v, (\mu,a) \bigr) \in W(i,j)\) 
  such that \(t = t_{v,(\mu,a)}\), i.e., \(t(b) = b - f_{v(\mu)}(b) 
  \cdot v(\mu -\nobreak a)\) for all \(b \in \cmplM(i)\). Let \(\ve 
  \in \mc{N}\) and \(b \in \ve\) be arbitrary. \(f_{v(\mu)}(b) \cdot 
  v(\mu) \in \ve\) by the condition on $\mc{N}$ and hence 
  \(f_{v(\mu)}(b) \cdot v(a) \in \ve\) by assumption. It follows that 
  \(t(b) = b - f_{v(\mu)}(b) \cdot v(\mu) + f_{v(\mu)}(b) \cdot v(a) \in 
  \ve - \ve + \ve = \ve\).
\end{proof}

In \cite[Lemma~3.25]{Avhandlingen}, a different proof of 
equicontinuity can be found which is feasible also in cases where the 
first conclusion of the above lemma does not hold; the idea is to 
consider $T(S)(i)$ that are compatible with some $P(i)$ and require 
the latter to satisfy a `squeeze property' (as in the \emph{Squeeze 
Theorem} of elementary analysis): for every \(\ve \in \cmplO(i)\) 
there must exist some \(\delta \in \cmplO(i)\) such that if \(\mu \in 
\mc{Y}(i) \cap \ve\) then every \(\nu < \mu \pin{P(i)}\) must satisfy 
\(\nu \in \delta\). However, I currently don't have any example of a 
situation where this additional generality is needed. That proof is 
also easily disturbed by the existence of ``small'' scalars, since it 
might happen that the \(r(\mu) \in \ve \cap \RstarY(i)\) some \(t \in 
T(S)(i)\) acts upon does not satisfy \(\mu \in \ve\); it is typically 
necessary to have some condition ensuring that elements of $R(i)$ act 
somewhat uniformly on $\cmplM(i)$, and even then things can get hairy.

The final lemma is instead a special case of Lemma~\ref{L1:Likgradig} 
which separates the norm conditions on $S$, $V$, and $\mc{R}$.

\begin{lemma} \label{L2:Likgradig}
  Let everything be as in the construction. Let a function 
  \(U_i\colon Y(i) \Fpil \Rp\) be given for every \(i \in I\). Assume 
  the topology in $\mc{R}$ is given by a ring ultranorm 
  $\norm{\cdot}$. Also assume for all \(i \in I\) that $\mc{O}(i)$ is 
  constructed from $\norm{\cdot}$ and $U(i)$ as in 
  Lemma~\ref{L:Modul-ultranorm}, and let $\Norm{\cdot}_i$ be the 
  $\mc{R}$-module ultranorm on $\cmplM(i)$.
  Assume that there for every \(v \in V\) exists a constant \(C_v \in 
  \Rp\) such that if \(v \in V(i,j)\) and \(\mu \in Y(j)\) then 
  \(\Norm[\big]{v(\mu)}_i \leqslant C_v U_j(\mu)\) and if in 
  addition \(v(\mu) \in Y(i)\) then \(U_i\bigl( v(\mu) \bigr) = 
  C_v U_i(\mu)\).
  
  If \(\Norm{a}_j \leqslant U_j(\mu)\) for all \((\mu,a) \in S(j)\) 
  and \(j \in I\), then \(\Norm[\big]{t(b)}_i \leqslant \Norm{b}_i\) for 
  all \(t \in T(S)(i)\), \(b \in \cmplM(i)\), and \(i \in I\), and 
  moreover $T(S)(i)$ is equicontinuous for every \(i \in I\).
\end{lemma}
\begin{proof}
  This is a special case of Lemma~\ref{L1:Likgradig}.  In order for 
  \(t(\ve) \subseteq \ve\) for any \(\ve \in \mc{N}\) and \(t \in 
  T(S)(i)\) to imply \(\Norm[\big]{t(b)}_i \leqslant \Norm{b}_i\) for 
  all \(t \in T(S)(i)\) and \(b \in \cmplM(i)\), it is necessary to 
  take
  \begin{equation*}
    \mc{N} = \setOf[\Big]{ 
      \setOf[\big]{ b \in \cmplM(i) }{ \Norm{b}_i < e }
    }{ e \in \Rp }
    \text{,}
  \end{equation*}
  but apart from that the proof is purely a matter of demonstrating 
  that the conditions in Lemma~\ref{L1:Likgradig} are met. That 
  \(f_\nu(a) \cdot \nu \in \ve\) for all \(a \in \ve\), \(\ve \in 
  \mc{N}\), and \(\nu \in Y(i)\) follows from 
  \eqref{Eq2:Modul-ultranorm}.
  
  For the main condition, let \(i,j \in I\), \(\bigl( v, (\mu,a) 
  \bigr) \in W(i,j)\), \(r \in \mc{R}\), and \(e \in \Rp\) such that 
  \(\Norm[\big]{r \cdot v(\mu)}_i < e\) be given. By 
  \eqref{Eq:Modul-ultranorm}, \(\Norm[\big]{r \cdot v(\mu)}_i = 
  \norm{r} U_i\bigl( v(\mu) \bigr) = \norm{r} C_v U_j(\mu)\). For 
  every \(\nu \in Y(j)\),
  \begin{multline*}
    \Norm[\Big]{ r \cdot v\bigl( f_\nu(a) \cdot \nu \bigr) }_i 
      \leqslant
    \norm{r} \Norm[\big]{ f_\nu(a) \cdot v(\nu) }_i \leqslant
    \norm{r} \norm[\big]{f_\nu(a)} \Norm[\big]{ v(\nu) }_i 
      \leqslant \\ \leqslant
    \norm{r} \norm[\big]{f_\nu(a)} C_v U_j(\nu) \leqslant
    \norm{r} C_v \Norm{a}_j \leqslant 
    \norm{r} C_v U_j(\mu) =
    \Norm[\big]{r \cdot v(\mu)}_i < e
    \text{.}
  \end{multline*}
  Since $\Norm{\cdot}_i$ is an ultranorm and $r \cdot v(a)$ is in the 
  topologically closed group generated by $\bigl\{ r \cdot\nobreak 
  v\bigl( f_\nu(a) \cdot\nobreak \nu \bigr) \bigr\}_{\nu \in Y(j)}$, 
  it now follows that \(\Norm[\big]{ r \cdot v(a) }_i < e\).
\end{proof}

\section{Gr\"obner bases}
\label{Sec:Grobner}

The following treatment of Gr\"obner bases is primarily aimed at 
demonstrating how some known results in this area can be derived from 
the diamond lemma, hence it does not seek to give a definition of 
Gr\"obner basis that applies in all situations covered by the 
$\bigl( \mc{M}, R, \mc{Y}, \mc{O}, T_1(S) \bigr)$ formalism. The 
restrictions that will be made are:
\begin{enumerate}
  \item
    There will only be one sort.
  \item 
    The topology will be discrete.
  \item
    $\mc{M}$ will be a free left $\mc{R}$-module, where $\mc{R}$ 
    is a unital ring, and $R$ will be the set of maps that multiply 
    by an element of $\mc{R}$.
  \item
    $\mc{Y}$ will be a basis of $\mc{M}$.
\end{enumerate}
One restriction that will \emph{not} be made is that of only 
considering total orders, as that is needed more to ensure 
existence of Gr\"obner bases than to define or use them. Some 
algebraic structures require compatible partial orders to be 
non-total, so a restriction to total orders really sacrifices some 
generality.

\subsection{Generic theory}

On a practical level, the property that something is a Gr\"obner 
basis is equivalent to the four claims in Theorem~\ref{S:CDL}, which 
means several equivalent characterisations of this concept could be 
made. The standard definition is however the fifth claim that 
`\emph{the leading monomial of an element of the ideal must be a 
multiple of the leading monomial of some element of the basis}', 
which accordingly appears as claim~\parenthetic{d} of 
Theorem~\ref{S:Grobner-DL}. One reason this characterisation has 
become so popular is no doubt that it is amenable to an informal 
presentation\Ldash everybody knows what the leading monomial is, 
don't they?\Dash although once one starts to do anything with the 
concept (such as reducing modulo a tentative Gr\"obner basis), most 
technical details of a reduction-based approach quickly suggest 
themselves. Moreover, even the issue of what it means to be the 
leading monomial is not without technical complications when 
considered in the present generality.

\begin{definition}
  Let \(\left\{ f_\mu\colon \mc{M} \Fpil \mc{R} 
  \right\}_{\mu\in\mc{Y}}\) be the family of coefficient-of-$\mu$ 
  homomorphisms associated with the basis $\mc{Y}$ for $\mc{M}$. The 
  \DefOrd{support} $\supp(a)$\index{supp@$\supp(a)$} of an \(a \in 
  \mc{M}\) is the set of \(\mu \in \mc{Y}\) for which \(f_\mu(a) \neq 
  0\).
  
  Let $P$ be a binary relation on $\mc{Y}$. A 
  \DefOrd[*{leading monomial}]{$P$-leading monomial} of some \(a 
  \in \mc{M}\) is a $P$-maximal element of $\supp(a)$, i.e., a
  \(\mu \in \supp(a)\) such that no \(\nu \in \supp(a)\) satisfies 
  \(\mu < \nu \pin{P}\). Denote by $\LM_P(a)$\index{LM@$\LM(g)$} 
  the set of $P$-leading monomials of $a$. If $\LM_P(a)$ has exactly 
  one element, then denote that by $\lm_P(a)$\index{lm@$\lm(a)$}.
\end{definition}

The main reason for restricting this treatment to $\mc{Y}$ being a 
basis of $\mc{M}$ and the topology being discrete is that this 
ensures $\supp(a)$\Ldash informally ``the set of monomials occurring 
in $a$''\Dash is well-defined. Linear dependencies in $\mc{Y}$ would 
obviously remove the foundation for this concept, and topology can 
(but doesn't have to) produce similar problems: on can choose a 
topology such that there are \(Y_1,Y_2 \subset \mc{Y}\) for which 
\(\Span(Y_1) \cap \Span(Y_2) = \{0\}\) but \(\Cspan(Y_1) \cap 
\Cspan(Y_2) \neq \{0\}\). Even after 
ensuring that $\supp(a)$ is well-defined for all \(a \in \cmplM\), 
a topology can cause the definition of $\LM_P(a)$ to fail, if some 
$\supp(a)$ is infinite and contains an infinite $P$-ascending chain. 
One approach for defining Gr\"obner bases without relying on the 
support concept could be to replace the concrete construction of 
$\LM_P(g)$ by an abstract map $L$ that assigns a set of leading 
monomials to each element of $\cmplM$. The effect would probably be 
similar to the formalism in~\cite{Mora:Seven}, even though that 
technically goes in the other direction: the ``$L$'' map has a 
canonical construction but the monomials are abstracted away.

\begin{definition}
  Let $V$ be a set of $\mc{R}$-module homomorphisms \(\mc{M} \Fpil 
  \mc{M}\). Let $P$ be a binary relation on $\mc{Y}$. Let \(N 
  \subseteq \mc{M}\) be a $V$-ideal. A subset $G$ of $\mc{M}$ is said 
  to be a \DefOrd[*{Gr\"obner basis}]{$P$-monic $V$-Gr\"obner-basis} 
  of $N$ if it is $P$-monic, \(v(g) \in N\) for all \(v \in V\) and 
  \(g \in G\), and there for every \(a \in N\) and \(\mu \in \LM_P(a)\) 
  exists some \(g \in G\) and \(v \in V\) such that \(\mu = v\bigl( 
  \lm_P(g) \bigr)\).
\end{definition}

Making $P$-monicity a precondition for Gr\"obner bases serves two 
purposes: it ensures there is a unique $P$-leading monomial and it 
ensures reductions compatible with $P$ can be manufactured from basis 
elements. While these are important ingredients in 
Lemma~\ref{L:Grobnerbas}, the $P$-monicity condition also works 
against a very elementary result in traditional Gr\"obner basis 
theory, namely that every ideal should have a Gr\"obner basis. 
Without the $P$-monicity it would be possible to simply make the 
observation that the ideal itself is a Gr\"obner basis for it\Ldash 
admittedly a ridiculously large basis (probably infinite in most 
cases where finite bases exist), but nonetheless a basis\Dash which 
formally justifies assuming every ideal one needs to work with is 
generated by a Gr\"obner basis. If a Gr\"obner basis is to be 
$P$-monic however, one has to be careful about what elements can be 
included, but as long as $P$ is a total order and $\mc{R}$ is a field 
there is always a $P$-monic counterpart of every nonzero element of 
$\mc{M}$.

It should also be observed that the definition of Gr\"obner basis does 
not explicitly require $G$ to be a $V$-ideal basis for $N$, and in 
fact it depends on $P$ whether this follows. A trivial counterexample 
is to consider \(N = \mc{M} = \mc{R}[\ssx]\) and 
\(G = \{1 +\nobreak \ssx\}\) where 
$\mc{R}$ is a field, \(\mc{Y} = \{\ssx^n\}_{n\in\N}\), \(V = \{ b 
\mapsto\nobreak \ssx^n b \}_{n\in\N}\), and \(\ssx^m \leqslant \ssx^n 
\pin{P}\) iff \(m \geqslant n\); since \(\lm_P(1 +\nobreak \ssx) = 
1\) it is easy to see that $G$ is a $V$-Gr\"obner-basis for $N$, but 
\(1 \notin \Span\bigl( \bigl\{ (1 +\nobreak \ssx) \ssx^n 
\bigr\}_{n\in\N} \bigr)\) and so $G$ isn't a $V$-ideal basis for $N$. 
The catch in this example is that $P$ isn't well-founded; the 
condition defining Gr\"obner bases lends itself to the step in an 
induction for proving the ideal basis property, but it cannot also 
provide the base for that induction. The next lemma gives 
sufficient conditions on $V$ and $P$ for Gr\"obner bases to be ideal 
bases.

\begin{lemma} \label{L:Grobnerbas}
  Let $V$ be a monoid of $\mc{R}$-module homomorphisms \(\mc{M} \Fpil 
  \mc{M}\). Let $P$ be a well-founded partial order on $\mc{Y}$ which 
  is correlated to itself by every \(v \in V\). If \(G \subseteq \mc{M}\) 
  is a $P$-monic $V$-Gr\"obner-basis for a $V$-ideal \(N \subseteq 
  \mc{M}\) then $G$ is a $V$-ideal basis for $N$. If furthermore 
  $T_1(S)$ is as in Construction~\ref{Konstr:t_v,s-reduktion} 
  for \(S = \bigl\{ \bigl( \lm_P(g), \lm_P(g) -\nobreak g \bigr) 
  \bigr\}_{g \in G}\) then \(\mc{I}(S) = N\) and for every \(a \in N\) 
  there exists some \(t \in T(S)\) such that \(t(a)=0\).
\end{lemma}
\begin{proof}
  The main claim is that about existence of reductions which map 
  elements of $N$ to $0$. What the Gr\"obner property implies is that 
  there for every nonzero \(a \in N\) and \(\mu \in \LM_P(a)\) exists 
  some \(t \in T_1(S)\) which acts nontrivially on $\mu$, namely 
  \(t = t_{\mu \mapsto \mu-v(g)}\) where \(g \in G\) and \(v \in V\) 
  are such that \(\mu = v\bigl( \lm_P(g) \bigr)\), since this is 
  $t_{v,(\nu,b)}$ where \(\nu = \lm_P(g)\) and \(b = \nu - g\). All 
  these simple reductions are compatible with $P$, since \(b \in 
  \DSM(\nu,P)\) because $g$ is $P$-monic and \(v(b) \in \DSM(\mu,P) = 
  \DSM\bigl( v(\nu), P\bigr)\) by assumption.
  
  Let \(a_0 \in N\) be given. Construct from any nonzero \(a_n \in N\) 
  the element \(a_{n+1} = u_n(a_n)\) by picking as $u_n$ some 
  composition $u_{n,m_n} \circ \dotsb \circ u_{n,1}$ of simple 
  reductions such that \(u_{n,k} \in T_1(S)\) acts nontrivially on 
  \(\mu_{n,k} \in \LM_P(a_n)\), where \(\{\mu_{n,1}, \dotsc, 
  \mu_{n,m_n}\} = \LM_P(a_n)\). The claim follows once it has been 
  shown that \(a_n=0\) for some $n$ (which means \(u_l=\id\) for all 
  \(l>n\)), and the way to establish this is to consider the sets 
  $\LM_P(a_n)$.
  
  Let \(Z = \bigcup_{n=0}^\infty \LM_P(a_n)\). 
  First observe that \(u_n(\mu_{n,k}) \in \DSM(\mu_{n,k},P)\) for any 
  $n$ and $k$ by Lemma~\ref{L:Kompabilitet}. Since any $P$-leading 
  monomial of $a_n$ is some $\mu_{n,k}$, and since \(\Span\bigl( 
  \supp(a_n) \setminus\nobreak \LM_P(a_n) \bigr) \subseteq 
  \sum_{k=1}^{m_n} \DSM(\mu_{n,k},P)\), it follows that
  \begin{equation*}
    \LM_P(a_{n+1}) \subseteq \supp(a_{n+1}) \subseteq
    \setOf[\big]{ \nu \in \mc{Y} }{ \text{\(\nu < \mu \pin{P}\) for 
      some \(\mu \in \LM_P(a_n)\)} }
  \end{equation*}
  for \(n=0,1,\dotsc\). Construct the directed acyclic graph $D$ 
  which has $Z$ as vertex set and has an edge from $\mu$ to $\nu$ iff 
  \(\mu > \nu \pin{P}\) and there exists some \(n\in\N\) such that 
  \(\mu \in \LM_P(a_n)\) and \(\nu \in \LM_P(a_{n+1})\). Since any 
  path in this graph is a $P$-descending chain, it is finite. Since 
  any $\LM_P(a_n)$ is finite, the graph has finite branching. Finally 
  the only roots in $D$ are the elements of $\LM_P(a_0)$. Hence 
  K\"onig's lemma (an infinite tree with finite branching has an 
  infinite path) applies, and it follows that $Z$ is finite. In 
  particular, there exists some $n$ for which \(\LM_P(a_n) = 
  \varnothing\) and thus \(a_n=0\), as claimed.
  
  That \(N \subseteq \mc{I}(S)\) is now immediate from the 
  definition of the latter. Conversely it may be observed that 
  if \(\mu - t(\mu) \neq 0\) for some \(t \in T_1(S)\) and \(\mu \in 
  \mc{Y}\) then there exist \(g \in G\) and \(v \in V\) such that 
  \(\mu - t(\mu) = v(g)\), since $t$ is of the form 
  $t_{\mu \mapsto \mu-v(g)}$. Thus
  \begin{multline*}
    \mc{I}(S) = \Cspan\Bigl( \setOf[\big]{ v(\mu - a) }{ 
      v \in V, (\mu,a) \in S, v(\mu) \in \mc{Y}
    } \Bigr) \subseteq \\ \subseteq
    \Cspan\Bigl( \setOf[\big]{ v(g) }{ v \in V, g \in G } \Bigr) 
    \subseteq N
  \end{multline*}
  and this also shows that $G$ is a $V$-ideal basis of $N$.
\end{proof}

With this result about the existence of reductions which map ideal 
elements to $0$, it becomes easy to link the Gr\"obner basis concept 
to those of Theorem~\ref{S:CDL}. Claim \parenthetic{a\texttris} below 
is included because it is literally the claim that ``all 
S-polynomials reduce to $0$'' which is practically used to verify 
that something is a Gr\"obner basis.

\begin{theorem} \label{S:Grobner-DL}
  Let $P$ be a well-founded partial order on $\mc{Y}$. 
  Let $V$ be a monoid of $\mc{R}$-module homomorphisms \(\mc{M} \Fpil 
  \mc{M}\) that map $\mc{Y}$ into $\mc{Y}$ and are strictly monotone 
  with respect to $P$.
  
  Let \(S \subseteq \mc{Y} \times \mc{M}\) be such that \(a \in 
  \DSM(\mu,P)\) for any \((\mu,a) \in S\). Let $T_1(S)$ be as 
  in Construction~\ref{Konstr:t_v,s-reduktion}. Then the 
  following conditions are equivalent:
  \begin{enumerate}
    \item[\parenthetic{a}] 
      Every ambiguity of $T_1(S)$ is resolvable.
    \item[\parenthetic{a\textprime}] 
      Every ambiguity of $T_1(S)$ is resolvable relative to $P$.
    \item[\parenthetic{a\textbis}] 
      Every $V$-critical ambiguity of $T_1(S)$ is resolvable 
      relative to $P$.
    \item[\parenthetic{a\texttris}] 
      For every $V$-critical ambiguity $(t_1,\mu,t_2)$ of $T_1(S)$ 
      there exists some \(t \in T(S)\) such that \(t\bigl( t_1(\mu) 
      -\nobreak t_2(\mu) \bigr) = 0\).
    \item[\parenthetic{b}]
      Every element of $\mc{M}$ is persistently and uniquely 
      reducible.
    \item[\parenthetic{c}]
      Every element of $\mc{M}$ has a unique normal form.
    \item[\parenthetic{d}]
      The set \(\{\mu -\nobreak a\}_{(\mu,a) \in S}\) is a $P$-monic 
      $V$-Gr\"obner-basis of $\mc{I}(S)$.
  \end{enumerate}
\end{theorem}
\begin{proof}
  First observe that strict monotonicity of $V$ implies correlation 
  by Lemma~\ref{L:OrdnFramflyttKompatibel}, and hence $T(S)$ is 
  compatible with $P$ by Lemma~\ref{L:Kompatibilitet}. All 
  elements of $V$ are advanceable with respect to $T_1(S)$ by 
  Lemma~\ref{L:Framflyttbar}, and thus $\mc{I}(S)$ is a 
  $V$-ideal by Lemma~\ref{L:Slutenhet,I'(S)}.
  
  Claims \parenthetic{a}, \parenthetic{a\textprime}, \parenthetic{b}, 
  and \parenthetic{c} are equivalent by Theorem~\ref{S:CDL}. 
  Claims \parenthetic{a\textprime} and \parenthetic{a\textbis} are 
  equivalent by Theorem~\ref{S:V-kritisk}. \parenthetic{a\texttris} 
  implies that every $V$-critical ambiguity is resolvable, and hence 
  \parenthetic{a\textbis} by Lemma~\ref{L:RelResolv}. Conversely 
  \parenthetic{b} implies that \(\mu \in \Red(S)\) for every 
  $V$-critical ambiguity $(t_1,\mu,t_2)$ of $T_1(S)$ and hence 
  \(\tS\bigl( t_1(\mu) \bigr) = \tS\bigl( t_2(\mu) \bigr)\), from 
  which follows \(\tS\bigl( t_1(\mu) -\nobreak t_2(\mu) \bigr) = 0\) 
  and thus \parenthetic{a\texttris}.
  
  Assume~\parenthetic{d}. By Lemmas~\ref{L: Persistently red.} 
  and~\ref{L:PerTillIrr}, \(\mc{M} = \Irr(S) + \mc{I}(S)\). 
  Furthermore every \(a \in \mc{I}(S) \cap \Irr(S)\) satisfies 
  \(t(a) = a\) for all \(t \in T(S)\), but by 
  Lemma~\ref{L:Grobnerbas} there is some \(t \in T(S)\) such that 
  \(t(a) = 0\). Hence \(a=0\), which has established \(\mc{M} = 
  \Irr(S) \oplus \mc{I}(S)\). It follows that claim~\parenthetic{d} 
  implies claim~\parenthetic{c}.
  
  Finally assume \(\Red(S) = \mc{M}\) and let \(b \in \mc{I}(S)\) be 
  arbitrary. Since \(\tS(b) = 0\) there exists some \(t \in T(S)\) 
  such that \(t(b)=0\). Let \(\lambda \in \LM_P(b)\) be arbitrary. 
  Since \(\lambda \notin \LM_P(0)\), there exists some decomposition 
  \(t = t_3 \circ t_2 \circ t_1\) where \(t_3,t_1 \in T(S)\) and 
  \(t_2 \in T_1(S)\) are such that \(\lambda \in \LM_P\bigl( t_1(b) 
  \bigr)\) but \(\lambda \notin \LM_P\bigl( (t_2 \circ\nobreak t_1)(b) 
  \bigr)\). Since there is no \(\nu \in \supp\bigl( t_1(b) \bigr)\) 
  such that \(\nu > \lambda \pin{P}\), it must be the case that $t_2$ 
  acts nontrivially on $\lambda$. Due to the way $T_1(S)$ was 
  constructed, this means there is some \(v \in V\) and \((\mu,a) \in 
  S\) such that \(v(\mu) = \lambda\). Since furthermore \(\mu = 
  \lm_P(\mu -\nobreak a)\), the condition with respect to $b$ and 
  $\lambda$ for \(\{\mu -\nobreak a\}_{(\mu,a) \in S}\) to be a 
  $V$-Gr\"obner-basis is fulfilled. Hence claim~\parenthetic{b} 
  implies claim~\parenthetic{d}.
\end{proof}


A classical case not handled by this theorem is that $P$ is a total order 
but elements of $\mc{R}$ sometimes aren't invertible. This is where 
monicity becomes a restriction, since there in for example the case 
that $\mc{R}$ is an euclidian domain exists an established 
theory\Ldash which in addition to gaussian elimination and polynomial 
division also generalises integer division (with remainder)\Dash for 
Gr\"obner bases where leading terms may have noninvertible 
coefficients. Reducing a term $a\mu\nu$ by a basis element $g$ whose 
leading term is $b\mu$ then consists of subtracting $q g \nu$ from 
$a\mu\nu$, where $q$ is the quotient of $a$ by $b$, and may thus fail 
to completely eliminate the $\mu\nu$ term. What makes this 
incompatible with the diamond lemma framework is however that the 
quotient $q$, and hence the reduction as a whole, is not given by a 
homomorphism; $(a_1 \Div b) + (a_2 \Div b)$ is not always equal to 
$(a_1 +\nobreak a_2) \Div b$, even through the error may be 
predictable. The standard bases formalism of~\cite{Mora:Seven,Robbiano} 
has facilities\footnote{
  In particular the duplication of addition operations: one which is 
  ``with carry'' (coming from the filtered structure) and one which 
  is ``without carry'' (coming from the associated graded structure). 
  For a suitable choice of filtered structure the latter addition has 
  \((a_1 \Div b) + (a_2 \Div b) = (a_1 +\nobreak a_2) \Div b\), and 
  since reductions are required to be homomorphisms with respect to 
  this ``without carry'' addition, it is then allowed to use integer 
  division when constructing reductions.
} that can handle this, and it's quite possible that the same trick 
could be applied also in a modification of the diamond lemma formalism, 
but for the moment I don't see a pressing need for this. It is more 
interesting to examine some alternative approaches for coping with 
noninvertible coefficients within the present framework, even though 
they are perhaps not as general.

If $\mc{R}$ can be regarded as an algebra over some smaller subring 
(maybe even subfield) $\mc{K}$, then a practical solution can be to 
change the boundary between $R$ and $\mc{Y}$, letting the former 
encode just $\mc{K}$ and extending $\mc{Y}$ accordingly. 
Corollary~\ref{Kor:KomAssPolynom} below can be viewed as using this 
approach to Gr\"obner bases in $\RavX$ where $\mc{R}$ itself is a 
commutative polynomial algebra $\mc{K}[X_1]$.

The other extreme is that \(\mc{R}=\Z\), in which case there is no 
additional freedom that can be gained from shrinking $R$ since the 
group structure alone determines what multiples of monomials are 
mapped to. Consider the case that one wishes to make a model for 
$\Z[x] \big/ \langle 2x \rangle$. It is easy to jump to the 
conclusion that the diamond lemma framework cannot handle this, on 
the grounds that \(\mc{M} = \Z[x]\) is a free $\Z$-module and hence 
any $\Irr(S)$ must be free too whereas the target \(\mc{M}/\mc{I}(S) 
= \Z[x] \big/ \langle 2x \rangle\) clearly is not. It is true that 
Theorem~\ref{S:Grobner-DL} is so restricted, but there is nothing in 
the generic theory which requires one to pick $\Z[x]$ as $\mc{M}$ 
(even though that would be the trivial choice). An interesting 
alternative in this case would be \(\mc{M} = \Z[x] \oplus 
\Z_2[x_2]\), since one for \(T_1(S) = 
\{t_{x^n \mapsto x_2^n}\}_{n=1}^\infty \cup \{t_{x_2^0 \mapsto 0}\}\) 
indeed gets \(\Z[x] \big/ \langle 2x \rangle \cong \mc{M}/\mc{I}(S) 
\cong \Irr(S)\) as $\Z$-modules.

This approach of introducing ``modular coefficients'' in parallel with 
the original coefficients will however not work for the formally similar 
case of $\Z[x] \big/ \langle 2x -\nobreak 1 \rangle$. Whereas a map 
that for all \(n\in\Z\) takes $2nx$ to $n$ and $(2n +\nobreak 1)x$ to 
$n + x_2$ makes sense as a map (and indeed is what one would 
arrive at in the standard bases formalism), it cannot serve as a 
reduction in the diamond lemma formalism because it is not a group 
homomorphism; \(x_2+x_2 = 0 \neq 1\). What \emph{will} work is instead 
to pick \(\mc{M} = \Z[\tfrac{1}{2}][x]\), where $\Z[\tfrac{1}{2}]$ 
should be regarded as the subring of $\Q$ generated by 
$\tfrac{1}{2}$. The main justification for introducing such a 
multiplicative inverse of $2$ would be the defining identity \(2x 
\equiv 1\) itself\Ldash whose interpretation must be that $x$ is 
precisely such an inverse\Dash and once $\tfrac{1}{2}$ is available 
the rest is trivial. 

The common idea generalising both cases appears to be that if one 
wants to make a reduction mapping $r\mu$ to $a$ and $r$ is neither 
invertible nor a zero divisor, then one should extend $\mc{M}$ with a 
new element $\mu'$ such that \(r\mu' = a\) and have the reduction map 
$\mu$ to $\mu'$. (The `new' is important here, because \(\Z[x,y] \big/ 
\langle 2x -\nobreak 2y \rangle \cong \Z[y] \oplus \xi\Z_2[\xi,y]\) 
where \(\xi = y - x'\); that \(2x' = 2y\) but \(x' \neq y\) since 
$x'$ is new is what creates the characteristic $2$ part.) Whether 
this method can be turned into an algorithm (as has been done for 
Gr\"obner bases over e.g.~euclidian domains) is at the time of writing 
unclear\Ldash automating this kind of modifications to the base group 
$\mc{M}$ seems highly nontrivial\Dash but it should illustrate the 
usefulness of not having Theorem~\ref{S:CDL} restricted to the case 
of $\mc{M}$ being a free module, even though that assumption 
simplifies the results in this section quite considerably.

\subsection{Commutative, associative, and nonassociative algebras}

Many forms of the fundamental theorem of Gr\"obner bases turn out to 
be special cases, with particular choices of $\mc{M}$ and $V$, of 
Theorem~\ref{S:Grobner-DL} and therefore follow from it as mere 
corollaries. The most classical is that for commutative polynomials 
over a field.

\begin{corollary}[Buchberger] \label{Kor:KomGrobner}
  Let $\mc{R}$ be a field, let $X$ be a set, let \(\mc{M} = 
  \mc{R}[X]\), and let $\mc{Y}$ be the set of monomials (power 
  products) in $\mc{M}$. Let \(V = 
  \{ b \mapsto\nobreak \mu b\}_{\mu \in \mc{Y}}\) (a set of maps 
  \(\mc{M} \Fpil \mc{M}\)). Let $P$ be a well-founded semigroup 
  total order on $\mc{Y}$. Define a map \(Z\colon \bigl( \mc{Y} 
  \times\nobreak \mc{M} \bigr)^2 \Fpil \mc{M}\) by
  \begin{equation}
    Z\bigl( (\mu_1,a_1), (\mu_2,a_2) \bigr) =
    \frac{\mathrm{lcm}(\mu_1,\mu_2)}{\mu_1} a_1 - 
      \frac{\mathrm{lcm}(\mu_1,\mu_2)}{\mu_2} a_2 
  \end{equation}
  where $\mathrm{lcm}(\mu_1,\mu_2)$ denotes the least common multiple 
  of $\mu_1$ and $\mu_2$.
  
  Let \(S \subseteq \mc{Y} \times \mc{M}\) be such that \(a \in 
  \DSM(\mu,P)\) for every \((\mu,a) \in S\) and let $T_1(S)$ be as 
  in Construction~\ref{Konstr:t_v,s-reduktion}. Then the following 
  are equivalent:
  \begin{enumerate}
    \item
      \(\{\mu-a\}_{(\mu,a) \in S}\) is a $P$-monic $V$-Gr\"obner 
      basis of $\mc{I}(S)$.
    \item
      For every pair \((s_1,s_2) \in S^2\) there exists some 
      \(t \in T(S)\) such that \(t\bigl( Z(s_1,s_2) \bigr) = 0\).
    \item
      For every \(s_1 = (\mu_1,a_1) \in S\) and \(s_2 = (\mu_2,a_2) 
      \in S\) such that $\mu_1$ and $\mu_2$ are not coprime there 
      exists some \(t \in T(S)\) such that 
      \(t\bigl( Z(s_1,s_2) \bigr) = 0\).
  \end{enumerate}
\end{corollary}
\begin{proof}
  This is mostly the equivalence of \parenthetic{d} and 
  \parenthetic{a\texttris} from Theorem~\ref{S:Grobner-DL}, but 
  there are minor variations so it doesn't hurt to make the chain of 
  implications explicit.
  
  The first claim is exactly \parenthetic{d}, so this is equivalent 
  to \(\Red(S) = \mc{M}\). Let \(s_1 = (\mu_1,a_1) \in S\) and 
  \(s_2 = (\mu_2,a_2) \in S\) be arbitrary. Let \(\nu_1 = 
  \mathrm{lcm}(\mu_1,\mu_2)/\mu_1\) and \(\nu_2 = 
  \mathrm{lcm}(\mu_1,\mu_2)/\mu_2\). Then \(Z(s_1,s_2) = \nu_1a_1 - 
  \nu_2a_2 = \nu_1(a_1 -\nobreak \mu_1) - \nu_2(a_2 -\nobreak \mu_2) 
  \in \mc{I}(S)\) and hence \(t^S\bigl( Z(s_1,s_2) \bigr) = 0\). 
  Since every value of $t^S$ is attained by some reduction, this has 
  shown that the first claim implies the second, and it is trivial 
  that the second implies the third.
  
  It only remains to show that the third claim is in fact 
  \parenthetic{a\texttris}. To that end, let 
  $(t_{v_1,s_1},\mu,t_{v_2,s_2})$ be a $V$-critical 
  ambiguity of $T_1(S)$. Let \((\mu_i,a_i) = s_i\) and \(\nu_i = 
  v_i(1)\) for \(i=1,2\). Then \(\nu_1\mu_1 = \mu = \nu_2\mu_2\) and 
  hence $\mathrm{lcm}(\mu_1,\mu_2)$ divides $\mu$. However if 
  \(\kappa := \mu / \mathrm{lcm}(\mu_1,\mu_2) \neq 1\) then 
  $(t_{v_1,s_1},\mu,t_{v_2,s_2})$ would be a proper $V$-shadow of 
  $(t_{v_1/\kappa,s_1},\mu/\kappa,t_{v_2/\kappa,s_2})$, which by 
  criticality is not the case. Similarly \(\gcd(\mu_1,\mu_2) \neq 1\) 
  since one would otherwise have \(\nu_1 = \mu_2\) and \(\nu_2 = 
  \mu_1\), in which case $(t_{v_1,s_1},\mu,t_{v_2,s_2})$ would be a 
  montage with composition map \(w(b_1,b_2) = b_1b_2\). Finally 
  \(Z(s_1,s_2) = \nu_1a_1 - \nu_2a_2 = t_{v_1,s_1}(\mu) - 
  t_{v_2,s_2}(\mu)\).
\end{proof}

Another applied specialisation of Theorem~\ref{S:Grobner-DL} would be 
to take $\mc{Y}$ to be a monoid on the form \(X_1^\bullet \times 
X_2^*\), where $X_1^\bullet$\index{X bullet@$X^\bullet$} denotes the 
free \emph{abelian} monoid generated by $X_1$. This can be used to 
formally justify Gr\"obner basis calculations in $\FAlg{X_2}{\mc{R}}$ 
where the given relations contain some set $X_1$ of commutative 
coefficients for which one doesn't want to fix the values, by 
making the calculations in $\FAlg{X_2}{\mc{R}[X_1]}$ instead. 

There is of course always the possibility to work in 
$\FAlg{X_1 \cup X_2}{\mc{R}}$ and add relations to make elements of 
$X_1$ commute with everything else, but that can get unintuitive and 
impractical (especially if $X_1$ is large compared to $X_2$). Another 
possibility would be to make a transcendental field extension of 
$\mc{R}$ with the variables in $X_1$, but that would then make it 
formally questionable to specialise to a case where the coefficients 
satisfy some algebraic relation.

\begin{corollary} \label{Kor:KomAssPolynom}
  Let $\mc{R}$ be an associative and commutative ring with unit, 
  let $X_1$ and $X_2$ be disjoint sets, let \(\mc{M} = 
  \FAlg{X_2}{\mc{R}[X_1]}\), and let $\mc{Y}$ be the monoid in 
  $\mc{M}$ which is generated by $X_1 \cup X_2$. Write 
  $X_1^\bullet$ for the abelian submonoid of $\mc{Y}$ which is 
  generated by $X_1$ alone. Let \(V = 
  \{ b \mapsto\nobreak \kappa \mu b\nu\}_{\kappa \in X_1^\bullet, 
  \mu,\nu \in X_2^*}\) (a set of maps \(\mc{M} \Fpil \mc{M}\)). Let 
  $P$ be a well-founded semigroup partial order on $\mc{Y}$. 
  
  Let \(S \subseteq \mc{Y} \times \mc{M}\) be such that \(a \in 
  \DSM(\mu,P)\) for every \((\mu,a) \in S\) and let $T_1(S)$ be as 
  in Construction~\ref{Konstr:t_v,s-reduktion}. Then the three claims 
  that \(\Red(S) = \mc{M}\), \(\mc{M} = \Irr(S) \oplus \mc{I}(S)\), 
  and \(\{\mu-a\}_{(\mu,a) \in S}\) is a $P$-monic $V$-Gr\"obner 
  basis of $\mc{I}(S)$ are each equivalent to the conjunction of the 
  following two conditions:
  \begin{itemize}
    \item
      For every octuplet \(\bigl( (\mu_1,a_1), (\mu_2,a_2), 
      r_1,r_2,r_3, \nu_1,\nu_2,\nu_3 \bigr) \in S^2 \times 
      (X_1^\bullet)^3 \times (X_2^*)^3\) such that \(\mu_1 = 
      r_1r_2\nu_1\nu_2\), \(\mu_2 = r_2r_3\nu_2\nu_3\), 
      \(\nu_1,\nu_2,\nu_3 \neq \ssI\), and \(\gcd(r_1,r_2) = 
      \gcd(r_2,r_3) = \gcd(r_1,r_3) = 1\), there exists some \(t \in 
      T(S)\) such that \(t(r_3 a_1 \nu_3 -\nobreak r_1 \nu_1 a_2) 
      = 0\).
    \item
      For every octuplet \(\bigl( (\mu_1,a_1), (\mu_2,a_2), 
      r_1,r_2,r_3, \nu_1,\nu_2,\nu_3 \bigr) \in S^2 \times 
      (X_1^\bullet)^3 \times (X_2^*)^3\) such that \(\mu_1 = 
      r_1r_2\nu_1\nu_2\nu_3\), \(\mu_2 = r_2r_3\nu_2\), 
      \((\mu_1,a_1) \neq (\mu_2,a_2)\), and 
      \(\gcd(r_1,r_2) = \gcd(r_2,r_3) = \gcd(r_1,r_3) = 1\), there 
      exists some \(t \in T(S)\) such that \(t(r_3 a_1 -\nobreak 
      r_1 \nu_1 a_2 \nu_3) = 0\).
  \end{itemize}
\end{corollary}
\begin{proof}[Proof sketch]
  Same overall structure as in the proof of 
  Corollary~\ref{Kor:KomGrobner}, only the identification of 
  $V$-critical ambiguities needs to be revised. This splits into a 
  noncommutative part for $X_2^*$ which is the same as in 
  Example~\ref{Ex:BergmanskTvetydighet} and a commutative part for 
  $X_1^\bullet$ which is the same as in 
  Corollary~\ref{Kor:KomGrobner}.
\end{proof}

\begin{corollary}[Gerritzen~\cite{Gerritzen}] \label{Kor:Gerritzen}
  Let $\mc{R}$ be a field, let $X$ be a set, let $\mc{Y}$ be the free 
  magma $\mathrm{Mag}(X)$ on $X$, and let $\mc{M}$ be the free 
  $\mc{R}$-module with basis $\mc{Y}$. Extend the multiplication on 
  $\mc{Y}$ to $\mc{M}$ by bilinearity, so that $\mc{M}$ is the 
  (nonunital) free nonassociative $\mc{R}$-algebra $\mc{R}\{X\}$ on 
  $X$. Let $V_1$ be the set of all maps \(\mc{M} \Fpil \mc{M} : 
  b \mapsto \nu b\) and \(\mc{M} \Fpil \mc{M} : b \mapsto b \nu\) 
  for \(\nu \in \mc{Y}\). Let $V$ be the monoid (with 
  composition as operation) generated by $V$. 
  
  Let $P$ be a well-founded total order on $\mc{Y}$ such that
  \begin{equation} \label{Eq1:Gerritzen}
    \lambda < \mu \pin{P} \Ipil
    \text{\(\lambda\nu < \mu\nu \pin{P}\) and \(\nu\lambda < \nu\mu 
    \pin{P}\)}
  \end{equation}
  for all \(\lambda,\mu,\nu \in \mc{Y}\). Let \(S \subseteq \mc{Y} 
  \times \mc{M}\) be such that \(a \in \DSM(\mu,P)\) for every 
  \((\mu,a) \in S\) and let $T_1(S)$ be as in 
  Construction~\ref{Konstr:t_v,s-reduktion}. Then the following are 
  equivalent:
  \begin{enumerate}
    \item
      \(\{\mu-a\}_{(\mu,a) \in S}\) is a $P$-monic $V$-Gr\"obner 
      basis of $\mc{I}(S)$.
    \item
      \(\mc{M} = \Irr(S) \oplus \mc{I}(S)\).
    \item
      For all \((\mu_1,a_1), (\mu_2,a_2) \in S\) and \(v \in V\) such 
      that \(\mu_1 = v(\mu_2)\) there exists some \(t \in T(S)\) such 
      that \(t\bigl( a_1 -\nobreak v(a_2) \bigr) = 0\).
  \end{enumerate}
\end{corollary}
\begin{proof}
  It follows from \eqref{Eq1:Gerritzen} that all elements of $V$ are 
  strictly monotone with respect to $P$. Hence the conditions in 
  Theorem~\ref{S:Grobner-DL} are fulfilled and one only has to 
  verify that the last condition is \parenthetic{a\texttris} by 
  characterising the $V$-critical ambiguities.
  
  An arbitrary ambiguity of $T_1(S)$ has the form 
  $(t_{v_1,(\mu_1,a_1)},\mu,t_{v_2,(\mu_2,a_2)})$ where 
  \(v_1(\mu_1) = \mu = v_2(\mu_2)\). The situation in the last 
  condition is exactly this for \(v_1 = \id\) or \(v_2 = \id\), so it 
  only remains to show that all other ambiguities are 
  non-$V$-critical. Unique factorisation in $\mc{Y}$ gives rise to a 
  unique factorisation in $V$ (as compositions of elements of $V_1$), 
  and thus there exist \(v_1',v_2' \in V_1\) and \(v_1'',v_2'' \in 
  V\) such that \(v_1 = v_1' \circ v_1''\) and \(v_2 = v_2' \circ 
  v_2''\).
  
  If \(v_1' = v_2'\) then $\bigl( t_{v_1'',(\mu_1,a_1)}, v_1''(\mu_1), 
  t_{v_2'',(\mu_2,a_2)} \bigr)$ is another ambiguity, of which 
  $(t_{v_1,(\mu_1,a_1)},\mu,t_{v_2,(\mu_2,a_2)})$ is a proper 
  $V$-shadow. Since there for every \(\mu \in \mc{Y}\) is only finitely 
  many \((v,\nu) \in V \times \mc{Y}\) such 
  that \(\mu = v(\nu)\), it follows that $V$-shadow-critical is the 
  same as $V$-shadow-minimal, and hence none of the ambiguities with 
  \(v_1' = v_2'\) are $V$-critical.
  
  If instead \(v_1' \neq v_2'\) then one of these must multiply on 
  the left and the other must multiply on the right; it can without 
  loss of generality be assumed that \(v_1'(b) = \nu_1 b\) and 
  \(v_2'(b) = b\nu_2\). This implies that \(\mu = \nu_1\nu_2 = 
  v_2''(\mu_2) v_1''(\mu_1)\) however, and thus 
  $(t_{v_1,(\mu_1,a_1)},\mu,t_{v_2,(\mu_2,a_2)})$ is a montage of 
  $(\nu_2, t_{v_1'',(\mu_1,a_1)})$ and $(\nu_1, t_{v_2'',(\mu_2,a_2)})$ 
  with composition map \(w(b_1,b_2) = b_2b_1\). Hence the ambiguities 
  with \(v_1' \neq v_2'\) aren't $V$-critical either.
\end{proof}

\subsection{Path algebras}
\label{Ssec:Stigalgebra}

There is in the literature also a theorem by Farkas, Feustel, and 
Green~\cite{FarkasFeustelGreen} which similarly characterises 
(reduced) Gr\"obner bases in path algebras and certain semigroup 
algebras; the result is derived in an axiomatic setting generalising 
path algebras. Not surprisingly, it is in that setting 
equally possible to derive from the generic diamond lemma theory a 
result on more general (uniform monic) Gr\"obner bases in these 
algebras. The proof is essentially the same as for 
Theorem~\ref{S:Grobner-DL}, but the result is not Yet Another 
Corollary due to some technicalities caused by allowing the product 
of two monomials to be zero.

In the present notation, one is given a field $\mc{R}$ and an 
associative $\mc{R}$-algebra $\mc{M}$ with basis $\mc{Y}$. This basis 
is assumed to be well-ordered, so let $P$ be that order. Another 
binary relation \emph{divides}, or symbolically $\mid$, is defined 
on $\mc{Y}$ by \(\mu \mid \lambda\) iff there exist \(\nu_1,\nu_2 \in 
\mc{Y}\) such that \(\lambda = \nu_1\mu\nu_2\). These data are 
furthermore required to satisfy five axioms:
\begin{enumerate}
  \item[M1.]
    \(\mc{Y} \cup \{0\} \subset \mc{M}\) is a semigroup under 
    multiplication.
  \item[M2.]
    `Divides' is reflexive.
  \item[M3.]
    For each \(\lambda \in \mc{Y}\), the set $\setOf{ \mu \in \mc{Y} 
    }{ \text{$\mu$ divides $\lambda$} }$ is finite.
  \item[M4.]
    If \(\mu,\nu,\lambda,\rho \in \mc{Y}\) are such that none of the 
    products below are zero, then
    \begin{equation}
      \nu < \mu \pin{P} \Ipil 
      \lambda\nu\rho < \lambda\mu\rho \pin{P}
      \text{.}
    \end{equation}
  \item[M5.]
    If \(\mu \mid \lambda\) then \(\mu \leqslant \lambda \pin{P}\).
\end{enumerate}
In the case that $\mc{M}$ is the path algebra $\FAlg{\Gamma}{\mc{R}}$ 
and $\mc{Y}$ is the set of all paths\footnote{
  To be formally correct, one should really say \emph{walk} rather 
  than `path', since a \emph{path} (as all graph theorists know) may 
  not have any repeated vertices, but speaking of `walk algebras' 
  here would probably cause more confusion than it avoids. 
} in $\Gamma$ (counting vertices as paths of length $0$), axioms 
M1--M3 are trivial properties; in particular M1 is characteristic. 
M4 is a natural modification of the monoid partial order axiom 
\eqref{Eq:Semigruppsordning} and M5 is another condition on $P$; 
the authors suggest that one meets it by using a length-lexicographic 
order, although a weighted-degree lexicographic order will work 
just as well. It should be observed that $\mc{Y} \cup \{0\}$ is 
typically not a monoid, since the unit in a path algebra is the sum 
of all length $0$ paths rather than any particular path.

Simple reductions may be constructed as in 
Construction~\ref{Konstr:t_v,s-reduktion}, with $V$ being the set of 
maps \(b \mapsto \lambda b \rho\) for \(\lambda,\rho \in \mc{Y}\); 
this is exactly the same as in~\cite[p.~731]{FarkasFeustelGreen}. 
Similarly the definition there of a ($P$-monic) `Gr\"obner generating 
set' is exactly the same as `$P$-monic $V$-Gr\"obner basis' here. 
Axiom M4 is exactly what is needed in 
Lemma~\ref{L:OrdnFramflyttKompatibel} to establish that $V$ 
correlates $P$ to itself, and then the compatibility with $P$ of 
$T(S)$ follows from Lemma~\ref{L:Kompatibilitet} for any $S$ 
constructed as in Lemma~\ref{L:Grobnerbas}. It is however not quite 
as straightforward to apply Lemma~\ref{L:Framflyttbar} to prove that 
the elements of $V$ are advanceable. Besides the trivial detail 
that $V$ is not in general closed under composition\Ldash if 
\(v_1(b) = \lambda_1 b \rho_1\) and \(v_2(b) = \lambda_2 b \rho_2\) 
then \((v_1 \circ\nobreak v_2)(b) = \lambda_1 \lambda_2 b \rho_2 
\rho_1\) which is only an element of $V$ if \(\lambda_1\lambda_2 
\neq 0\) and \(\rho_2\rho_1 \neq 0\), although that can be worked 
around by considering $V \cup \{0\}$ instead\Dash there is in this 
lemma also the more significant condition that every \((\mu,a) \in S\) 
and \(v \in V\) must satisfy \(v(a) = 0\) if \(v(\mu) = 0\). This is 
why the result was above described as being about \emph{uniform} 
monic Gr\"obner bases. 

In \cite[p.\,733]{FarkasFeustelGreen}, two elements \(\mu,\nu \in 
\mc{Y}\) are defined to be \emDefOrd{uniform-equivalent} if
\begin{equation}
  \lambda\mu\rho = 0 \Epil \lambda\nu\rho = 0
  \qquad\text{for all \(\lambda,\rho \in \mc{Y}\).}
\end{equation}
In a path algebra, this simply means that $\mu$ and $\nu$ have the 
same endpoints, but in principle the matter might be more 
complicated. Nonetheless, uniform-equivalence is an equivalence 
relation on $\mc{Y}$ and defines a partition of $\mc{Y}$ into 
equivalence classes. An element $a$ of $\mc{M}$ is said to be 
\emDefOrd{uniform} if all elements of $\supp(a)$ are 
uniform-equivalent, and consequently a pair $(\mu,a) \in \mc{Y} 
\times \mc{M}\) can be said to be uniform if every element of 
$\supp(a)$ is uniform-equivalent to $\mu$. Considering only uniform 
Gr\"obner bases may seem like a severe restriction, but at least in 
the case of a path algebra it is actually rather trivial. The reason 
for this is that there is in a path algebra no way in which a path
can be ``uniform-superior'' to another path; they're either 
equivalent or quite different. More concretely, if \(\mu,\nu \in 
\mc{Y}\) are \emph{not} uniform-equivalent then for each \(v \in V\), 
at most one of $v(\mu)$ and $v(\nu)$ can be nonzero. This has the 
effect that only the uniform parts of rules get encoded into 
$T_1(S)$; for $t_{v,(\mu,a)}$ to even exist $v(\mu)$ must be nonzero 
and thus all \(\nu \in \supp(a)\) which are not uniform-equivalent to 
$\mu$ will be killed by $v$.

In a path algebra, it is easy to see that any ideal is generated by a 
set of uniform elements; writing $\Gamma_0$ for the set of vertices 
in $\Gamma$, any \(a \in \FAlg{\Gamma}{\mc{R}}\) can be expressed as 
the sum of uniform elements 
\(\sum_{\kappa,\rho \in \Gamma_0} \kappa a \rho\), and these terms are 
elements of every ideal containing $a$. That the same 
should hold in general is not obvious, but any algebra satisfying 
M1--M5 must contain idempotent elements which fill the role of 
vertices in this argument; in particular axiom~M2 is not as innocent 
as it may seem, since what it claims is really that there for every 
\(\mu \in \mc{Y}\) exist \(\kappa,\rho \in \mc{Y}\) such that 
\(\kappa\mu\rho = \mu\). The structure of algebras satisfying M1--M5 
is the subject of~\cite[Sec.~4]{FarkasFeustelGreen}, and the 
conclusion is roughly that any such algebra has to be a path algebra 
in which some paths have been identified.

Anyhow, with $\mc{M}$, $\mc{R}$, $\mc{Y}$, $V$, $P$, $S$, and $T_1(S)$ 
as above, it follows that \parenthetic{a}, \parenthetic{a\textprime}, 
\parenthetic{a\textbis}, \parenthetic{a\texttris}, \parenthetic{b}, 
\parenthetic{c}, and \parenthetic{d} of Theorem~\ref{S:Grobner-DL} 
are equivalent. (When employing Lemma~\ref{L:Grobnerbas} one must 
extend $V$ with the identity map to make it a monoid, but since $S$ 
is uniform that doesn't contribute any additional reductions.) The 
structure of $V$-critical ambiguities can be analysed as in 
Example~\ref{Ex:BergmanskTvetydighet}; \cite{FarkasFeustelGreen} 
gives the characterisation of overlaps between $(\mu_1,a_1)$ and 
$(\mu_2,a_2)$ as being determined by \(\nu_1,\nu_2,\lambda \in 
\mc{Y}\) such that \(\mu_1 = \nu_1 \lambda\), \(\mu_2 = \lambda 
\nu_2\), \(\nu_2 \neq \mu_2\), and \(\nu_2 \neq \mu_2\).

\subsection*{Acknowledgments}

Part of the research reported herein was carried out in 2003--2004, 
when the author was a postdoc at the Mittag-Leffler institute, 
participating in the NOG Noncommutative Geometry programme.

%

\begin{theindex}
\addcontentsline{toc}{section}{Index}

  \item \dots$(i)$, \hyperpage{42}
  \item $\pm R^*$, \hyperpage{6}
  \item $\equiv\pmod{S}$, \hyperpage{16}
  \item $\equiv\pmod{S < \mu \pin{P}}$, \hyperpage{37}

  \indexspace

  \item $\ssI$, \hyperpage{4}

  \indexspace

  \item act trivially, \hyperpage{16}
  \item advanceable, \hyperpage{19}, \hyperpage{43}
    \subitem absolutely, \hyperpage{19}, \hyperpage{43}
    \subitem bi-, \hyperpage{48}
    \subitem conditionally, \hyperpage{19}
  \item algebra ultranorm, \hyperpage{12}
  \item ambiguity, \hyperpage{38}
    \subitem absolute shadow, \hyperpage{43}
    \subitem critical, \hyperpage{50}
    \subitem inclusion, \hyperpage{53}
    \subitem montage, \hyperpage{48}
    \subitem overlap, \hyperpage{53}
    \subitem proper inclusion, \hyperpage{53}
    \subitem proper shadow, \hyperpage{50}
    \subitem resolvable, \hyperpage{38}
    \subitem resolvable relative to, \hyperpage{38}
    \subitem shadow, \hyperpage{43}, \hyperpage{50}
    \subitem shadow-critical, \hyperpage{50}
    \subitem shadow-minimal, \hyperpage{50}
  \item antitone, \hyperpage{46}

  \indexspace

  \item $B_n$, \hyperpage{5}
  \item $B_n(i)$, \hyperpage{42}
  \item biadvanceable, \hyperpage{48}
  \item bihomomorphism, \hyperpage{48}

  \indexspace

  \item category, \hyperpage{50}
    \subitem generated by, \hyperpage{50}
  \item compatible
    \subitem partial order, \hyperpage{46}
    \subitem reduction, \hyperpage{33}
  \item composition lemma, \hyperpage{39}
  \item composition map, \hyperpage{48}
  \item confluent, \hyperpage{39}
  \item correlate, \hyperpage{45}
  \item critical pair, \hyperpage{39}
  \item $\Cspan$, \hyperpage{8}

  \indexspace

  \item down-set, \hyperpage{33}
    \subitem module, \hyperpage{33}
  \item $\DSM(\mu,P)$, \hyperpage{33}

  \indexspace

  \item equicontinuous, \hyperpage{29}

  \indexspace

  \item fork, \hyperpage{39}

  \indexspace

  \item Gr\"obner basis, \hyperpage{22}, \hyperpage{61}

  \indexspace

  \item $I$ (set of sorts), \hyperpage{42}
  \item $\mc{I}(S)$, \hyperpage{16}
  \item $V$-ideal, \hyperpage{21}
  \item $V$-ideal basis, \hyperpage{21}
  \item inner reduction, \hyperpage{53}
  \item $\Irr$, \hyperpage{16}
  \item irreducible, \hyperpage{16}

  \indexspace

  \item leading monomial, \hyperpage{61}
  \item $\LM(g)$, \hyperpage{61}
  \item $\lm(a)$, \hyperpage{61}
  \item locally confluent, \hyperpage{39}

  \indexspace

  \item $\mc{M}$, \hyperpage{5}
  \item $\mc{M}(i)$, \hyperpage{42}
  \item $\cmplM$, \hyperpage{8}
  \item $\cmplM(i)$, \hyperpage{42}
  \item $R$-module, \hyperpage{6}, \hyperpage{8}
  \item module ultranorm, \hyperpage{12}
  \item monic, \hyperpage{33}
  \item monomial, \hyperpage{4}
  \item monotone, \hyperpage{45}
  \item montage, \hyperpage{48}

  \indexspace

  \item $\varepsilon $-neighbourhood, \hyperpage{7}
  \item normal, \hyperpage{17}
  \item normal form, \hyperpage{16}

  \indexspace

  \item $\mc{O}$, \hyperpage{5}
  \item $\cmplO$, \hyperpage{8}
  \item $\mc{O}(i)$, \hyperpage{42}
  \item open, \hyperpage{7}
  \item outer reduction, \hyperpage{53}

  \indexspace

  \item $\Per$, \hyperpage{24}
  \item persistently $\varepsilon $-reducible, \hyperpage{24}
  \item persistently reducible, \hyperpage{24}
  \item piece, \hyperpage{48}

  \indexspace

  \item $R$, \hyperpage{5}
  \item $R(i)$, \hyperpage{42}
  \item $R^*$, \hyperpage{6}
  \item $\RstarY$, \hyperpage{6}
  \item $\Red(S)$, \hyperpage{27}
  \item $\Red_\ve(S)$, \hyperpage{26}
  \item reduction, \hyperpage{14}
  \item rewrite rule, \hyperpage{18}
  \item rewriting system, \hyperpage{14}
  \item ring ultranorm, \hyperpage{12}

  \indexspace

  \item $S$, \hyperpage{14}
  \item simple reduction, \hyperpage{14}
  \item $\Span$, \hyperpage{8}
  \item strictly monotone, \hyperpage{45}
  \item strong triangle inequality, \hyperpage{12}
  \item stuck in, \hyperpage{24}
  \item $\supp(a)$, \hyperpage{61}
  \item support, \hyperpage{61}

  \indexspace

  \item $T_1(S)$, \hyperpage{5}
  \item $T_1(S)(i)$, \hyperpage{42}
  \item $t_{\mu\mapsto a}$, \hyperpage{15}
  \item $t_{\nu_1 s \nu_2}$, \hyperpage{18}
  \item $T(S)$, \hyperpage{14}
  \item $t_{v,s}$, \hyperpage{18}, \hyperpage{57}
  \item TDCC, \hyperpage{32}
  \item term, \hyperpage{4}
  \item terminal, \hyperpage{17}
  \item topological descending chain condition, \hyperpage{32}
  \item trivial ultranorm, \hyperpage{12}

  \indexspace

  \item ultranorm, \hyperpage{12}
  \item uniform, \hyperpage{70}
  \item uniform-equivalent, \hyperpage{69}
  \item uniquely reducible, \hyperpage{27}
  \item $\varepsilon $-uniquely reducible, \hyperpage{26}

  \indexspace

  \item weight function, \hyperpage{12}
  \item well-founded, \hyperpage{32}

  \indexspace

  \item $X^\bullet$, \hyperpage{67}
  \item $X^*$, \hyperpage{4}

  \indexspace

  \item $\mc{Y}$, \hyperpage{5}
  \item $\mc{Y}(i)$, \hyperpage{42}

\end{theindex}

\end{document}